\documentclass{article}

\pdfoutput=1

\usepackage{amsmath}
\usepackage{amsthm}
\usepackage{thmtools}
\usepackage{amssymb,bm,verbatim,bbm,mathrsfs}
\usepackage{amssymb,bbm}
\usepackage{tikz}
\usepackage{enumerate}
\usepackage[font=small,labelfont=bf]{caption}
\usepackage{subcaption}
\usepackage{hyperref}
\usepackage{soul}
\usepackage{microtype}
\usepackage{cleveref}

\hypersetup{
    colorlinks=true,
    linkcolor=black,
    citecolor=black,
    filecolor=black,
    urlcolor=black,
}

\usetikzlibrary{calc,decorations.markings}
\usetikzlibrary{decorations.pathmorphing}
\usetikzlibrary{matrix}

\numberwithin{equation}{section}
\allowdisplaybreaks[1]

\declaretheorem[numberwithin=section]{theorem}
\declaretheorem[sibling=theorem]{lemma}
\declaretheorem[sibling=theorem]{corollary}
\declaretheorem[sibling=theorem]{proposition}
\declaretheorem[sibling=theorem]{remark}
\declaretheorem[sibling=theorem]{definition}

\declaretheorem[]{assumption}

\newcommand{\bb}[1]{\mathbb{#1}}
\newcommand{\cc}[1]{\mathcal{#1}}

\newcommand{\R}{\bb R}
\newcommand{\Z}{\bb Z}

\newcommand{\C}{\bb C}
\newcommand{\N}{\bb N}
\renewcommand{\P}{\bb P}

\newcommand{\co}[1]{\left[#1\right )} 
\newcommand{\oc}[1]{\left(#1\right ]} 
\newcommand{\ob}[1]{\left(#1\right )} 
\newcommand{\cb}[1]{\left[#1\right ]} 
\newcommand{\set}[1]{\left\{#1\right\}} 
\newcommand{\ab}[1]{\langle #1 \rangle} 
\newcommand{\abs}[1]{\left\vert#1\right\vert} 
\newcommand{\norm}[1]{\|#1\|} 

\newcommand{\indicator}{\mathbbm{1}}
\newcommand{\indicatorthat}[1]{{\mathbbm{1}}_{\left\{#1\right\}}}

\newcommand{\LE}{\mathrm{LE}} 
\newcommand{\Walks}{\Omega} 
\newcommand{\SAW}{\Walks_\mathrm{SAW}} 
\newcommand{\SAP}{\Walks_\mathrm{SAP}} 
\newcommand{\WH}[0]{\cc V} 
\newcommand{\loopadd}[0]{\oplus} 
\newcommand{\OC}[0]{\vec{\cc C}} 

\newcommand{\range}[1]{\mathrm{range}(#1)}
\newcommand{\edgespan}[0]{\mathrm{span}}
\newcommand{\IH}{\cc G} 
\newcommand{\graphs}{G} 
\newcommand{\cgraphs}{G^{c}} 
\newcommand{\splice}[0]{\colon\!}
\newcommand{\truebubblechain}[0]{\mathrm{BC}}
\newcommand{\bubblechain}[0]{\mathrm{B^{\star}}}
\newcommand{\FS}[0]{\cb{-\pi,\pi}} 
\newcommand{\FSint}[0]{(2\pi)^{d}} 
\newcommand{\hyp}[2]{F_{#1,#2}} 

\makeatletter
\newcommand{\subjclass}[2][2010]{%
  \let\@oldtitle\@title%
  \gdef\@title{\@oldtitle\footnotetext{#1 \emph{Mathematics subject classification:} #2}}%
}
\newcommand{\keywords}[1]{%
  \let\@@oldtitle\@title%
  \gdef\@title{\@@oldtitle\footnotetext{\emph{Key words and phrases:} #1.}}%
}
\makeatother

\begin{document}

\title{Loop-Weighted Walk}
\author{Tyler Helmuth\footnote{Institute for Computational and
    Experimental Research in Mathematics, 121 South Main St.\,
    Providence, RI, 02903, USA.} \footnote{Current address: Department of
    Mathematics, 899 Evans Hall, Berkeley, CA, 94720-3840 USA. Email:
    jhelmt@math.berkeley.edu}}
\date{}


\maketitle

\begin{abstract}
  Loop-weighted walk with parameter $\lambda\geq 0$ is a non-Markovian
  model of random walks that is related to the loop $O(N)$ model of
  statistical mechanics. A walk receives weight $\lambda^{k}$ if it
  contains $k$ loops; whether this is a reward or punishment for
  containing loops depends on the value of $\lambda$. A challenging
  feature of loop-weighted walk is that it is not purely repulsive,
  meaning the weight of the future of a walk may either increase or
  decrease if the past is forgotten. Repulsion is typically an
  essential property for lace expansion arguments. This article
  circumvents the lack of repulsion and proves, for any $\lambda>0$,
  that loop-weighted walk is diffusive in high dimensions
  by lace expansion methods.
\end{abstract}

\section{Introduction and Main Results}
\label{sec:LWW-Intro-Results}

\emph{Loop-weighted walk with parameter $\lambda$}, abbreviated
\emph{$\lambda$-LWW}, is a model of self-interacting walks that can be
informally defined as follows. Formal definitions will be given in
\Cref{sec:LWW-Basics}. Let $\omega$ be a walk on a graph. A walk is
called a \emph{loop} if $\omega$ begins and ends at the same
vertex. The \emph{loop erasure} $\LE(\omega)$ is formed by
chronologically removing loops from $\omega$. If $n_{L}(\omega)$
denotes the number of loops removed, the $\lambda$-LWW weight of a
walk $\omega$ is
\begin{equation}
  \label{eq:LWW-Weight-Intro}
  w_{\lambda}(\omega) = \lambda^{n_{L}(\omega)}.
\end{equation}

Throughout this article it will be assumed that $\lambda\geq 0$, so
\Cref{eq:LWW-Weight-Intro} defines a non-negative weight on walks. In
particular, $w_{\lambda}$ defines a probability measure on $n$-step
walks that begin at a fixed vertex of a graph by defining the
probability of $\omega$ to be proportional to
$w_{\lambda}(\omega)$. If $0\leq \lambda<1$ the effect of the weight
is to discourage walks from containing loops, and for this parameter
range $\lambda$-LWW interpolates between the uniform measure on
$n$-step self-avoiding walks ($0$-LWW) and the uniform measure on all
$n$-step walks ($1$-LWW). If $\lambda>1$ the weight encourages the
existence of loops: walks are rewarded for returning to vertices that
have been visited in the past. Note that $\lambda$-LWW for
$\lambda\neq 1$ is not a Markovian model of walks.

In addition to being an interesting model of self-interacting random
walks that encompasses the well-known models of self-avoiding and
simple random walk, $\lambda$-LWW also has connections with spin
models in statistical mechanics. The description of these connections
will be deferred until after the results of the article are described,
see \Cref{sec:LWW-Intro-SM}.

This article consists of a lace expansion analysis of $\lambda$-LWW. The
lace expansion, originally introduced by Brydges and
Spencer~\cite{BrydgesSpencer1985}, is a powerful tool for proving
mean-field behaviour in high dimensions~\cite{Slade2006}. With
few exceptions, see the discussion at the end of
\Cref{sec:LWW-Proof-Outline}, walk models that have been successfully
studied with the lace expansion have been \emph{purely repulsive}. A
walk model being purely repulsive means that the weight $w$ on walks
that defines the model satisfies the inequality
\begin{equation}
  \label{eq:LWW-Repulsive-Intro}
  w(\omega\circ\eta) \leq w(\omega) w(\eta),
\end{equation}
where $\omega\circ\eta$ is the concatenation of two walks $\omega$ and
$\eta$. For example, self-avoiding walk is purely repulsive. In
general $\lambda$-LWW is \emph{not} purely repulsive if
$\lambda\neq 0,1$. See~\Cref{fig:LWW-Not-Repulsive}.

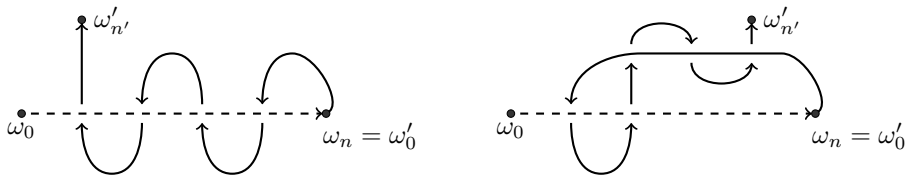
\begin{figure}[h]
  \centering
  \beginpgfgraphicnamed{fig1a}
  \begin{tikzpicture}[scale=.80]
    \node (s0) at (1,0) {};
    \node (s2) at (2,0) {};
    \node (s3) at (3,0) {};
    \node (s4) at (4,0) {};
    \node (s5) at (5,0) {};
    \node (s6) at (6,0) {};
    \draw[black,thick,dashed,->] (s0) to (6,0);
    \draw[black, thick,->] (6,0) to[out=0, in=0] (5.5,1) to[out=180,in=90]
    (s5);
    \draw[black,thick,->] (s5) to[out=270, in= 0] (4.5,-1) to[out=180,
    in=270] (s4);
    \draw[black, thick,->] (s4) to[out=90, in=0] (3.5,1) to[out=180,in=90]
    (s3);
    \draw[black,thick,->] (s3) to[out=270, in= 0] (2.5,-1) to[out=180,
    in=270] (s2);
    \draw[black,thick,->] (s2) to (2,1.5);

    \node[circle,draw,fill=black!80,inner sep=1pt] at (s0) {};
    \node[circle,draw,fill=black!80,inner sep=1pt] at (6.06,0) {};
    \node[circle,draw,fill=black!80,inner sep=1pt] at (2.0,1.56) {};
    \node at (s0) [below] {$\omega_{0}$};
    \node at (s6) [below right] {$\!\!\omega_{n}=\omega^{\prime}_{0}$};
    \node at (2.05,1.55) [right] {$\omega^{\prime}_{n^{\prime}}$};
  \end{tikzpicture}
  \endpgfgraphicnamed
  \qquad
  \beginpgfgraphicnamed{fig1b}
  \begin{tikzpicture}[scale=.80]
    \node (s0) at (1,0) {};
    \node (s2) at (2,0) {};
    \node (s3) at (3,0) {};
    \node (h3) at (3,1) {};
    \node (h4) at (4,1) {};
    \node (h5) at (5,1) {};
    \node (s6) at (6,0) {};
    \draw[black,thick,dashed,->] (s0) to (6,0);
    \draw[black,thick,->] (6,0) to[out=0, in=0] (5.5,1) to[out=180,in=0]
    (h3) to[out=180,in=90] (s2);
    \draw[black,thick,->] (s2) to[out=270, in= 180] (2.5,-1) to[out=0,
    in=270] (s3);
    \draw[black,thick,->] (s3) to[out=90, in=270] (h3);
    \draw[black,thick,->] (h3) to[out=90, in=180] (3.5,1.5) to[out=0,
    in=90] (h4);
    \draw[black,thick,->] (h4) to[out=270, in=180] (4.5,.5) to[out=0,
    in=270] (h5);
    \draw[black,thick,->] (h5) to (5,1.5);

    \node[circle,draw,fill=black!80,inner sep=1pt] at (s0) {};
    \node[circle,draw,fill=black!80,inner sep=1pt] at (6.06,0) {};
    \node[circle,draw,fill=black!80,inner sep=1pt] at (5.0,1.56) {};
    \node at (s0) [below] {$\omega_{0}$};
    \node at (s6) [below right] {$\!\!\omega_{n}=\omega^{\prime}_{0}$};
    \node at (5.05,1.55) [right] {$\omega^{\prime}_{n^{\prime}}$};
  \end{tikzpicture}
  \endpgfgraphicnamed
  \caption{For each diagram consider (i) the concatenation of the
    dashed walk $\omega$ and the solid walk $\omega^{\prime}$ and (ii)
    the two walks as being separate. On the left (i) results in four
    loops being erased, $n_{L}(\omega\circ\omega^{\prime})=4$, while
    (ii) results in no loops being erased,
    $n_{L}(\omega)=n_{L}(\omega^{\prime})=0$. On the right (i) results
    in one loop being erased, $n_{L}(\omega\circ\omega^{\prime})=1$,
    while (ii) results in three loops being erased, $n_{L}(\omega)=0$,
    $n_{L}(\omega^{\prime})=3$. It follows that the $\lambda$-LWW
    weight is not purely repulsive for $\lambda\neq 0,1$.}
  \label{fig:LWW-Not-Repulsive}
\end{figure}

The most significant step required to analyze $\lambda$-LWW with the
lace expansion is therefore a technique to overcome the lack of
repulsion. This is done by resumming $\lambda$-LWW to obtain a model
of self-interacting and self-avoiding walks. The particular form of
the $\lambda$-LWW weight leads to a very explicit description of the
self-interaction in terms of a generalization of the loop measure
of~\cite{LawlerLimic2010}, and this explicit description makes it
clear that the self-interaction is repulsive. This enables a lace
expansion to be performed. Further details about the proof follow
after the statement of \Cref{thm:LWW-Main-Intro}.

Some notation will be needed to state the results. Let
$\ab{\cdot}_{n}^{\lambda}$ denote expectation with respect to the
measure on $n$-step walks associated to $w_{\lambda}$. Let
$c_{n}^{\lambda}$ be the normalizing factor for the expectation, i.e.,
the sum over all $n$ step walks weighted by $\lambda^{n_{L}(\omega)}$
as in~\eqref{eq:LWW-Weight-Intro}. Let $\chi_{\lambda}(z) =
\sum_{n}c_{n}^{\lambda}z^{n}$, and let $z_{c}(\lambda)$ be the radius
of convergence of $\chi_{\lambda}(z)$. The main result of this article
can be summarized as saying that, in high dimensions, $\lambda$-LWW
has mean field behaviour at criticality.

\begin{theorem}
  \label{thm:LWW-Main-Intro}
  Fix $\lambda\geq 0$ and consider $\lambda$-LWW on
  $\Z^{d}$. There exists $d_{0}=d_{0}(\lambda)$ such that for $d\geq
  d_{0}$ there are constants $A$ and $D$ such
  that 
  \begin{enumerate}
  \item The susceptibility diverges linearly: $\chi_{\lambda}(z) \sim
    Az_{c} (z_{c}-z)^{-1}$ as $z\nearrow z_{c}$,
  \item $c_{n}^{\lambda} = A\ob{z_{c}(\lambda)}^{-n}(1+O(n^{-\delta}))$ for any
    $\delta<1$, and
  \item $\lambda$-LWW is diffusive:
    $\ab{\abs{\omega_{n}}^{2}}^{\lambda}_{n} = Dn(1+O(n^{-\delta}))$ for any
    $\delta<1$.
  \end{enumerate}
\end{theorem}
For $\lambda=0$ \Cref{thm:LWW-Main-Intro} has been proven with
$d_{0}=5$ by Hara and Slade~\cite{HaraSlade1992}. It is worth
emphasizing that \Cref{thm:LWW-Main-Intro} holds for $\lambda>1$ when
$\lambda$-LWW is attractive in the sense that the formation of loops
is encouraged.

\begin{remark}
  \label{rem:LWW-Dimension}
  No attempt has been made to track the value of $d_{0}$ that is
  required, and the proof presented in this article requires $d_{0}\gg
  9$. The true behaviour of $d_{0}(\lambda)$ is an interesting
  question for future study.
\end{remark}

Let us say a few more words about the proof of
\Cref{thm:LWW-Main-Intro}. As described earlier, the key step is a
resummation of $\lambda$-LWW into a self-interacting self-avoiding
walk. The self-interaction of the self-avoiding walk is a many-body
interaction, and this leads to a hypergraph-based lace expansion
instead of the graph-based lace expansion that is used for
self-avoiding walk. We stress that hypergraphs are merely
an organizational tool, and no prior knowledge of hypergraphs is
needed to understand the expansion. Once the lace expansion has been
performed the various self-interacting self-avoiding walk quantities
can be re-expressed in terms of $\lambda$-LWW. The diagrams that occur
in analyzing the expansion generalize the diagrams for self-avoiding
walk, and when $\lambda=0$ they reduce to the diagrams for
self-avoiding walk. With some effort it is possible to analyze the
diagrams for $\lambda>0$ with existing methods. Once the analysis of
the diagrams is completed it is possible to apply established
techniques to analyze $\lambda$-LWW, namely the trigonometric approach
to the convergence of the lace expansion~\cite{Slade2006} and complex
analytic methods for studying asymptotics.

In fact \Cref{thm:LWW-Main-Intro} holds in greater generality. Let
$\lambda_{\ell}\geq 0$ be the weight of the loop $\ell$. Replace the
weight $\lambda$ per loop in~\Cref{eq:LWW-Weight-Intro} with the
product of $\lambda_{\ell}$ over the set of loops $\ell$ that are
erased when performing loop erasure on $\omega$. Assume the set of
weights $\{\lambda_{\ell}\}$ satisfy a mild symmetry hypothesis,
see~\Cref{sec:LWW-Definition}, and are uniformly bounded above. Then the
results of \Cref{thm:LWW-Main-Intro} continue to hold.

The remainder of the introduction is as
follows. \Cref{sec:LWW-Intro-SM} describes an important connection
between $\lambda$-LWW and the loop $O(N)$ model of statistical
physics. \Cref{sec:LWW-Basics} gives a formal definition of
$\lambda$-LWW, relates $\lambda$-LWW to a self-interacting and
self-avoiding walk, and outlines how this enables a lace expansion
analysis. Lastly, \Cref{sec:LWW-Conventions} establishes a few
conventions used in the remainder of the article.

\subsection{Motivation from Statistical Mechanics}
\label{sec:LWW-Intro-SM}

For $N\in \N$ the \emph{$O(N)$ model} on a graph finite $G = (V,E)$ is
a generalization of the Ising model. To each vertex $x\in V$ is
associated a \emph{spin} $\vec s_{x}$ taking values in the unit sphere
in $\R^{N}$. The probability of a spin configuration is defined by
\begin{equation}
  \label{eq:LWW-Intro-SM-1}
  \P\ob{ \{\vec s_{x}\}_{x\in V}} \propto \exp ( \beta \sum_{x\sim y}
  \vec s_{x}\cdot \vec s_{y}),
\end{equation}
where $\beta$ is a real parameter and the summation is over all edges
$\{x,y\}\in E$.  In~\cite{DomanyMukamelNienhuisSchwimmer1981} a
simplification of the $O(N)$ model known as the \emph{loop $O(N)$
  model} was introduced. The loop $O(N)$ model is defined in terms of
subgraph configurations on $G$. In the special case of a graph with vertex
degree bounded by $3$, the loop $O(N)$ model configurations are subgraphs
that are disjoint unions of cycles of length at least $3$, and the
probability of a subgraph $H$ is given by
\begin{equation}
  \label{eq:LWW-Intro-SM-2}
  \P(H) \propto z^{\abs{E(H)}} N^{\# H},
\end{equation}
where $\# H$ denotes the number of connected components of $H$. Note
that the probability in \Cref{eq:LWW-Intro-SM-2} may be negative if
$N<0$: \Cref{eq:LWW-Intro-SM-2} defines a signed measure in
general.

The definition of the loop $O(N)$ model on an arbitrary graph $G$
involves noncyclic subgraphs, see for
example~\cite{ChayesPryadkoShtengel2000}. The noncyclic subgraphs are
predicted by non-rigorous renormalization group arguments to be
irrelevant~\cite{Nienhuis2010}, at least for $\abs{N}\leq 2$ on
planar graphs.  Call the model whose configurations are disjoint
unions of cyclic subgraphs the $O(N)$ cycle gas. For $N\in\N$ this
model has previously appeared in the physics literature as a model for
melting transitions~\cite{HelfrichRys1982}.

As described in~\Cref{sec:LWW-Cycle-Correlation},
$\lambda$-LWW is a walk representation of the $O(N)$ cycle gas.  The
two-point function of $\lambda$-LWW corresponds to a two-point
correlation in the $O(N)$ cycle gas for $N=-2\lambda$. In other words,
$\lambda$-LWW yields a \emph{probabilistic} interpretation of the
$O(N)$ cycle gas for $N<0$. This is an example of a ``negative
activity isomorphism theorem'': an equivalence between a statistical
mechanics model at negative activity ($N<0$) and a probability
model. An important previous example of such a theorem is the
Brydges--Imbrie isomorphism between branched polymers in $\R^{d+2}$
and the hard-core gas in $\R^{d}$~\cite{BrydgesImbrie2003}. In the
present work the isomorphism allows results about $\lambda$-LWW to be
transferred to the $O(N)$ cycle gas for $N<0$. For example, the
isomorphism theorem combined with \Cref{thm:LWW-Main-Intro}
immediately implies the following corollary.

\begin{corollary}
  \label{cor:LWW-CYCLE}
  For $d$ sufficiently large the susceptibility of the $O(N)$ cycle
  gas on $\Z^{d}$ for $N<0$ diverges linearly at the critical point.
\end{corollary}

This section may be summarized as saying that $\lambda$-LWW can be
viewed as a random walk representation of an approximation of the
$O(N)$ model. Thus $\lambda$-LWW fits into a long history of random
walk representations of spin
models~\cite{Aizenman1982,BrydgesFrohlichSpencer1982,
  FernandezFrohlichSokal1992} inspired by the pioneering work of
Symanzik~\cite{Symanzik1969}.

\subsection{Introduction to the Loop-Weighted Walk Model}
\label{sec:LWW-Basics}

The rest of the paper will be concerned with $\Z^{d}$, the simple
cubic lattice in $d$ dimensions. Edges $\{x,y\}$ will often be
abbreviated $xy$. Two vertices $x$ and $y$ will be called
\emph{adjacent}, written $x\sim y$, if $xy$ is an edge in
$\Z^{d}$. Let $\Omega = \{ y\in \Z^{d} \mid y\sim 0\}$, so
$\abs{\Omega} = 2d$ is the number of vertices adjacent to the origin
$0$.

\subsubsection{Model Definition}
\label{sec:LWW-Definition}

The next paragraphs establish some conventions about walks. An
\emph{$n$-step walk} $\omega = (\omega_{0}, \omega_{1}, \dots,
\omega_{n})$ is a sequence of $n+1$ adjacent vertices in
$\Z^{d}$. Given a walk $\omega$, $\abs{\omega}$ will denote the number
of steps in $\omega$. A walk is a \emph{loop} if
$\omega_{\abs{\omega}} = \omega_{0}$, \emph{self-avoiding} if
$\omega_{i}=\omega_{j}$ implies $i=j$, and a \emph{self-avoiding
  polygon} if $\omega_{i}=\omega_{j}$ and $i\neq j$ implies $\{i,j\} =
\{0,\abs{\omega}\}$.

A walk $\omega$ \emph{begins} at $\omega_{0}$ and \emph{ends} at
$\omega_{\abs{\omega}}$. Let $\omega\colon x\to y$ denote the set of
walks beginning at $x$ and ending at $y$. Let $\SAW(x,y) = \{
\omega\colon x\to y \mid \textrm{$\omega$ self-avoiding}\}$; if $x=y$
this is taken to be the set of self-avoiding polygons beginning at
$x$. Let $\SAP = \cup_{x}\SAW(x,x)$ and $\SAW =
\cup_{x}\cup_{y}\SAW(x,y)$. If $\omega^{(i)} = (\omega^{(i)}_{0},
\dots, \omega^{(i)}_{k_{i}})$ for $i=1,2$ and $\omega^{(1)}_{k_{1}} =
\omega^{(2)}_{0}$ the \emph{concatenation} $\omega^{(1)}\circ
\omega^{(2)}$ of $\omega^{(1)}$ with $\omega^{(2)}$ is defined by
\begin{equation}
  \label{eq:LWW-Concatenation}
  \omega^{(1)}\circ\omega^{(2)} = (\omega^{(1)}_{0}, \dots,
  \omega^{(1)}_{k_{1}}, \omega^{(2)}_{1}, \dots,
  \omega^{(2)}_{k_{2}}).
\end{equation}

To define $\lambda$-LWW precisely requires an explicit description
of the loop erasure of a walk $\omega$. Define
\begin{align}
  \tau_{\omega} &= \min \set{ i \mid \exists\, j<i \textrm{ such that }
    \omega_{i}=\omega_{j}}, \\
  \tau_{\omega}^{\star} &= \min \set{j \mid \omega_{j} =
    \omega_{\tau_{\omega}}}.
\end{align}
If $\omega$ is a self-avoiding walk, define $\tau_{\omega} =
\tau_{\omega}^{\star}=\infty$. The time $\tau_{\omega}$ is the first
time a walk visits a vertex twice.

\begin{definition}
  \label{def:LWW-Erasure}
  Let $\omega$ be a walk of length $n$. The \emph{single loop
    erasure $\LE^{1}(\omega)$} of $\omega$ is given by
  \begin{equation}
    \label{eq:LWW-Single-Erase}
    \LE^{1}(\omega) = (\omega_{0}, \dots,
    \omega_{\tau_{\omega}^{\star}\wedge n}, \omega_{\tau_{\omega}+1},
    \dots, \omega_{n}),
  \end{equation}
  where $a\wedge b$ denotes the minimum of $a$ and $b$.  The walk
  $(\omega_{\tau_{\omega}^{\star}}, \omega_{\tau_{\omega}^{\star}+1},
  \dots, \omega_{\tau_{\omega}})$ is the \emph{loop removed by loop
    erasure}. The \emph{loop erasure} $\LE(\omega)$ of $\omega$ is the
  result of iteratively applying $\LE^{1}$ until
  $\tau_{\omega}=\infty$.
\end{definition}
By construction, each loop removed from a walk by loop erasure is a
self-avoiding polygon.
\begin{definition}
  \label{def:LWW-Loop-Vector}
  The \emph{loop vector $n_{L}(\omega)$} of
  $\omega$ is the vector with coordinates
  \begin{equation}
    \label{eq:LWW-Loop-Vector}
    n_{L}^{\eta}(\omega) = \textrm{\# of times $\eta$
      is removed by loop erasure applied to $\omega$}, \qquad \eta\in\SAP.
  \end{equation}
\end{definition}
In what follows $\lambda$ will denote a vector of activities
$\lambda_{\eta}\geq 0$ for $\eta\in\SAP$. Inequalities with respect to
$\lambda$ are to be interpreted pointwise in $\eta\in\SAP$. Define
\begin{equation}
  \label{eq:LWW-Variable-Loop-Weight}
  \lambda^{n_{L}(\omega)} = \prod_{\eta}\lambda_{\eta}^{n_{L}^{\eta}(\omega)}.
\end{equation}

\begin{definition}
  \label{def:LWW-LRW}
  Let $\lambda\geq 0$, $z\geq 0$. The weight $w_{\lambda,z}$ of
  $\lambda$-LWW at activity $z$ is given by
  \begin{equation}
    \label{eq:LWW-LRW}
    w_{\lambda,z}(\omega) = z^{\abs{\omega}}\lambda^{n_{L}(\omega)}.
  \end{equation}
\end{definition}

\begin{definition}
  \label{def:LWW-Susceptibility}
  The \emph{susceptibility $\chi_{\lambda}(z)$} of $\lambda$-LWW is
  \begin{equation}
    \label{eq:LWW-Susceptibility}
    \chi_{\lambda}(z) = \sum_{x\in \Z^{d}}\sum_{\omega\colon 0 \to x} w_{\lambda,z}(\omega).
  \end{equation}
  The \emph{critical point $z_{c}(\lambda)$} of $\lambda$-LWW is
  defined to be the radius of convergence of $\chi_{\lambda}(z)$.
\end{definition}

If $0\leq\lambda\leq 1$ then $\chi_{\lambda}(z) \leq \chi_{1}(z)$, and
hence $\chi_{\lambda}(z)$ converges for $z<\abs{\Omega}^{-1}$. The
next proposition gives a mild condition under which the critical
point is nontrivial.
\begin{proposition}
  \label{prop:LWW-Trivial-G-Bound}
  Let $\bar\lambda = \sup_{\eta}\lambda_{\eta}>1$. If
  $z<(\abs{\Omega}\sqrt{\bar\lambda})^{-1}$ then $\chi_{\lambda}(z)$
  is finite.
\end{proposition}
\begin{proof}
  An $n$-step walk contains at most $\lfloor n/2\rfloor$
  loops, and weighting each loop by $\bar\lambda$ yields an
  upper bound for $\chi_{\lambda}(z)$. Cancelling the factors of
  $\sqrt{\bar\lambda}$ gives the claim, as the resulting sum is
  $\chi_{1}(\bar z)$ for some $\bar z < \abs{\Omega}^{-1}$.
\end{proof}

If $\cc R$ is an isometry of $\Z^{d}$, and $A\subset \Z^{d}$, let $\cc
R A = \{\cc Ra \mid a\in A\}$.
\begin{assumption}
  \label{ass:LWW-Symmetry}
  Assume that $\lambda_{\eta} = \lambda_{\cc R \eta}$ for any isometry
  $\cc R$ and any $\eta\in\SAP$. Further assume that $\lambda_{\eta} =
  \lambda_{\tilde \eta}$ if $\eta$ and $\tilde \eta$ are self-avoiding
  polygons that differ only in terms of initial vertex and
  orientation.
\end{assumption}
\begin{assumption}
  \label{ass:LWW-Bounded}
  Assume $\sup_{\eta\in\SAP}\lambda_{\eta}<\infty$.
\end{assumption}

\begin{theorem}
  \label{thm:LWW-Main}
  Fix $\lambda\geq 0$. If \Cref{ass:LWW-Symmetry} and
  \Cref{ass:LWW-Bounded} hold, then there exists
  $d_{0}=d_{0}(\lambda)$ such that for $d\geq d_{0}$ there are
  constants $A$ and $D$ such that the conclusions of
  \Cref{thm:LWW-Main-Intro} hold.
\end{theorem}

\Cref{thm:LWW-Main-Intro} is the special case of \Cref{thm:LWW-Main}
when the loop activities $\lambda$ are constant. The constants $A$ and
$D$ have explicit expressions, see
\Cref{sec:LWW-Susceptibility-MF}. For the remainder of the article it
will be assumed that \Cref{ass:LWW-Symmetry} and
\Cref{ass:LWW-Bounded} hold.

\subsubsection{Aspects of Proof}
\label{sec:LWW-Proof-Outline}

This section describes the basic facts about $\lambda$-LWW that allow
for a lace expansion analysis, and gives an outline of the proof of
\Cref{thm:LWW-Main}.

\begin{definition}
  \label{def:LWW-LELWW}
  The \emph{loop-erased $\lambda$-LWW} weight $\bar w_{\lambda,z}$ on
  self-avoiding walks is
  \begin{equation}
    \label{eq:LWW-LELWW}
    \bar w_{\lambda,z}(\eta) = \indicatorthat{\eta\in\SAW}
    \sum_{\omega \colon \LE(\omega)=\eta} w_{\lambda,z}(\omega).
  \end{equation}
\end{definition}
Note that the definition of $\bar w_{\lambda,z}$ assigns non-zero
weight only to self-avoiding walks. The definition of $\bar
w_{\lambda,z}$ implies that for any $x\in \Z^{d}$
\begin{equation}
  \label{prop:LWW-2PT-Equivalence}
  \sum_{\omega\colon 0\to x} \bar w_{\lambda,z}(\omega) =
    \sum_{\omega\colon 0 \to x} w_{\lambda,z}(\omega),
\end{equation}
as the left-hand side is just a reorganization of the right-hand
side. This identity will be important in what follows. 

\begin{definition}
  \label{def:LWW-range}
  The \emph{range}, $\range{\omega}$, of a walk $\omega$ is the set of
  vertices visited by $\omega$.
\end{definition}

The \emph{$\lambda$-LWW loop measure at activity $z$} of a closed walk
$\omega$ is given by $w_{\lambda,z}(\omega) / \abs{\omega}$. The next
definition introduces a convenient shorthand for the loop measure of
certain subsets of walks; note that $\mu_{\lambda,z}$ is not a
measure.

\begin{definition}
  \label{def:LWW-LM}
  Let $A,B\subset \Z^{d}$. The \emph{$\lambda$-LWW loop measure}
  $\mu_{\lambda,z}(A;B)$ is
  \begin{equation}
    \label{eq:LWW-LM}
    \mu_{\lambda,z}(A;B) = \sum_{x} \mathop{\sum_{\omega\colon x\to
        x}}_{\abs{\omega} \geq 1}
    \indicatorthat{\range{\omega} \cap A\neq \emptyset} 
    \indicatorthat{\range{\omega} \cap B = \emptyset} 
    \frac{ w_{\lambda,z}(\omega)}{\abs{\omega}}.
  \end{equation}
  Define $\mu_{\lambda,z}(A) = \mu_{\lambda,z}(A;\emptyset)$. For
  singleton sets $\{x\}, \{y\}$, let $\mu_{\lambda,z}(x;y) =
  \mu_{\lambda,z}(\{x\};\{y\})$.
\end{definition}

For the special case of $\lambda=1$ the next theorem
is~\cite[Proposition~9.5.1]{LawlerLimic2010}.
\begin{theorem}
  \label{thm:LWW-LM-Rep}
  The loop erased $\lambda$-LWW weight on self-avoiding walks can be
  written in terms of the $\lambda$-LWW loop measure:
  \begin{equation}
    \label{eq:LWW-LM-Rep}
    \bar w_{\lambda,z}(\eta) = \sum_{\omega\colon\LE(\omega)=\eta}
    w_{\lambda,z}(\omega) = z^{\abs{\eta}} \exp (\mu_{\lambda,z}(\range{\eta})).
  \end{equation}
\end{theorem}
\begin{proof}
  Deferred to~\Cref{sec:LWW-Viennot}.
\end{proof}

A function $f$ on subsets of $\Z^{d}$ is said
to be \emph{weakly increasing} if $A\subset B$ implies $f(A) \leq
f(B)$, and \emph{weakly decreasing} if $f(A)\geq f(B)$.
\begin{proposition}
  \label{prop:LWW-LM-Properties}
  Assume $z\geq 0$, $\lambda\geq 0$.
  \begin{enumerate}
    \item Let $A,B\subset \Z^{d}$. Then for any isometry $\cc R$
      \begin{equation}
        \label{eq:LWW-LM-Symmetries}
        \mu_{\lambda,z}(\cc R A;\cc R B) = \mu_{\lambda,z}(A;B),
      \end{equation}
    \item $\mu_{\lambda,z}(A;B)$ is weakly increasing in $A$ and
      weakly decreasing in $B$.    
  \end{enumerate}
\end{proposition}
\begin{proof}
  The first item follows from the isometry invariance of
  $w_{\lambda,z}$, which follows from \Cref{ass:LWW-Symmetry}. The second
  follows as increasing $A$ (decreasing $B$) reduces (increases) the
  constraints on the set of walks that contribute to the defining sum,
  and $w_{\lambda,z}(\omega)\geq 0$.
\end{proof}

If $\eta = \eta_{1}\circ \eta_{2}$ is self-avoiding then
\Cref{thm:LWW-LM-Rep} and the definition of the loop measure imply
\begin{equation}
  \label{eq:LWW-LELWW-Repulsive}
  \bar w_{\lambda,z}(\eta) = z^{\abs{\eta_{1}}} z^{\abs{\eta_{2}}}
  \exp( \mu_{\lambda,z}(\range{\eta_{1}})) \exp
  (\mu_{\lambda,z}(\range{\eta_{2}} ; \range{\eta_{1}})).
\end{equation}
By the second statement of \Cref{prop:LWW-LM-Properties} dropping the
constraint in the second loop measure increases the weight, and hence
\emph{loop-erased $\lambda$-LWW} is purely repulsive.  This enables a
lace expansion analysis of $\lambda$-LWW as the two-point functions of
$\lambda$-LWW and loop-erased $\lambda$-LWW coincide by
\Cref{prop:LWW-2PT-Equivalence}. This is done as follows:
\begin{itemize}
\item \Cref{sec:LWW-Lace} derives a lace expansion for $\lambda$-LWW. This
  is done by manipulating the identity
  \begin{equation}
    \label{eq:LWW-Intro-X-Gas}
    \bar w_{\lambda,z} (\eta) = z^{\abs{\eta}}\indicatorthat{\eta\in\SAW}
    \prod_{X\in\cc X}(1+\alpha_{X})^{\indicatorthat{\ell(X) \cap \range{\eta}\neq \emptyset}},
  \end{equation}
  where 
  \begin{align}
    \label{eq:LWW-Specialization-1}
    \cc X &= \cup_{x\in\Z^{d}}\set{ \omega\colon x\to x,
      \abs{\omega}\geq 1}, \\
    \label{eq:LWW-Specialization-2}
    \ell(\omega) &= \range{\omega}, \\
    \label{eq:LWW-Specialization-3}
    \alpha_{\omega} &= \exp \ob{ \frac{ w_{\lambda,z}(\omega)}
      {\abs{\omega}}} - 1.
  \end{align}
  In \Cref{eq:LWW-Specialization-1} the condition $\abs{\omega}\geq 1$
  can be relaxed to $\abs{\omega}>1$ as all closed walks have length
  at least $2$. Note that $\alpha_{\omega}\geq 0$ for any closed walk
  $\omega$ as $\lambda\geq 0$, and that the product in
  \Cref{eq:LWW-Intro-X-Gas} converges for $z$ sufficiently small by
  \Cref{prop:LWW-Trivial-G-Bound}.
\item \Cref{sec:LWW-LM-Formulation} expresses the results
  of~\Cref{sec:LWW-Lace} in terms of $\mu_{\lambda,z}$, as opposed to
  the variables $\alpha_{\omega}$.
\item \Cref{sec:LWW-Convergence} and
  \Cref{sec:LWW-Diagrammatic-Bounds} prove the convergence of the lace
  expansion at the critical point. The strategy is based
  on~\cite{Slade2006}.
\item Lastly, \Cref{sec:LWW-Further} proves the main theorem after
  establishing some further estimates on the lace expansion
  coefficients. The analysis is based on~\cite{MadrasSlade2013}.
\end{itemize}

Before carrying out the arguments outlined above, let us briefly
comment on other relevant non-repulsive random walks that have
been studied. Ueltschi~\cite{Ueltschi2002} has given a lace expansion
analysis of a self-avoiding walk with nearest neighbour attractions;
the attraction means his model is not repulsive. The analysis
in~\cite{Ueltschi2002} overcomes the lack of repulsion by exploiting
the self-avoiding nature of the walk. Implementing this idea requires
technical assumptions that (i) the attraction is sufficiently weak and
(ii) the self-avoiding walk can take steps of unbounded range. A
second non-repulsive model that has been studied is excited random
walk: the analysis of this model in~\cite{vdHofstadHolmes2012} is
essentially a lace expansion analysis. These results have a somewhat
different flavour as the walk being studied has non-zero
speed. Roughly speaking, the lack of repulsion is overcome by using
the transience of the walk in the excited direction.

\subsection{Notation and Conventions}
\label{sec:LWW-Conventions}

Let $\indicatorthat{A}$ denote the indicator function
of a set $A$. For notational ease we will occasionally also make use of the
Kronecker delta $\delta_{x,y} = \indicatorthat{x=y}$. The \emph{single
  step distribution} $D(x)$ is defined by $D(x) = \abs{\Omega}^{-1}
\indicatorthat{x\sim 0}$, where we recall that $\abs{\Omega}=2d$ and
$x\sim 0$ indicates that $x$ is a nearest neighbour of $0$ in
$\Z^{d}$.

The Fourier transform $\hat f\colon \FS^{d}\to \C$ of a function $f$ on
$\Z^{d}$ is defined by
\begin{equation}
  \label{eq:LWW-FT}
  \hat f(k) = \sum_{x\in \Z^{d}} e^{ik\cdot x}f(x)
\end{equation}

Subwalks of a walk $\omega$ can be identified by specifying the
subinterval that defines them. That is, for $0\leq a <b\leq
\abs{\omega}$ define $\omega\cb{a,b} = (\omega_{a}, \dots,
\omega_{b})$, $\omega\co{a,b} = \omega\cb{a,b-1}$, $\omega\oc{a,b} =
\omega\cb{a+1,b}$, and $\omega\ob{a,b} = \omega\cb{a+1,b-1}$. By
convention $\cb{a,a}=\{a\}$, so $\omega\cb{a,a}=\omega_{a}$. To avoid
some ungainly notation, let $\omega\cb{a\splice} =
\omega\cb{a,\abs{\omega}}$.

By convention $\inf\emptyset=\infty$ and $\sup\emptyset =
-\infty$. The set $\{0,1,\dots, n\}$ will be denoted $\cb{n}$, and
$\cb{\omega}$ will denote $\cb{\abs{\omega}}$ when $\omega$ is a
walk. Further, $c$ will denote a positive constant independent of the
dimension $d$ and activity $z$; the precise value of $c$ may change
from line to line.

\section{A Lace Expansion}
\label{sec:LWW-Lace}

\begin{remark}
  \label{rem:LWW-Memory-Expansion}
  The lace expansion presented here can be derived by other means,
  e.g., the technique developed for self-interacting walks
  in~\cite{vdHofstadHolmes2012}.
\end{remark}

\subsection{Graphical Representations}
\label{sec:LWW-Graph-Representation}

This section provides a representation of the weight $\bar
w_{\lambda,z}$ in terms of graphs. The utility of such a
representation is that it allows recursive identities to be
derived.

\subsubsection{Graph Representation of Self Avoidance}
\label{sec:LWW-Timelike}

\begin{definition}
  \label{def:LWW-Graph}
  Let $A$ be a set. For $s,t\in A$, $s\neq t$, the pair $\{s,t\}
  \equiv st$ is called an \emph{edge}. A \emph{graph} $\Gamma$ on $A$
  is a set of edges.
\end{definition}

The condition $\omega\in\SAW$ that a walk $\omega$ is self-avoiding
can be expressed using graphs.
\begin{equation}
  \label{eq:LWW-Timelike}
  \indicatorthat{\omega\in \SAW} = \prod_{0\leq s<t\leq\abs{\omega}}
  \indicatorthat{\omega_{s}\neq\omega_{t}} 
  =  \prod_{0\leq s<t\leq \abs{\omega}}(1-\indicatorthat{\omega_{s}=\omega_{t}}) 
  = \sum_{\Gamma} \prod_{st\in\Gamma}
  \ob{-\indicatorthat{\omega_{s}=\omega_{t}}},
\end{equation}
The sum in the rightmost term is over all graphs $\Gamma$ on
$\cb{\omega}$, where we recall the definition $\cb{\omega} =
\cb{\abs{\omega}} = \{0,1,\dots, \abs{\omega}\}$.

\subsubsection{Hypergraph Decomposition of LWW Weight}
\label{sec:LWW-Spacelike}

A representation of the weight on self-avoiding walks due to the product
over $\cc X$ in \Cref{eq:LWW-Intro-X-Gas} is less straightforward than the
graph representation of self-avoidance. This is because the condition
of self-avoidance involves two distinct times, while the
condition that $\range{\omega}\cap \ell(X)\neq\emptyset$ involves
many distinct times. This issue can be handled by using
inclusion-exclusion. A convenient way to represent the results of
inclusion-exclusion is in terms of hypergraphs. We emphasize, however,
that no prior knowledge of hypergraphs are needed to understand the
expansion -- they are only used as a bookkeeping instrument.

\begin{lemma}
  \label{lem:LWW-Hypergraph-Representation}
  Let $\omega$ be a walk, and let $X\in \cc X$. Then
  \begin{align}
  \label{eq:LWW-Hypergraphs-General-1}
  (1+\alpha_{X})^{\indicatorthat{\ell(X)\cap
      \range{\omega}\neq\emptyset}}
  &= \mathop{\prod_{J\subset \cb{\omega}\colon \abs{J}\geq 1}}(1+\hyp{J}{X}(\omega)),
\end{align}
  where 
  \begin{equation}
    \label{eq:LWW-Hypergraphs-General-3}
    \hyp{J}{X}(\omega) \equiv
    \begin{cases}
      \phantom{-}\alpha_{X} \prod_{j\in J}\indicatorthat{\omega_{j}
        \in \ell(X)},
      & \abs{J}\in 2\N + 1\\
      -\frac{\alpha_{X}}{1+\alpha_{X}} \prod_{j\in
        J}\indicatorthat{\omega_{j} \in \ell(X)} & \abs{J}\in 2\N.
    \end{cases}
  \end{equation}
  In \Cref{eq:LWW-Hypergraphs-General-3} $0$ is included in $2\N$.
\end{lemma}
\begin{proof}
  Apply inclusion-exclusion to the condition $\ell(X)\cap
  \range{\omega}\neq\emptyset$:
  \begin{align}
    \indicatorthat{\ell(X)\cap \range{\omega}\neq\emptyset} &=
    1 - \indicatorthat{\ell(X)\cap \range{\omega}=\emptyset} \\
    &= 1 -  \prod_{j=0}^{\abs{\omega}}(1- \indicatorthat{\omega_{j}\in\ell(X)}) \\
    &= \mathop{\sum_{J\subset\cb{\omega}\colon\abs{J}\geq 1}}
    (-1)^{\abs{J}+1} \prod_{j\in
      J}\indicatorthat{\omega_{j}\in\ell(X)}.
  \end{align}
  Then
  \begin{align}
    (1+\alpha_{X})^{\indicatorthat{\ell(X)\cap
      \range{\omega}\neq\emptyset}} &= \prod_{J\subset
    \cb{\omega}\colon \abs{J}\geq 1}
  (1+\alpha_{X})^{(-1)^{\abs{J}+1}\prod_{j\in
      J}\indicatorthat{\omega_{j}\in \ell(X)}} \\
  &= \mathop{\prod_{J\subset \cb{\omega}\colon \abs{J}\geq 1}}(1+\hyp{J}{X}(\omega)),
  \end{align}
  where the weights $\hyp{J}{X}$ arise from the identities 
  $(1+\alpha)^{-\indicator_{A}} = 1 - \frac{\alpha}{1+\alpha}\indicator_{A}$ and 
  $(1+\alpha)^{\indicator_{A}} = 1+ \alpha\indicator_{A}$.
\end{proof}

\begin{definition}
  \label{def:LWW-Hypergraph}
  A \emph{hypergraph} $G$ on a countable set $A$ is a (possibly empty)
  finite subset of $A$. Each element of $G$ is called a \emph{hyperedge}.
\end{definition}
To connect this definition with the more familiar notion of a graph,
consider the case when $A$ is $V^{2}\setminus \{ \{x,x\} \mid x\in
V\}$ for $V$ a finite set. A subset of $A$ is then the edge set of a
graph on $V$.

If $F(a)$ is an indeterminate associated to the hyperedge $a$ then, as
formal power series,
\begin{equation}
  \label{eq:LWW-Hypergraph-Product}
  \prod_{a\in A}(1+ F(a))= \sum_{G} \prod_{a\in G}F(a),
\end{equation}
where the sum on the right-hand side
of~\eqref{eq:LWW-Hypergraph-Product} is over all hypergraphs on
$A$. In what follows we perform calculations in the sense of formal
power series. We will ultimately find that our final
expressions have interpretations as convergent objects.

To represent the product over $\cc X$ in
\Cref{eq:LWW-Intro-X-Gas} in terms of hypergraphs take
$A$ in \Cref{def:LWW-Hypergraph} to be $(2^{\cb{n}}\setminus
\emptyset) \times \cc X$. If $a\in A$ then $a=(J,X)$ for $J$ a
non-empty subset of $\cb{n}$ and $X\in \cc X$. Define $F(a) =
F_{J,X}$. This implies
\begin{align}
  \label{eq:LWW-Hypergraph-Product-1}
  \prod_{X\in \cc X}(1+\alpha_{X})^{\indicatorthat{\ell(X)\cap
      \range{\omega}\neq \emptyset}} &= \prod_{X\in \cc
    X} \mathop{\prod_{J\subset \cb{\omega}\colon \abs{J}\geq 1}}(1+F_{J,X}(\omega)) \\
  \label{eq:LWW-Hypergraph-Product-2}
   &= \sum_{G}\prod_{(J,X)\in G}F_{J,X}(\omega),
\end{align}
where the sum in~\eqref{eq:LWW-Hypergraph-Product-2} is over all
hypergraphs. 

The next corollary is a useful hypergraph representation
of the weight carried by a subwalk.

\begin{corollary}
  \label{cor:LWW-Remainder}
  Let $\omega$ be an $n$-step walk. For $k\leq n$, $X\in \cc X$,
  \begin{equation}
    \label{eq:LWW-Remainder}
    (1+\alpha_{X})^{ \indicatorthat{\range{\omega\co{0,k}} \cap
        \ell(X) = \emptyset} \indicatorthat{\range{\omega\cb{k,n}}
        \cap \ell(X)\neq \emptyset}} =
    \mathop{\prod_{J\subset \cb{n}\colon \abs{J}\geq 1,}}_{J\cap
      \cb{k,n}\neq\emptyset}(1+\hyp{J}{X}(\omega)).
  \end{equation}
\end{corollary}
\begin{proof}
  Observe that
  $\indicatorthat{\range{\omega\cb{k,n}}\cap \ell(X)\neq
    \emptyset}\indicatorthat{\range{\omega\co{0,k}} \cap \ell(X) =
    \emptyset} $
  can be rewritten as
  $\indicatorthat{\range{\omega}\cap \ell(X)\neq\emptyset} -
  \indicatorthat{\range{\omega\co{0,k}}\cap \ell(X)\neq
    \emptyset}$.
  The corollary follows by applying
  \Cref{lem:LWW-Hypergraph-Representation} to both $\omega$ and
  $\omega\co{0,k}$ and dividing.
\end{proof}

\subsubsection{The Full Graphical Representation}
\label{sec:LWW-Hypergraph-Representation}

\begin{definition}
  Let $J\subset \cb{n}$ be non-empty and let $X$ denote an
  element of $\cc X\cup\{\emptyset\}$. A pair $(J,X)$ is
  \emph{timelike} if $\abs{J}=2$, $X=\emptyset$. A pair is
  \emph{spacelike} if $X\neq\emptyset$. 
\end{definition}
The use of spacelike and timelike as labels has no relation
to the use of these terms in physics. Extend the definition of
$\hyp{J}{X}$ by defining $\hyp{J}{X}$ via
\eqref{eq:LWW-Hypergraphs-General-3} if $(J,X)$ is spacelike, and
defining $\hyp{st}{\emptyset} =
-\indicatorthat{\omega_{s}=\omega_{t}}$ for timelike hyperedges
$(st,\emptyset)$. Let $\IH\cb{a,b}$ denote the set of hypergraphs
whose hyperedges are pairs $(J,X)$ such that (i) $X\in\cc
X\cup\{\emptyset\}$, (ii) $J\subset \{a,a+1,\dots, b\}$, $\abs{J}\geq
1$, and (iii) $X=\emptyset$ implies $\abs{J}=2$. Define $\IH(n)\equiv
\IH\cb{0,n}$. The decompositions of
\Cref{sec:LWW-Graph-Representation} imply that
\begin{align}
  \label{eq:LWW-Hypergraph-Interaction}
  c_{n}(0,x) &=\mathop{\sum_{\omega\colon 0 \to x}}_{\abs{\omega}=n}
  \indicatorthat{\omega\in\SAW} \prod_{X\in\cc
    X}(1+\alpha_{X})^{\indicatorthat{ \ell(X)\cap
      \range{\omega} \neq \emptyset}} \\
  \label{eq:LWW-Hypergraph-Interaction-1}
  &= \mathop{\sum_{\omega\colon 0 \to x}}_{\abs{\omega}=n}
  \sum_{G\in\IH(n)} \prod_{(J,X)\in G}\hyp{J}{X}(\omega).
\end{align}

\subsection{Lace Graphs}
\label{sec:LWW-Laces}

\begin{definition}
  \label{def:LWW-Lace-Connectedness}
  A graph $\Gamma$ on $\cb{a,b}$ is \emph{(lace) connected} if (i)
  $b>a+1$, (ii) for all $a<j<b$ there is an edge $st\in\Gamma$ such
  that $s<j<t$ and (iii) there are $j_{1},j_{2}$ such that $aj_{1}$,
  $j_{2}b\in\Gamma$. Let $\graphs\cb{a,b}$ (resp.\ $\cgraphs\cb{a,b}$)
  denote the set of graphs (resp.\ lace connected graphs) on $\cb{a,b}$.
\end{definition}

We caution the reader that the definition of lace connectedness is not
the same as the graph theoretical definition of connectedness. The
adjective lace will be dropped in what follows, as the graph-theoretic
notion of connectedness is not relevant in this section.

A function $w$ on graphs on the discrete interval $\cb{a,b}$ is called
\emph{multiplicative} if $w(G) = \prod_{st\in E(G)}w(st)$. Note that a
multiplicative function on graphs assigns the empty graph weight $1$. If
$w$ is a multiplicative function on graphs on $\cb{a,b}$ define
\begin{equation}
  K\cb{a,b} = \sum_{G\in\graphs\cb{a,b}}w(G), \qquad J\cb{a,b} =
  \sum_{G\in\cgraphs\cb{a,b}} w(G),
\end{equation}
and let $K\cb{a,b}=J\cb{a,b}=0$ if $a>b$. For $a<b$ the
observation that a graph on $\cb{a,b}$ either contains $a$ in a
connected subgraph or does not and the definition of connectedness
imply
\begin{equation}
  \label{eq:LWW-Connectedness-Recursion}
  K\cb{a,b} = K\cb{a,a+1}K\cb{a+1,b} + \sum_{j\geq 2}J\cb{a,a+j} K\cb{a+j,b}.
\end{equation}

\begin{definition}
  \label{def:lace-graph}
  A graph is a \emph{lace graph} if the removal of any edge results in
  a graph which is not connected.
\end{definition}

A \emph{labelled graph} is a graph where each edge is given a label of
either spacelike or timelike; a labelled graph may contain both the
edge $(st,\mathrm{spacelike})$ and the edge
$(st,\mathrm{timelike})$. The definition of a lace graph applies to
labelled graphs as the notion of connectedness does not depend on the
labelling. The following procedure associates a unique lace
$L_{\Gamma}$ to each labelled connected graph $\Gamma$ on
$\cb{a,b}$. The labelled lace $L_{\Gamma}$ consists of the set of
edges $s_{i}t_{i}$ along with their labellings, where $s_{i}t_{i}$ are
determined by $s_{1}=a$,
$t_{1} = \max \{v \colon s_{1}v \in \Gamma\}$,
$t_{i+1} = \max\{ v \colon \exists\, \textrm{$s<t_{i}$ such that
  $sv\in \Gamma$}\}$, and
$s_{i+1} = \min \{s\colon st_{i+1}\in \Gamma\}$.  If this does not
uniquely specify $s_{i}t_{i}$ then $s_{i}t_{i}$ is chosen to have the
label spacelike. The procedure terminates when $t_{i+1}=b$. See
\Cref{fig:LWW-Lace}.

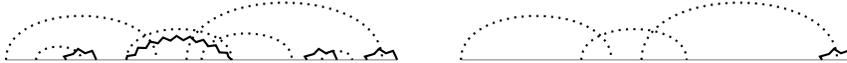
\begin{figure}[h]
  \centering
  \beginpgfgraphicnamed{fig2a}
  \begin{tikzpicture}[scale=.20]
    \draw[gray] (0,0) -- (26,0);
    \draw[black,dotted, thick] (0,0) to[out=90,in=90] (10,0);
    \draw[black,dotted,thick] (2,0) to[out=90,in=90] (5,0);
    \draw[black, thick] decorate [decoration={zigzag, segment length=5,
      amplitude = .8}] {(4,0) to[out=90,in=90] (6,0)};
    \draw[black,dotted,thick] (8,0) to[out=90,in=90] (15,0);
    \draw[black, thick] decorate [decoration={zigzag, segment length=5,
      amplitude = .8}] {(8,0) to[out=45,in=135] (15,0)};
    \draw[black,dotted,thick] (12,0) to[out=90,in=90] (25,0);
    \draw[black,dotted,thick] (13,0) to[out=90,in=90] (19,0);
    \draw[black, thick] decorate [decoration={zigzag, segment length=5,
      amplitude = .8}] {(20,0) to[out=90,in=90] (22,0)};
    \draw[black,dotted,thick] (21,0) to[out=90,in=90] (23,0);
    \draw[black, thick] decorate [decoration={zigzag, segment length=5,
      amplitude = .8}] {(24,0) to[out=90,in=90] (26,0)};
  \end{tikzpicture}
  \endpgfgraphicnamed
  \qquad
  \beginpgfgraphicnamed{fig2b}
    \begin{tikzpicture}[scale=.20]
    \draw[gray] (0,0) -- (26,0);
    \draw[black,dotted,thick] (0,0) to[out=90,in=90] (10,0);
    \draw[black,dotted,thick] (8,0) to[out=90,in=90] (15,0);
    \draw[black,dotted,thick] (12,0) to[out=90,in=90] (25,0);
    \draw[black, thick] decorate [decoration={zigzag, segment length=5,
      amplitude = .8}] {(24,0) to[out=90,in=90] (26,0)};
  \end{tikzpicture}
  \endpgfgraphicnamed
  \caption[A labelled graph and the corresponding labelled lace
  graph]{A labelled graph and the corresponding labelled lace
    graph. The left-hand side depicts a connected labelled graph,
    while the right-hand side depicts the corresponding labelled lace
    graph. The dotted black edges are labelled spacelike, while the
    solid black zigzag edges are labelled timelike.}
  \label{fig:LWW-Lace}
\end{figure}

A labelled edge $st$ is said to be \emph{compatible} with a lace $L$ if
$L_{L\cup\{st\}} = L$, i.e., if the addition of the labelled edge $st$ does not
alter the outcome of the above algorithm. Let $\cc L\cb{a,b}$ denote
the set of labelled lace graphs on $\cb{a,b}$ and $\cc C(L)$ the set of
compatible labelled edges for a lace $L\in\cc L$.

\begin{lemma}
  \label{lem:LWW-Lace-Prescription}
  Let $w$ be a weight on labelled edges $st$. Then
  \begin{equation}
    \sum_{\Gamma\in \cgraphs\cb{a,b}}\prod_{st\in \Gamma}w(st) =
    \sum_{L\in\cc L\cb{a,b}} \prod_{st\in L}
    w(st)\prod_{s^{\prime}t^{\prime}\in\cc
      C(L)}(1+w(s^{\prime}t^{\prime})),
  \end{equation}
  where the sums are over labelled connected graphs and labelled
  laces, respectively.
\end{lemma}
\begin{proof}
  The proof is the same as the proof for unlabelled
  graphs, see~\cite{BrydgesSpencer1985}, \cite{Slade2006}, or
  \cite{Zeilberger1997}.
\end{proof}

\begin{remark}
  \label{rem:LWW-Connectedness}
  \Cref{def:LWW-Lace-Connectedness} is \emph{not} the definition of
  lace connectedness typically used for self-avoiding walk, as the
  graph consisting of the single edge $\{a,a+1\}$ is not being
  considered connected. This change is entirely cosmetic for
  self-avoiding walk as graphs consisting of a single edge $\{a,a+1\}$
  do not contribute.
\end{remark}

\subsection{Laces and Hypergraphs}
\label{sec:LWW-Hyperlaces}

This section obtains an analogue of \Cref{lem:LWW-Lace-Prescription} for
hypergraphs.

\subsubsection{Recursion Relation for Hypergraphs}
\label{sec:LWW-Recursion}

\begin{definition}
  \label{def:LWW-Graph-Spans}
  For a hyperedge $(J,X)$ define $\edgespan{(J,X)} = \{\min J, \max
  J\}$. If $(J,X)$ is spacelike label
  $\edgespan{(J,X)}$ spacelike, and if $(J,X)$ is timelike label
  $\edgespan{(J,X)}$ timelike. If $G$ is a
  hypergraph the labelled graph $\Gamma_{G}$ with labelled edges
  $\{\edgespan{(J,X)} \mid (J,X)\in G\}$ will be called the
  \emph{graph of spans} of $G$.
\end{definition}

\begin{definition}
  A hypergraph $G$ on $\cb{a,b}$ is \emph{connected} if the graph of
  spans of $G$ is connected on $\cb{a,b}$. The set of connected
  hypergraphs on $\cb{a,b}$ is denoted $\IH^{c}\cb{a,b}$.
\end{definition}

The objects $\alpha$ and $\alpha_{0}$ in the next
definition have interpretations in terms of the loop measure, but
for now should be thought of as convenient shorthand.
\begin{definition}
  \label{def:LWW-Renormalized-Activity}
  Let $\cc X_{0} = \{X\in\cc X \mid 0\in\ell(X)\}$ and let $y\in\Omega$ be
  a vertex adjacent to $0$. Define
  \begin{equation}
    \label{eq:LWW-RA}
    \alpha_{0} = \alpha_{0}(\cc X) =  \prod_{X\in\cc X_{0}}(1+\alpha_{X}), \qquad
    \alpha = \alpha(\cc X) = \prod_{X\in \cc
      X_{0}}(1+\alpha_{X})^{\indicatorthat{y\notin \ell(X)}},
  \end{equation}
  That $\alpha$ is independent of the vertex $y\in\Omega$ chosen
  follows from the isometry invariance of the loop-weighted walk
  weight.
\end{definition}

By translation invariance $\alpha_{0}$ is also given by the product
over $X\in \cc X$ such that any single fixed vertex is contained in
$\ell(X)$, and hence
\begin{equation}
  \label{eq:LWW-Alpha0-Graphs}
  \alpha_{0} = \sum_{G\in \IH\cb{1,1}} w(G),
\end{equation}
where $w(G) = \prod_{(J,X)\in G}\hyp{J}{X}$. Using
\Cref{eq:LWW-Alpha0-Graphs} and the definition of connectedness for
hypergraphs implies that for $n\geq 1$
\begin{multline}
  \label{eq:LWW-Hypergraph-Recursion-A}
  \sum_{G\in
  \IH\cb{0,n}}w(G) =
  \alpha_{0}^{-1}\sum_{G_{1}\in\IH\cb{0,1}}
  \sum_{G_{2}\in\IH\cb{1,n}}w(G_{1}) w(G_{2}) \\+
  \alpha_{0}^{-1}\sum_{j\geq 2} 
  \sum_{G_{1}\in\IH^{c}\cb{0,j}} \sum_{G_{2}\in \IH\cb{j,n}} w(G_{1}) w(G_{2}).
\end{multline}
The factor of $\alpha_{0}^{-1}$ multiplying the first term arises
since the hypergraphs $G\in\IH\cb{1,1}$ are double counted due to
being present in both $\IH\cb{0,1}$ and $\IH\cb{1,n}$. The factor of
$\alpha_{0}^{-1}$ multiplying the second factor arises similarly, due
to double counting of the sum over $\IH\cb{j,j}$; translation
invariance implies this is the same as the sum over
$\IH\cb{1,1}$. The next lemma simplifies
\Cref{eq:LWW-Hypergraph-Recursion-A} by computing the sum over
$\IH\cb{0,1}$.

\begin{lemma}
  \label{lem:LWW-Hypergraph-Recursion}
  Fix $n\geq 1$. Then $\sum_{G\in \IH\cb{0,n}} \prod_{(J,X)\in
    G}\hyp{J}{X}$ is equal to
  \begin{equation}
    \label{eq:LWW-Hypergraph-Recursion}
    \alpha \sum_{G\in\IH\cb{1,n}}\prod_{(J,X)\in G}\hyp{J}{X}
    + \alpha_{0}^{-1}\sum_{j\geq 2}
    \sum_{G_{1}\in\IH^{c}\cb{0,j}}\sum_{G_{2}\in \IH\cb{j,n}}
    \prod_{(J,X)\in G_{1}} \hyp{J}{X} \prod_{(J^{\prime},X^{\prime})\in G_{2}}
    \hyp{J^{\prime}}{X^{\prime}}
  \end{equation}
\end{lemma}
\begin{proof}
  Let $\omega$ be a walk. \Cref{lem:LWW-Hypergraph-Representation}
  and~\eqref{eq:LWW-Timelike} imply that
  \begin{equation}
    \label{eq:LWW-HR-2}    
    \sum_{G\in\IH\cb{0,1}}\prod_{(J,X)\in G}\hyp{J}{X}(\omega) =
    \indicatorthat{\omega_{0}\neq \omega_{1}} \prod_{X\in\cc
      X}(1+\alpha_{X})^{\indicatorthat{\{\omega_{0},\omega_{1}\}\cap \ell(X)\neq \emptyset}}.
  \end{equation}
  The constraint that $\omega_{0}\neq \omega_{1}$ is irrelevant as
  $\omega_{j+1}\neq \omega_{j}$ for any walk. Using the
  representation of $\alpha_{0}$ in \Cref{eq:LWW-Alpha0-Graphs} gives
  \begin{equation}
    \label{eq:LWW-HR-3}
    \frac{ \sum_{G\in\IH\cb{0,1}} \prod_{(J,X)\in
        G}\hyp{J}{X}(\omega)} {\sum_{G\in \IH\cb{1,1}} 
      \prod_{(J,X)\in G}\hyp{J}{X} (\omega)} = \prod_{X\in \cc X}
    (1+\alpha_{X})^{\indicatorthat{\omega_{1}\in\ell(X)}\indicatorthat{\omega_{0}\notin
        \ell(X)}},
  \end{equation}
  and this last quantity is $\alpha$ by
  \Cref{eq:LWW-RA}. Using~\eqref{eq:LWW-Alpha0-Graphs}
  and~\eqref{eq:LWW-HR-3} to
  rewrite~\eqref{eq:LWW-Hypergraph-Recursion-A} gives the
  claim.
\end{proof}

\subsubsection{Laces for Hypergraphs and Weights on Lace Edges}
\label{sec:LWW-Hypergraph-Lace}

The weight $w(G) = \prod \hyp{J}{X}$ on hypergraphs can be pushed
forward to a weight $w_{\star}^{\omega}(st)$ on labelled graphs;
recall that labelled graphs were introduced following
\Cref{def:lace-graph}. Explicitly, the weight $w_{\star}^{\omega}(st)$
is defined by
\begin{align}
  \label{eq:LWW-Pushforward-Weight-1}
  w_{\star}^{\omega}(st,\mathrm{timelike}) &\equiv
  -\indicatorthat{\omega_{s}=\omega_{t}} \\
  \label{eq:LWW-Pushforward-Weight-2}
  w_{\star}^{\omega}(st,\mathrm{spacelike}) &\equiv
  (1-\indicatorthat{\omega_{s}=\omega_{t}})\!\!\!\sum_{\{(J_{i},X_{i})\}\colon
    \edgespan(J_{i},X_{i})=st} \prod_{i}\hyp{J_{i}}{X_{i}}(\omega).
\end{align}
The sum for a spacelike edge in~\eqref{eq:LWW-Pushforward-Weight-2} is
over all non-empty collections of hyperedges, each of whose span is
the labelled edge $(st,\mathrm{spacelike})$. The factor
$(1-\indicatorthat{\omega_{s}=\omega_{t}})$ accounts for the
possibility that a timelike hyperedge exists when the edge $st$ is
given the label spacelike. Note that this weight neglects hyperedges
$(J,X)$ with $\abs{J}=1$. For notational ease let $\hyp{j}{X} =
\hyp{\{j\}}{X}$.

\begin{lemma}
  \label{lem:LWW-Hypergraph-Lace-Lift}
  The following identity holds for $a<b$:
  \begin{equation}
    \label{eq:LWW-Hypergraph-Lace}
    \sum_{G\in \IH^{c}\cb{a,b}}\prod_{(J,X)\in G}\hyp{J}{X}  =
    \mathop{\prod_{a\leq j\leq b}}_{X\in \cc X}(1+\hyp{j}{X})
    \sum_{L\in \cc L\cb{a,b}} \prod_{st\in L}w_{\star}(st)\!\!\!\!\!\!\!
    \mathop{\prod_{(J^{\prime},X^{\prime})\colon}}_{\edgespan
      (J^{\prime},X^{\prime})\in \cc
      C(L)}(1+\hyp{J^{\prime}}{X^{\prime}}).
  \end{equation}
  The left-hand sum is over all connected hypergraphs on $\cb{a,b}$,
  while the right-hand sum is over labelled laces.
\end{lemma}
\begin{proof}
  Apply~\Cref{lem:LWW-Lace-Prescription} with the weight $w_{\star}$,
  and take the product of this equation with the first term on the
  right-hand side of~\eqref{eq:LWW-Hypergraph-Lace}:
  \begin{equation}
    \mathop{\prod_{a\leq j\leq b}}_{X\in \cc
      X}(1+\hyp{j}{X})\!\!\! \sum_{\Gamma\in \cgraphs\cb{a,b}}\prod_{st\in
      \Gamma}w_{\star}(st) = \mathop{\prod_{a\leq j\leq b}}_{X\in \cc X}(1+\hyp{j}{X})
    \!\!\!\sum_{L\in\cc L\cb{a,b}} \prod_{st\in L}
    w_{\star(st)}\!\!\!\!\!\!\! \prod_{s^{\prime}t^{\prime}\in\cc
      C(L)}(1+w_{\star}(s^{\prime}t^{\prime})).
  \end{equation}
  Expanding the product over connected labelled graphs with weight
  $w_{\star}$ gives the left-hand side
  of~\eqref{eq:LWW-Hypergraph-Lace} as hyperedges of the form
  $(\{j\},X)$ play no role in connectivity, and for each $st$ the
  weight $w_{\star}$ is a sum of the possible collections of
  hyperedges whose span is $st$. Similarly, $1+w_{\star}(ij)$ for
  $ij\in \cc C(L)$ can be written in the product form used above,
  giving the right-hand side of~\eqref{eq:LWW-Hypergraph-Lace}.
\end{proof}

The next definition and lemma simplifies the sum over laces
in~\eqref{eq:LWW-Hypergraph-Lace} by resumming the contributions to
the product over $st\in L$.

\begin{definition}
  \label{def:LWW-Walk-I2PF}
  For $0\leq s<t$ define $I^{\omega}_{\cc X}(s,t) = 1$ if
  $\omega_{s}=\omega_{t}$, and if $\omega_{s}\neq
  \omega_{t}$ define
  \begin{equation}
    \label{eq:LWW-Walk-I2PF}
    I^{\omega}_{\cc X}(s,t) =
      1-\prod_{X\in\cc X}\ob{1 -
      \frac{\alpha_{X}}{1+\alpha_{X}}\indicatorthat{\omega_{s}\in \ell(X)}
      \indicatorthat{\omega_{t}\in \ell(X)} \indicatorthat{\ell(X)\cap
        \range{\omega\ob{s,t}}=\emptyset}}.
  \end{equation}
\end{definition}

\begin{lemma}
  \label{lem:LWW-Span-Resummation}
    Let $st$ be an edge. Then
  \begin{equation}
    \label{eq:LWW-Span-Resummation-1}
    w_{\star}^{\omega}(st,\mathrm{spacelike}) +
    w_{\star}^{\omega}(st,\mathrm{timelike}) 
    = -I^{\omega}_{\cc X}(s,t)
  \end{equation}
\end{lemma}
\begin{proof}
  The case $\omega_{s}=\omega_{t}$ corresponds to the timelike
  edge. Consider the spacelike term.  As any non-empty collection of
  spacelike hyperedges $\{(J_{i},X_{i})\}$ may be chosen in
  \Cref{eq:LWW-Pushforward-Weight-2} the equation can be rewritten as
  \begin{equation}
    w_{\star}^{\omega}(st,\mathrm{spacelike}) =
    (1-\indicatorthat{\omega_{s}=\omega_{t}})\cb{
      \mathop{\prod_{(J,X)\colon}}_{\edgespan(J,X)=st}(1+\hyp{J}{X}(\omega)) - 1}.
  \end{equation}
  A hyperedge with span $st$ and second element $X$ is equivalent to a
  possibly empty subset $J$ of $\ob{s,t}$. Using $\hyp{J\cup\{ab\}}{X} =
  \indicatorthat{\omega_{a}\in\ell(X)}
  \indicatorthat{\omega_{b}\in\ell(X)} \hyp{J}{X}$ gives
  \begin{equation}
    w_{\star}^{\omega}(st,\mathrm{spacelike}) =
    \indicatorthat{\omega_{s}\neq\omega_{t}}\ob{\prod_{X\in\cc
      X}\prod_{J\subset \ob{s,t}}\ob{1+\indicatorthat{\omega_{s}\in
      \ell(X)}\indicatorthat{\omega_{t}\in \ell(X)} \hyp{J}{X}(\omega)} - 1},
  \end{equation}
  where we recall that $\hyp{\emptyset}{X}(\omega) =
  -\alpha_{X}(1+\alpha_{X})^{-1}$.  Putting the condition that
  $\omega_{s}$ and $\omega_{t}$ are in $\ell(X)$ into the product,
  separating the case $J=\emptyset$, and then
  applying~\Cref{lem:LWW-Hypergraph-Representation} yields
  \begin{align}
    w_{\star}^{\omega}(st,\mathrm{spacelike}) &=
    \indicatorthat{\omega_{s}\neq\omega_{t}}\ob{\mathop{\prod_{X\in\cc
        X\colon}}_{\omega_{s},\omega_{t}\in\ell(X)}\cb{ (1-\frac{\alpha_{X}}{1+\alpha_{X}})
      \mathop{\prod_{J\subset \ob{s,t}}}_{\abs{J}\geq 1} (1+\hyp{J}{X}(\omega))} - 1} \\
      &=\indicatorthat{\omega_{s}\neq\omega_{t}}\ob{
      \mathop{\prod_{X\in\cc X\colon}}_{\omega_{s},\omega_{t}\in
        \ell(X)}(1+\alpha_{X})^{-\indicatorthat{
          \range{\omega\ob{s,t}}\cap\ell(X)=0}} - 1},
  \end{align}
  which is the second half of~\eqref{eq:LWW-Walk-I2PF}.
\end{proof}

\subsection{The Lace Expansion Equation}
\label{sec:LWW-X-Gas-Expansion}

This section shows how the recursion for the interaction expressed in
\Cref{lem:LWW-Hypergraph-Recursion,lem:LWW-Hypergraph-Lace-Lift}
translates into a recursion for the $c_{n}$. By summing the resulting
recursion over $n$ the desired lace expansion is obtained.

\subsubsection{Lace Expansion Equation}
\label{sec:LWW-Expansion-Equation}

For $m\geq 2$ define $\pi^{(N)}_{m}(x)$ to be
\begin{equation}
  \label{eq:LWW-Pi-Definition}
  z^{m} \alpha_{0}^{-1}
  \mathop{\sum_{\omega\colon 0\to x}}_{\abs{\omega}=m} 
  \sum_{L\in \cc L^{(N)}\cb{0,m}}
  \ob{\prod_{st\in L} I^{\omega}_{\cc X}(s,t)}
  \prod_{\edgespan (J,X)\in\cc C(L)}(1+\hyp{J}{X}(\omega)) \mathop{\prod_{a\leq j\leq
      b}}_{X^{\prime}\in \cc X}(1+\hyp{j}{X^{\prime}}(\omega)),
\end{equation}
where $\cc L^{(N)}\cb{0,m}$ is the set of laces with $N$ edges on the
interval $\cb{0,m}$. Let $\pi_{m}$ denote $\sum_{N\geq 1} (-1)^{N}
\pi_{m}^{(N)}$. Define $c_{m}=0$ for
$m<0$. \Cref{eq:LWW-Hypergraph-Interaction-1} combined with
\Cref{lem:LWW-Hypergraph-Recursion,lem:LWW-Hypergraph-Lace-Lift} imply
\begin{equation}
  \label{eq:LWW-Lace-1}
  z^{n}c_{n}(0,x) =
  \begin{cases}
    z\alpha\sum_{y\sim 0}z^{n-1}c_{n-1}(y,x) + \sum_{j\geq
      2}\sum_{y}\pi_{j}(y)z^{n-j}c_{n-j}(y,x) & n \geq 1 \\
    \alpha_{0}\delta_{0,x} & n=0.
  \end{cases}
\end{equation}
Let $G_{z}(x) = \sum_{n}z^{n}c_{n}(0,x)$. Summing~\eqref{eq:LWW-Lace-1} over
$n$, using the translation invariance of $G_{z}(x)$, and taking the
Fourier transform yields
\begin{equation}
  \label{eq:LWW-Lace-2}
  \hat G_{z}(k) = \alpha_{0} + \alpha z\abs{\Omega}\hat D(k) \hat
  G_{z}(k) + \hat \Pi_{z}(k) \hat G_{z}(k),
\end{equation}
where $\Pi_{z}(x) = \sum_{m\geq 2}\pi_{m}(x)$.

The next two sections give expressions for $\pi_{m}^{(N)}(x)$ in terms
of the quantities $\alpha_{X}$.

\subsubsection{Walk Representation of $\pi^{(N)}_{m}(x)$ for $N=1$}
\label{sec:LWW-X-Gas-pi-1}

If $N=1$ the lace consists of a single edge $0m$. If $x=0$ then
$\omega_{0}=\omega_{m}$, $I^{\omega}_{\cc X}(0,m)=1$, and
\begin{equation}
  \label{eq:LWW-X-Gas-pi-1-timelike}
  \pi^{(1)}_{m}(0) = z^{m}\alpha_{0}^{-1} \mathop{\sum_{\omega\colon 0
      \to 0}}_{\abs{\omega}=m} \indicatorthat{\omega\in\SAP}
  \prod_{X}(1+\alpha_{X})^{\indicatorthat{\range{\omega}\cap \ell(X) \neq\emptyset}}.
\end{equation}
If $x\neq 0$ the set of incompatible hyperedges are those that contain
both $0$ and $m$. Let $m_{1} = m-1$. \Cref{cor:LWW-Remainder} implies
that for $\omega\in \SAW$
\begin{equation}
  \label{eq:Edit-1}
  \prod_{(J,X)\in \cc C(0m)}(1+\hyp{J}{X}(\omega)) =
  \prod_{X\in \cc X} (1+\alpha_{X})^{ \indicatorthat{ \range{\omega}\cap
      \ell(X)\neq\emptyset}  +\indicator_{A}}
\end{equation}
where
\begin{equation}
  \label{eq:Edit-2}
  \indicator_{A} = \indicatorthat{\omega_{0}\in \ell(X)}
    \indicatorthat{\omega_{m}\in \ell(X)} \indicatorthat{
      \range{\omega\cb{1,m_{1}}} \cap \ell(X) = \emptyset},
\end{equation}
while if $\omega$ is not self-avoiding the right-hand side of
\Cref{eq:Edit-1} is zero. To see these two claims, use
\Cref{cor:LWW-Remainder} to compute the products over hyperedges
$(J,X)$ with (i) $J\subset \cb{1,m_{1}}$, (ii) $J\subset \cb{1,m}$
with $m\in J$, and (iii) $J\subset \cb{0,m_{1}}$ with $0\in J$. The
product over compatible hyperedges is the product of these terms. The
definition of $I^{\omega}_{\cc X}(0,m)$ when $\omega_{m}=x\neq 0$ then
gives a formula for $\pi^{(1)}_{m}(x)$:
\begin{align}
  \label{eq:LWW-X-Pi-1}
  \pi^{(1)}_{m}(x) = z^{m}\alpha_{0}^{-1}
  &\mathop{\sum_{\omega\colon 0 \to x}}_{\abs{\omega}=m}
  \indicatorthat{\omega\in \SAW}
  \prod_{X}(1+\alpha_{X})^{\indicatorthat{\range{\omega} \cap
    \ell(X)\neq \emptyset}} \\
  &\ob{ \prod_{X\in \cc X}(1+\alpha_{X})^{\indicatorthat{\omega_{0}\in \ell(X)}
  \indicatorthat{\omega_{m}\in \ell(X)}
  \indicatorthat{\range{\omega\cb{1,m_{1}}} 
  \cap \ell(X) = \emptyset}} - 1}.
\end{align}

\subsubsection{Walk Representation of $\pi^{(N)}_{m}(x)$ for $N\geq 2$}
\label{sec:LWW-X-Gas-pi-N}

For $N\geq 2$ the central observation is that the edges of a lace on
the discrete interval $\cb{a,b}$ divides the interval $\cb{a,b}$ into
$2N-1$ subintervals, see \Cref{fig:LWW-Lace}.
\begin{definition}
  \label{def:LWW-Valid}
  Let $m\in \N$. A vector $\vec m$ with components $m_{1}, \dots,
  m_{2N-1}$ is called \emph{valid} if (i) $m_{1}\geq 1$, $m_{2N-1}\geq
  1$, and $m_{2j}\geq 1$ for $1\leq j \leq N-1$, (ii) $m_{2j+1}\geq 0$
  for $1\leq j\leq N-1$, and (iii) $\sum m_{i} = m$.
\end{definition}
The lengths of the subintervals determined by a lace form valid vector
$\vec m$. The restrictions on which $m_{i}$ are strictly positive
arise from the definition of connectedness,
see~\cite[Section~3.3]{Slade2006} for more details. The subintervals
are given by
\begin{equation}
  \label{eq:LWW-UB-N.1}
  \bar I_{1} = \cb{0,m_{1}}, \quad \bar I_{2} = \cb{m_{1},
    m_{1}+m_{2}},\dots,
  \bar I_{2N-1} = \cb{m_{1} + \dots m_{2N-2}, m_{1} + \dots m_{2N-1}}.
\end{equation}
To each interval $\bar I_{k}$ associate a walk $\omega^{(k)}$, e.g.\
$\omega^{(2)} = (\omega_{m_{1}}, \omega_{m_{1}+1}, \dots,
\omega_{m_{1}+m_{2}})$. The walks $\omega^{(k)}$ interact with one
another through the compatible edges.

To the $k^{\mathrm{th}}$ interval associate (i) all hyperedges whose
span is contained in $\bar I_{k}$ and (ii) all compatible hyperedges
$(J,X)$ such that $\edgespan(J,X)$ is not contained in $\bar I_{k}$ with $\max
J\in\bar I_{k}$ \emph{and} $\max J \neq \max \bar I_{k}$.

For the subinterval $2N-1$ omit the last condition. That is, if a
hyperedge has $\max J=m$ associate this edge to $\bar
I_{2N-1}$. Subintervals $\bar I_{k}$ for $k<2N-1$ are missing
hyperedges of the form $(\max \bar I_{k},X)$. Including them, and
dividing by $\alpha_{0}$ to correct for this, shows the weight
associated to the interval $\bar I_{k}$ is
\begin{equation}
  \label{eq:LWW-UB-N.2}
  \alpha_{0}^{-1}\mathop{\prod_{(J,X)}}_{J\subset \bar I_{K}}(1+\hyp{J}{X})
  \prod_{\edgespan(J^{\prime},X^{\prime})\in \cc C_{k}}(1+\hyp{J^{\prime}}{X^{\prime}}),
\end{equation}
where the factor of $\alpha_{0}^{-1}$ for $k=2N-1$ comes from the
prefactor $\alpha_{0}^{-1}$ in the definition of $\pi^{(N)}_{m}$.

The last two factors can be evaluated together. A compatible
hyperedge must have its minimum index be at least the second index of
either $\omega^{(k-2)}$ or $\omega^{(k-3)}$. Suppose the first case;
the second is similar. \Cref{cor:LWW-Remainder} implies the product
in~\eqref{eq:LWW-UB-N.2} forces $\omega^{(k)}$ to be self-avoiding,
$\omega^{(k)}$ to avoid $\omega^{(k-1)}$ and
$\omega^{(k-2)}\cb{1\splice}$, and assigns
$\omega^{(k)}$ the weight
\begin{equation}
  \label{eq:LWW-UB-N.3}
  \alpha_{0}^{-1} \prod_{X\in \cc X}
  (1+\alpha_{X})^{\indicatorthat{ \range{\omega^{(k)}}\cap
      \ell(X) \neq \emptyset} \indicatorthat{
      \range{\omega^{(k-2)}\circ\omega^{(k-1)}\cb{1\splice}}\cap
      \ell(X) = \emptyset}},
\end{equation}

\tikzset{
  dot/.style={
    circle, inner sep=0pt, 
    minimum size=1.5mm, fill=black
 }
}

\begin{figure}[h]
  \centering
  \beginpgfgraphicnamed{fig3}
  \begin{tikzpicture}
    \node[dot] (v0) at (0,0) {};
    \node[dot] (v1) at (2,0) {};
    \node[dot] (v0p) at (2,2) {};
    \node[dot] (v2) at (4,2) {};
    \node[dot] (v1p) at (4,0) {};
    \node[dot] (v3) at (6,0) {};
    \node[dot] (v2p) at (6,2) {};
    \node[dot] (v4) at (8,2) {};
    \node[dot] (v3p) at (8,0) {};
    \node[dot] (v5) at (10,0) {};

    \node at (v1) [below] {$x_{1}$};
    \node at (v0p) [above] {${x_{0}^{\prime}}$};
    \node at (v2) [above] {$x_{2}$};
    \node at (v1p) [below] {$x_{1}^{\prime}$};
    \node at (v3) [below] {$x_{3}$};
    \node at (v2p) [above] {$x_{2}^{\prime}$};
    \node at (v4) [above] {$x_{4}$};
    \node at (v3p) [below] {$x_{3}^{\prime}$};

    \node at (v0) [below left] {$0$};
    \node at (v5) [below right] {$x$};

    \node(e1) at (1,0) {};
    \node (e2) at (2,1) {};
    \node (e3) at (3,2) {};
    \node (e4) at (4,1) {};
    \node (e5) at (5,0) {};
    \node (e6) at (6,1) {};
    \node (e7) at (7,2) {};
    \node (e8) at (8,1) {};
    \node (e9) at (9,0) {};

    \draw[black] (v0) -- (v1) -- (v0p) -- (v2) -- (v1p) -- (v3) --
    (v2p) -- (v4) -- (v3p) -- (v5);

    \draw[black, decorate, decoration={zigzag,segment length = 5,
      amplitude=.8}] (v0) to[out=30, in=240 ] (v0p); 
    \draw[black, decorate, decoration={zigzag,segment length = 5,
      amplitude=.8}] (v0) to[out=60, in=210] (v0p); 
    \draw[black, decorate, decoration={zigzag,segment length = 5,
      amplitude=.8}] (v1) to[out=15, in=165] (v1p); 
    \draw[black, decorate, decoration={zigzag,segment length = 5,
      amplitude=.8}] (v1) to[out=-15, in=195] (v1p); 
    \draw[black, decorate, decoration={zigzag,segment length = 5,
      amplitude=.8}] (v2) to[out=15, in=165] (v2p);
    \draw[black, decorate, decoration={zigzag,segment length = 5,
      amplitude=.8}] (v2) to[out=-15, in=195] (v2p);
    \draw[black, decorate, decoration={zigzag,segment length = 5,
      amplitude=.8}] (v3) to[out=15, in=165] (v3p);
    \draw[black, decorate, decoration={zigzag,segment length = 5,
      amplitude=.8}] (v3) to[out=-15, in=195] (v3p);
    \draw[black, decorate, decoration={zigzag,segment length = 5,
      amplitude=.8}] (v4) to[out=330,in=120] (v5);
   \draw[black, decorate, decoration={zigzag,segment length = 5,
      amplitude=.8}] (v4) to[out=300,in=150] (v5);

    \node at (e1) [below] {$m_{1}$};
    \node at (e2) [right] {$m_{2}$};
    \node at (e3) [above] {$m_{3}$};
    \node at (e4) [left] {$m_{4}$};
    \node at (e5) [below] {$m_{5}$};
    \node at (e6) [right] {$m_{6}$};
    \node at (e7) [above] {$m_{7}$};
    \node at (e8) [left] {$m_{8}$};
    \node at (e9) [below] {$m_{9}$};
  \end{tikzpicture}
  \endpgfgraphicnamed
  \caption[The diagrammatic representation of $\pi^{(5)}_{m}$]{The
    diagrammatic representation of $\pi^{(5)}_{m}(x)$ with $m = \sum
    m_{i}$. The vertices $x_{1}, \dots, x_{4}$ and
    $x_{0}^{\prime},\dots, x_{3}^{\prime}$ are summed over.}
  \label{fig:LWW-Diagrammatic-1}
\end{figure}
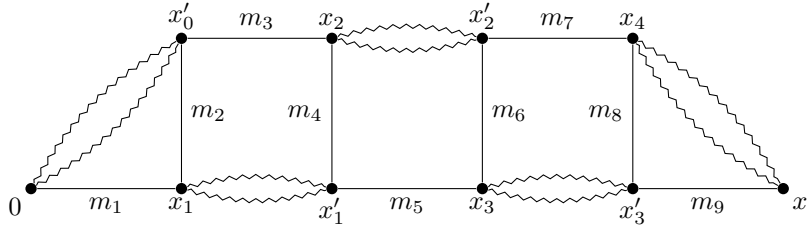

As an explicit formula for $\pi^{(N)}_{m}$ detailing the constraints
is unwieldy, let us explain the formula with a brief discussion of the
diagrammatic representation of $\pi^{(N)}_{m}$ in
\Cref{fig:LWW-Diagrammatic-1}. The solid lines represent a subdivision
of a walk $\omega$ into subwalks; these subwalks are subject to
self-avoidance constraints detailed below. Pairs of zigzag lines
represent $I^{\omega}_{\cc X}(t_{i},t_{i+1})$, where $t_{i}$ is the
time $x_{i}$ occurs in the walk $\omega = \omega^{(1)} \circ \dots
\circ \omega^{(2N-1)}$. Each walk $\omega^{(i)}$ has length $m_{i}$
and is self-avoiding. Further, each walk $\omega^{(i)}$ avoids some of
the previous walks $\omega^{(j)}$ for $j<i$, excluding the endpoint of
$\omega^{(i-1)}$. To be precise, $\omega^{(2)}$ avoids $\omega^{(1)}$,
$\omega^{(2k+1)}$ avoids $\omega^{(2k-1)}$ and $\omega^{(2k)}$, and
$\omega^{(2k+2)}$ avoids $\omega^{(2k-1)}$, $\omega^{(2k)}$, and
$\omega^{(2k+1)}$. The walk $\omega^{(j)}$ is weighted by those
closed walks in $\cc X$ that do not intersect the $\omega^{(j)}$ which
$\omega^{(i)}$ is forbidden to intersect; for example, in
\Cref{eq:LWW-UB-N.3} the walk $\omega^{(k)}$ is being weighted by all
closed walks that do not intersect $\omega^{(k-1)}$ or
$\omega^{(k-2)}\cb{1\splice}$.

\begin{remark}
  \label{rem:LWW-X-Gas-Repulsion}
  Since $\alpha_{X}\geq 0$ for each $X$, ignoring the constraint that some
  closed walks do not weight a subwalk gives an upper bound for the
  weight on the subwalks $\omega^{(i)}$. Ignoring the constraint of
  avoiding $\omega^{(j)}$ for some $j<i$ gives a further upper bound
  on $\pi^{(N)}_{m}(x)$.
\end{remark}

\section{Concrete Expressions for the Lace Expansion for
  $\lambda$-LWW}
\label{sec:LWW-LM-Formulation}

Quantities such as $\alpha_{0}(\cc X)$ and $I_{\cc X}^{\omega}$ will
be written as $\alpha_{0}(\lambda,z)$, $I_{\lambda,z}^{\omega}$ and
similarly in what follows. The arguments $\lambda$ and $z$ may be
omitted to lighten the notation. As emphasized
earlier, $\lambda\geq 0$ and $z\geq 0$ implies
$w_{\lambda,z}(\omega)\geq 0$, and hence $\alpha_{\omega}\geq 0$. In
particular, by \Cref{rem:LWW-X-Gas-Repulsion} we can obtain upper
bounds by ignoring constraints.

\begin{definition}
  \label{def:LWW-2-PT}
  The \emph{two point function} $G_{\lambda,z}(x,y)$ for $\lambda$-LWW
  is defined by
  \begin{equation}
    \label{eq:LWW-WLAW-2PT}
    G_{\lambda,z}(x,y) =  \sum_{\omega\colon x \to y} w_{\lambda,z}(\omega).
  \end{equation}
\end{definition}

By \Cref{thm:LWW-LM-Rep} and~\Cref{prop:LWW-2PT-Equivalence} the
two-point function $G_{\lambda,z}$ of $\lambda$-LWW is given by the
two-point function of self-avoiding walks weighted as in
\Cref{eq:LWW-Intro-X-Gas}. For future reference we state a
reformulation of~\eqref{eq:LWW-Lace-2} as a proposition.

\begin{proposition}
  \label{prop:LWW-Lace-Expansion-WLAW}
  \begin{equation}
    \label{eq:LWW-Lace-Expansion-WLAW}
    \hat G_{\lambda,z}(k) = \frac{\alpha_{0}(\lambda,z)}
    {1 - \alpha(\lambda,z)\abs{\Omega}\hat D(k) - \hat\Pi_{\lambda,z}(k)}.
  \end{equation}
\end{proposition}

To analyze the recursion~\eqref{eq:LWW-Lace-Expansion-WLAW} it will be
convenient to rewrite the equation in terms of $w_{\lambda,z}$ and the
loop measure $\mu_{\lambda,z}$. The quantities $\alpha_{0}(\lambda,z)$ and
$\alpha(\lambda,z)$ can be expressed as, for $y\sim 0 \in \Z^{d}$,
\begin{equation}
  \label{eq:LWW-Specalization-4}
  \alpha_{0}(\lambda,z) = \exp \ob{\mu_{\lambda,z}(0)}, \qquad
  \alpha(\lambda,z) = \exp \ob{\mu_{\lambda,z}(0;y)}.
\end{equation}
Note that $\alpha_{0}\geq \alpha \geq 1$. Let $I^{\omega}_{\lambda,z}
= I^{\omega}_{\cc X}$. $I^{\omega}_{\lambda,z}(a,b)$ can be written in
a loop measure like way:
\begin{equation}
  \label{eq:LWW-Specialization-I2PF}
  I^{\omega}_{\lambda,z}(a,b) = \indicatorthat{\omega_{a} = \omega_{b}} +
  \indicatorthat{\omega_{a}\neq\omega_{b}} \ob{1 - e^{-
    \mu_{\lambda,z}( \omega_{a},\omega_{b}; \range{\omega\ob{a,b}})}},
\end{equation}
where
\begin{equation}
  \label{eq:LWW-Loop-Measure-Generalized}
  \mu_{\lambda,z}(A,B;C) = \sum_{x} \mathop{\sum_{\omega\colon x\to
      x}}_{\abs{\omega} \geq 1}
  \frac{w_{\lambda,z}(\omega)}{\abs{\omega}}
  \indicatorthat{\range{\omega}\cap A \neq\emptyset}
  \indicatorthat{\range{\omega} \cap B \neq \emptyset}
  \indicatorthat{\range{\omega} \cap C = \emptyset}.
\end{equation}
As with the loop measure, define $\mu_{\lambda,z}(A,B) =
\mu_{\lambda,z}(A,B;\emptyset)$. The effect of this more complicated
object is to require that both an element from $A$ \emph{and}
$B$ are in the range of the walk.

\section{Convergence of the Lace Expansion I. Preliminaries}
\label{sec:LWW-Convergence}

This section establishes the basic facts used to prove the convergence
of the lace expansion. The strategy is that of~\cite{Slade2006},
suitably adapted and modified for $\lambda$-LWW. An important role is
played by the function $H_{\lambda,z}$ in the next definition.

\begin{definition}
  \label{def:LWW-Reduced-2PT-Function}
  The \emph{reduced two point function} $H_{\lambda,z}(x,y)$ is
  defined by
  \begin{equation}
    \label{eq:LWW-WLAW-2PT-H}
    H_{\lambda,z}(x,y) = (1-\delta_{x,y})G_{\lambda,z}(x,y).
  \end{equation}
\end{definition}
A useful fact that will be used repeatedly is that
\begin{equation}
  \label{eq:LWW-G-H-Relation}
  G_{\lambda,z}(x,y) = \delta_{x,y}\alpha_{0}(\lambda,z) +
  H_{\lambda,z}(x,y).
\end{equation}
The two-point functions $G_{\lambda,z}$ and $H_{\lambda,z}$ inherit
the isometry invariance of the weight $w_{\lambda,z}$. By translation
invariance $G_{\lambda,z}(x,y) = G_{\lambda,z}(0,y-x)$; it will be
convenient to write $G_{\lambda,z}(x)$ for $G_{\lambda,z}(0,x)$.

\subsection{Random Walk Quantities and Bounds}
\label{sec:LWW-SRW-Versions}

\begin{definition}
  The \emph{random walk $2$-point function} $C_{z}(x)$ and its Fourier
  transform $\hat C_{z}(k)$ are given by
  \begin{equation}
    \label{eq:LWW-SRW-FT}
    C_{z}(x) = \sum_{\omega\colon x\to x}z^{\abs{\omega}}, \qquad 
    \hat C_{z}(k) = \frac{1}{1-z\abs{\Omega}\hat D(k)}.
  \end{equation}
\end{definition}

The following facts about the random walk two-point function will be
useful. For notational clarity, let $\beta$ be a quantity that is
$O(\abs{\Omega}^{-1})$. $\beta$ is to be thought of as being a small
parameter. 

\begin{lemma}[name = Lemma~5.5 of~\cite{MadrasSlade2013}]
  \label{lem:LWW-SL5.5}
  Assume $d>4$. Then for $0\leq z \leq \abs{\Omega}^{-1}$
  \begin{align}
    \label{eq:LWW-SL5.5.1}
    \sup_{x}D(x) &\leq \beta \\
    \label{eq:LWW-SL5.5.2}
    \norm{C_{z}}_{2}^{2} &\leq 1 + c\beta\\
    \label{eq:LWW-SL5.5.3}
    \norm{(1-\cos(k\cdot x))C_{z}(x)}_{\infty} &\leq
    5(1+c\beta)(1-\hat D(k))
  \end{align}
\end{lemma}

\begin{proposition}
  \label{prop:LWW-Small}
  Let $r\in \N$. There is a constant $K$ independent of $d$ such that
  for $d>2r$.
  \begin{equation}
    \label{eq:LWW-Small}
    \int_{\FS^{d}} \ob{\frac{1}{1-\hat D(k)}}^{r}\,
    \frac{d^{d}k}{\FSint} \leq 1+c\beta.
  \end{equation}
\end{proposition}
\begin{proof}
  This follows by the argument used in the proof
  of~\cite[Lemma~A.3]{MadrasSlade2013}.
\end{proof}

\subsection{Convergence Strategy and Basic Bounds}
\label{sec:LWW-Convergence-Proof}

The proof of convergence is based on comparing the behaviour of simple
random walk and $\lambda$-LWW. Define $p(z)$
by
\begin{equation}
  \label{eq:LWW-Rescaling-Definition}
  \frac{\hat G_{\lambda,z}(0)}{\alpha_{0}(\lambda,z)} =
  \frac{1}{1-p(z)\abs{\Omega}} = \hat C_{p(z)}(0).
\end{equation}
Roughly speaking, the intuition is that $\lambda$-LWW should behave
like simple random walk. The definition of $p(z)$ serves to determine
the activity of the simple random walk that matches $\lambda$-LWW with
activity $z$. The following bootstrap lemma is what enables
conclusions to be drawn for $z<z_{c}(\lambda)$.
\begin{lemma}[name = {\cite[Lemma 5.9]{Slade2006}}]
  \label{lem:LWW-Bootstrap}
  Let $a<b$, let $f$ be a continuous function on the interval
  $\co{z_{1},z_{2}}$, and assume that $f(z_{1})\leq a$. Suppose for
  each $z\in \ob{z_{1},z_{2}}$ that $f(z)\leq b$ implies
  $f(z)\leq a$. Then $f(z)\leq a$ for all $z\in \co{z_{1},z_{2}}$.
\end{lemma}
To describe the function $f$ used in applying \Cref{lem:LWW-Bootstrap}
some definitions are needed.

\begin{definition}
  \label{def:LWW-Delta-k}
  Define $\Delta_{k}\hat A(\ell)$ by
  \begin{equation}
    -\frac{1}{2} \Delta_{k} \hat A(\ell) = \hat A(\ell) - \frac{1}{2}
    \ob{ \hat A(\ell +k) + \hat A(\ell -k)},
  \end{equation}
  and define
  \begin{multline}
    \nonumber
    U_{p(z)}(k,\ell) = 16\hat C_{p(z)}^{-1}(k) \bigg( \hat C_{p(z)}(\ell-k)
    \hat C_{p(z)}(\ell)  + \hat C_{p(z)}(\ell + k) \hat C_{p(z)}(\ell)\\+ \hat
    C_{p(z)}(\ell - k)\hat C_{p(z)}(\ell + k)\bigg).
  \end{multline}
\end{definition}
The quantity $U_{p(z)}$ is a convenient upper bound for $\frac{1}{2}
\abs{\Delta_{k}\hat C_{p(z)}(\ell)}$: this can be seen
by~\cite[Lemma~5.7]{Slade2006}. Define $f(z) = \max \{f_{1}(z),
f_{2}(z), f_{3}(z)\}$, where
\begin{equation}
  \label{eq:LWW-Comparison-Functions}
  f_{1}(z) = z\alpha(\lambda,z)\abs{\Omega}, \quad f_{2}(z) =
  \sup_{k\in\cb{-\pi,\pi}^{d}} \frac{ \abs{\hat G_{\lambda,z}(k)}
  }{\hat C_{p(z)}(k)}, \quad f_{3}(z) =
  \sup_{k,\ell\in\cb{-\pi,\pi}^{d}} \frac{ \Delta_{k}\hat
    G_{\lambda,z}(\ell)}{U_{p(z)}(k,\ell)}.
\end{equation}

The next lemma will be useful for estimating $G_{\lambda,z}$.
\begin{lemma}
  \label{lem:LWW-One-Step-SM}
  Assume $y\neq x$. The following inequality holds:
  \begin{equation}
    \label{eq:LWW-One-Step-SM}
    G_{\lambda,z}(x,y) \leq
    z\alpha(\lambda,z)\abs{\Omega}\sum_{u}D(u)G_{\lambda,z}(u,y).
  \end{equation}
\end{lemma}
\begin{proof}
  This can be proven using the loop measure representation. For
  $\eta$ a walk beginning at $u\sim 0$, let $0\eta = (0,u)\circ\eta$.
  \begin{align}
    G_{\lambda,z}(0,y) &= \sum_{\eta\colon 0\to y}
    \indicatorthat{\eta\in \SAW}
    z^{\abs{\eta}} \exp \ob{ \mu_{\lambda,z}(\range{\eta})} \\
    &= \sum_{u\sim 0} \sum_{\eta\colon u\to y}
    \indicatorthat{0\eta\in \SAW}
    z \exp \ob{ \mu_{\lambda,z}(0; \range{\eta})} z^{\abs{\eta}} \exp
    \ob{ \mu_{\lambda,z}(\range{\eta})} \\
    &\leq z\alpha(\lambda,z)\abs{\Omega} \sum_{u}D(u) \sum_{\eta\colon
        u\to y} \indicatorthat{\eta\in \SAW} z^{\abs{\eta}} \exp
    \ob{ \mu_{\lambda,z}(\range{\eta})} \\
    &= z\alpha(\lambda,z)\abs{\Omega} \sum_{u}D(u)G_{\lambda,z}(u,y),
  \end{align}
  The inequality follows as (a) \Cref{prop:LWW-LM-Properties} implies
  $\mu_{\lambda,z}(0;\range{\eta})$ is bounded above by
  $\mu_{\lambda,z}(0;u) = \alpha_{0}$ and (b)
  $\indicatorthat{0\eta\in\SAW}$ is bounded above by
  $\indicatorthat{\eta\in\SAW}$.
\end{proof}

\begin{proposition}
  \label{prop:LWW-SL5.10}
  Assume $d>4$. Fix $z\in \ob{0,z_{c}}$ and assume $f(z)\leq K$. Then
  there is a constant $c_{K}$ independent of $z$ and $d$ such that
  \begin{align}
    \label{eq:LWW-SL5.10.1}
    \norm{(1-\cos(k\cdot x))H_{\lambda,z}}_{\infty} &\leq
    c_{K}(1+\beta) \hat C_{p(z)}(k)^{-1}, \\
    \label{eq:LWW-SL5.10.2}
    \norm{H_{\lambda,z}}_{2}^{2}  &\leq c_{K}\beta \\
    \label{eq:LWW-SL5.10.3}
    \norm{H_{\lambda,z}}_{\infty} & \leq c_{K}\beta.
  \end{align}
\end{proposition}
\begin{proof}
  The general fact that $\norm{g}_{\infty}\leq \norm{\hat g}_{1}$ and
  the identity
  \begin{equation}
    \sum_{x}\cos (k\cdot x) f(x) e^{i\ell\cdot x} = \frac{1}{2} \ob{\hat
    f(\ell+k) + \hat f(\ell-k)}
  \end{equation}
  imply that
  \begin{equation}
    \norm{(1-\cos(k\cdot x))H_{\lambda,z}(x)}_{\infty} =
    \norm{(1-\cos(k\cdot x))G_{\lambda,z}(x)}_{\infty}
    \leq \frac{1}{2} \norm{\Delta_{k}\hat G_{\lambda,z}(\ell)}_{1}.
  \end{equation}
  The definition of $U$, the fact that $f_{3}\leq K$, and the
  Cauchy-Schwarz inequality then imply
  \begin{equation}
    \norm{(1-\cos(k\cdot x))H_{\lambda,z}(x)}_{\infty} \leq 16K \hat
    C_{p(z)}(k)^{-1} 3 \norm{\hat C_{p(z)}}_{2}^{2},
  \end{equation}
  which yields~\eqref{eq:LWW-SL5.10.1} after
  using~\eqref{eq:LWW-SL5.5.2}.

  To estimate $\norm{H_{\lambda,z}}_{2}^{2}$ note that
  \Cref{lem:LWW-One-Step-SM} implies
  \begin{equation}
    H_{\lambda,z}(x) \leq z\alpha(\lambda,z)\abs{\Omega}
    D\ast G_{\lambda,z}(x)
  \end{equation}
  The factor $z\alpha\abs{\Omega}$ is estimated using
  $f_{1}(z)\leq K$. To estimate $D\ast G_{\lambda,z}$ use Parseval's
  identity, $f_{2}(z)\leq K$, and \Cref{lem:LWW-SL5.5}:
  \begin{equation}
    \norm{H_{\lambda,z}}_{2}^{2} \leq K^{2}
    \norm{D\ast G_{\lambda,z}}_{2}^{2} 
    \leq K^{4} \norm{\hat D \hat C_{\abs{\Omega}^{-1}}}_{2}^{2}  =
    K^{4}(\norm{\hat C_{\abs{\Omega}^{-1}}}_{2}^{2}-1)\leq
    c K^{4} \beta.
  \end{equation}

  For the last inequality use the fact that
  $\sup_{x}H_{\lambda,z}(x) = \sup_{x\neq 0}G_{\lambda,z}(x)$,
  \Cref{lem:LWW-One-Step-SM}, \Cref{eq:LWW-G-H-Relation} and then
  \Cref{lem:LWW-One-Step-SM} again. Using $f_{1}\leq K$ gives
  \begin{equation}
    H_{\lambda,z}(x) \leq K\alpha_{0}(\lambda,z)D(x) +
    K^{2}D\ast D\ast G_{\lambda,z}(x).
  \end{equation}
  A little manipulation shows that $\norm{D\ast
    D \ast G_{\lambda,z}}_{\infty}\leq \norm{\hat D^{2}\hat
    C^{2}_{p(z)}}_{1}$, so \Cref{lem:LWW-SL5.5} implies
  \begin{equation}
    \norm{D\ast D \ast G_{\lambda,z}}_{\infty} \leq c K\beta.
  \end{equation}
  \Cref{eq:LWW-SL5.5.1} implies $D(x)\leq \beta$ so it suffices to show
  $\alpha_{0}(\lambda,z)$ is bounded above. This follows from
  $f_{2}\leq K$:
  \begin{align}
    \alpha_{0} = \int_{\FS^{d}} \hat G_{\lambda,z}(k)\, \frac{d^{d}k}{\FSint}
    \leq K \int_{\FS^{d}} \hat C_{p(z)}(k)\, \frac{d^{d}k}{\FSint} \leq K
    \norm{\hat C_{\abs{\Omega}^{-1}}}_{1},
  \end{align}
  and this last integral is finite for $d\geq 3$, and decreases as the
  dimension $d$ increases.
\end{proof}

\section{Convergence of the Lace Expansion II. Diagrammatic 
  Bounds and Convergence}
\label{sec:LWW-Diagrammatic-Bounds}

To control the lace expansion it is necessary to show that
$\hat \Pi_{\lambda,z}$ is small. This is done by obtaining bounds on
norms of $\Pi^{(N)}_{\lambda,z}$ in terms of $H_{\lambda,z}$,
$G_{\lambda,z}$, and $I_{\lambda,z}$. These bounds are known as
\emph{diagrammatic bounds}. Coupled with \Cref{prop:LWW-SL5.10}
diagrammatic bounds are what make the hypothesis $f(z)\leq K$
powerful. 

Obtaining diagrammatic bounds requires bounding the weight of walks
constrained to have $\omega_{j}=x$ in terms of unconstrained
walks. This is best illustrated by an example. Consider obtaining a
bound for $\frac{d}{dz}G_{\lambda,z}(0,x)$. For self-avoiding walk
($\lambda=0$) this is straightforward: the Leibniz rule implies the
derivative is a sum over all self-avoiding walks from $0$ to $x$
together with a marked edge. Splitting the walk at the marked edge and
using the fact that self-avoiding walk is purely repulsive yields
\begin{equation}
  \label{eq:LWW-DB-Intro-1}
  \frac{d}{dz} G_{0,z}(0,x) \leq z^{-1}G_{0,z}\ast H_{0,z}(0,x).
\end{equation}

For $\lambda>0$ a similar argument is possible, but the weight on the
second half of the walk is not $w_{\lambda,z}$: memory of the first
half of the walk is needed to know when loops are
erased. \Cref{sec:LWW-Decompositions} derives identities for walks
that play the role of \Cref{eq:LWW-DB-Intro-1} for
$\lambda>0$. \Cref{sec:LWW-DB-Pi} uses these identities to derive the
diagrammatic bounds necessary to apply \Cref{lem:LWW-Bootstrap}.

\subsection{Decompositions for $\lambda$-LWW}
\label{sec:LWW-Decompositions}

The formulas presented in this section are the result of
tracking what happens when loop erasure is performed. The reader may
find it helpful to draw examples while reading the text.

\subsubsection{Decompositions from Loop Erasure}
\label{sec:LWW-Decomp-LE}

The loop erasure of a walk can be viewed as a last exit
decomposition: if $\omega\colon x\to y$ then the second vertex in the
loop erasure is the first vertex visited after the last visit to
$x$. Iterating this implies the next proposition.
\begin{proposition}
  \label{prop:LWW-LE-LE}
  Let $\omega$ be a walk. Define $\ell_{0}=0$, and $\ell_{k} = \sup
  \{j \mid \omega_{j}=\omega_{\ell_{k-1}} \} + 1$ for $k\in
  \N$. Suppose there are $n+1$ finite values of $\ell_{k}$ such that
  $\ell_{k}\leq \abs{\omega}$. Then
  \begin{equation}
    \label{eq:LWW-LE-LE}
    \LE(\omega) = (\omega_{\ell_{0}}, \omega_{\ell_{1}}, \dots, \omega_{\ell_{n}}).
  \end{equation}
\end{proposition}
In \Cref{prop:LWW-LE-LE} the restriction to finite values at most
$\abs{\omega}$ is due to the fact that there will be an $\ell_{k} =
\abs{\omega}+1$, and then $\ell_{k+1} = -\infty$. The loop erasure of
a walk $\omega$ induces a decomposition of $\omega$. Let $\eta
= \LE(\omega) = (\omega_{\ell_{0}}, \dots,
\omega_{\ell_{k}})$. Define, for $0\leq r< s\leq k$,
\begin{equation}
  \label{eq:LWW-LE-Pre}
  \eta^{-1}\cb{r,s} = \omega\cb{\ell_{r},\ell_{s}-1},
\end{equation}
where, recalling \Cref{prop:LWW-LE-LE},
$\ell_{k+1}=\abs{\omega}+1$. See \Cref{fig:LWW-Splitting}.

\begin{remark}
  It would be more accurate to write $\LE(\omega)^{-1}\cb{r,s}$ as the
  definition requires knowledge of the walk $\omega$ whose loop
  erasure is $\eta$. As the walk $\omega$ will be clear from context
  this will not cause any confusion.
\end{remark}

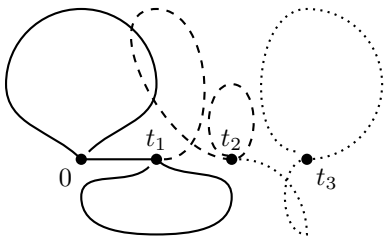
\begin{figure}[h]
  \centering
  \beginpgfgraphicnamed{fig4}
  \begin{tikzpicture}
    \draw[black,thick] (0.1,0.1) to[out=45, in=270] (1,1) to[out=90, in=0]
    (0,2) to[out=180, in=90] (-1,1) to[out=270, in=135] (0,0) to (1,0)
    to[out=-45, in=90] (2,-.5) to[out=270,in=0] (1,-1) to[out=180,in=270]
    (0,-.5) to[out=90,in=225] (.90,-.1);
    \draw[black,dashed,thick] (1.1,0) to[out=0,in=0] (1,2) to[out=180,in=180]
    (2,0) to[out=0,in=0] (2,1) to[out=180,in=180] (1.95,.05);
    \draw[black,thick, dotted] (2.1,0) to[out=0,in=90] (3,-1) to[out=180,in=180]
    (3,0) to[out=0,in=270] (4,1) to[out=90,in=0] (3,2)
    to[out=180,in=180] (2.9,0.1);
    \node[below left] at (0,0) {$0$};
    \node[above] at (1,0) {$t_{1}$};
    \node[above] at (2,0) {$t_{2}$};
    \node[below right] at (3,0) {$t_{3}$};
    \node[black,fill,inner sep=1.5pt,circle] at (0,0) {};
    \node[black,fill,inner sep=1.5pt,circle] at (1,0) {};
    \node[black,fill,inner sep=1.5pt,circle] at (2,0) {};
    \node[black,fill,inner sep=1.5pt,circle] at (3,0) {};
  \end{tikzpicture}
  \endpgfgraphicnamed
  \caption[Example of the splitting of a walk according to
  \Cref{eq:LWW-Segment-Split}]{An illustration of the definition of
    $\eta^{-1}\cb{r,s}$ and of the subdivision of a walk given by
    \Cref{eq:LWW-Segment-Split}. The initial walk
    $\eta^{-1}\cb{0,t_{1}}$, drawn with a solid line, is the preimage
    of the initial solid segment of the loop erasure of the walk. The
    dashed and dotted curves are $\eta^{-1}\cb{t_{1},t_{2}}$ and
    $\eta^{-1}\cb{t_{2},t_{3}}$ respectively. The small gaps in curves
    indicate the flow of time.}
  \label{fig:LWW-Splitting}
\end{figure}

The following extension of the notion of the concatenation of two
walks will be notationally convenient. If $\omega^{i}\colon x_{i}\to
y_{i}$ and $y_{1}\sim x_{2}$ write $\omega^{1}\diamond \omega^{2}$ for
the walk that consists of $\omega^{1}$ followed by a step from $y_{1}$
to $x_{2}$ followed by the walk $\omega^{2}$.

Fix a walk $\omega$ whose loop erasure is $k$ steps long. A sequence of
times $0=t_{0}<t_{1}<t_{2}<\dots <t_{n}=k$ induces a
decomposition of $\omega$ by using \Cref{eq:LWW-LE-Pre}:
\begin{equation}
  \label{eq:LWW-Segment-Split}
  \omega = \eta^{-1}\cb{t_{0},t_{1}} \diamond  \dots \diamond
  \eta^{-1}\cb{t_{n-1},t_{n}}.
\end{equation}
This decomposition has two notable features. First, the loop erasure
of the segments of the decomposition yield
$\eta\cb{t_{i},t_{i+1}-1}$. Second, each segment, barring perhaps the
first segment, never returns to its starting vertex. See
\Cref{fig:LWW-Splitting}.

The next definitions serve to formalize the idea that given the loop
erasure $\eta = \LE(\omega\cb{0,j})$ of a walk $\omega$ up to time
$j$, the remainder of $\omega$ has the effect of erasing some of
$\eta$, and then extending the remainder of $\eta$ to complete the
formation of $\LE(\omega)$.

\begin{definition}
  \label{def:LWW-hitting}
  Let $A\subset \Z^{d}$. The \emph{hitting time $\tau_{\omega}(A)$} of
  $A$ by $\omega$ is $\tau_{\omega}(A) = \inf \{ j\geq 0 \mid
  \omega_{j}\in A\}$.
\end{definition}

\begin{definition}
  \label{def:LWW-Shrinking}
  Let $\eta\colon x\to y$ be a self-avoiding walk, and let $\omega$ be
  a walk beginning at $y$. Let $\eta^{0} = \eta\co{0,\abs{\eta}}$. For
  $k\geq 1$ inductively define
  \begin{equation}
    \label{eq:LWW-Shrinking}
    s^{k}_{\eta}(\omega) = \tau_{\omega}(\eta^{k-1}), \qquad
    t^{k}_{\omega}(\eta) = \eta^{-1}(\omega_{s^{k}_{\eta}(\omega)}), \qquad
    \eta^{k} = \eta\co{0,t^{k}_{\omega}(\eta))}.
  \end{equation}
  The times $s^{k}_{\eta}(\omega)$ are the \emph{shrinking times of
    $\eta$ by $\omega$}.
\end{definition}
See~\Cref{fig:LWW-Shrinking} for an illustration of shrinking times.  The
walks $\eta^{k}$ in the definition are decreasing in length, and it
follows that the times $t^{k}_{\omega}(\eta)$ are decreasing in $k$.

\begin{figure}[h]
  \centering
  \beginpgfgraphicnamed{fig5}
  \begin{tikzpicture}
    \node (s1) at (2,0) {};
    \node (m12) at (3,1) {};
    \node (s2) at (4,0) {};
    \node (m23) at (5,1) {};
    \node (s23) at (5,0) {};
    \node (s3) at (6,0) {};
    \draw[black,thick,dashed] (0,0) to[out=3,in=177] (s1)
    to[out=-2,in=181] (s2)
    to[out=3,in=178] (s3);
    \draw[black, thick,->] (s3) to[out=60,in=0] (m23) to[out=180,in=90]
    (s2);
    \draw[black, thick,->] (s2) to[out=270,in=250] (s23);
    \draw[black,thick,->] (s23) to[out=70,in=270] (m23);
    \draw[black,thick,->] (m23) to[out=90,in=90] (m12) to[out=270, in=90]
    (s1);
    \node[below left] at (s2) {$\tau_{\omega}(\eta^{1})$};
    \node[below left] at (s1) {$\tau_{\omega}(\eta^{0})$};
  \end{tikzpicture}
  \endpgfgraphicnamed
  \caption[Splitting a walk at shrinking
  times]{An illustration of the shrinking times of the self-avoiding
    walk $\eta$ (dashed) by $\omega$ (solid). The gaps in $\omega$
    are to indicate the progress of time. Note that the second hitting
    time of $\eta$ is \emph{not} a shrinking time as it occurs on a
    portion of $\eta$ that is erased at the first hitting time.}
  \label{fig:LWW-Shrinking}
\end{figure}

\subsubsection{Expected Visits of $\lambda$-LWW}
\label{sec:LWW-Visits}

The next proposition gives a formula for the expected number of visits
of a closed $\lambda$-LWW to a given vertex $y$. We will first give an informal description of the formula. The number of visits by a walk $\omega$ to a vertex $y$
can be expressed as
\begin{equation}
  \label{eq:LWW-Bubble-Chain-Heuristic-1}
  \abs{ \{ j\geq 1 \mid \omega_{j}=y\}} = \sum_{j\geq 1}
  \indicatorthat{\omega_{j}=y}.
\end{equation}
Consider a walk with $\omega_{j}=y$. This naturally splits into two
pieces: the walk $\omega^{(a)}$ up to time $j$, and the walk
$\omega^{(b)}$ after time $j$. The splitting times introduced in
\Cref{sec:LWW-Decomp-LE} then splits each of $\omega^{(a)}$ and
$\omega^{(b)}$ into $k$ segments if there are $k$ splitting times. In
\Cref{prop:LWW-Bubble-Chain} the segments of $\omega^{(a)}$ are called
$\omega^{(i)}$ for $i=1, \dots, k$, and the segments of $\omega^{(b)}$
are called $\omega^{(k+i)}$ for $i=1, \dots, k$. The conditions
$A_{i}$ and $B_{i}$ are formalizations of the fact that these subwalks
arise from splitting times.

\begin{proposition}
  \label{prop:LWW-Bubble-Chain}
  Fix $x,y\in \Z^{d}$, $y\neq x$. Then
  \begin{align}
    \label{eq:LWW-Bubble-Chain}
    \sum_{\omega\colon x\to x} &\abs{\{j\geq 1 \mid \omega_{j}=y\}}
    w_{\lambda,z}(\omega)
    = \alpha_{0}\sum_{k\geq 1} \mathop{\sum_{x_{0}, \dots,
        x_{k}}}_{\mathrm{distinct}} \sum_{i=1}^{k}\indicatorthat{x_{0}=x} 
    \indicatorthat{x_{k}=y} \lambda^{k} \\
    &\!\!\!\!\!\!\!\!\mathop{\sum_{\omega^{(i)}\colon x_{i-1}\to
        x_{i}}}_{\omega^{(k+i)}\colon x_{k-i+1}\to x_{k-i}}
    \cb{\prod_{i=1}^{k} w_{\lambda,z}(\omega^{(i)})
    \indicatorthat{\omega^{(i)}\in A_{i}}}
    \cb{\prod_{i=1}^{k} w_{\lambda,z}(\omega^{(k+i)})
    \indicatorthat{\omega^{(k+i)}\in B_{i}}}
  \end{align}
  where $A_{i}$ and $B_{i}$ are defined as follows. A walk $\omega$ is
  in $A_{i}$ if $\omega\cb{1\splice}$ does not hit $\LE(\omega^{(j)})$
  for any $j<i$.  A walk $\omega$ is in $B_{i}$ if
  \begin{enumerate}
  \item $\omega$ does not hit $\omega^{k-j}$ for $j>i+1$,
  \item $\omega\cb{1\splice}$ hits $\omega^{k-i}$ at
    $\omega^{k-i}_{0}$,
  \item $\omega$ hits $\omega^{k-i-1}$ at $\omega^{k-i}_{0}$, and
    $\omega$ does not hit $\LE(\omega^{k-i-1})\setminus \{
    \omega^{k-i}_{0}\}$.
  \end{enumerate}
\end{proposition}
\begin{proof}
  Rewrite $\abs{ \{ j\geq 1 \mid \omega_{j}=y\}}$ as $\sum_{j\geq 1}
  \indicatorthat{\omega_{j}=y}$. To prove the claim it suffices to
  show that walks with $\omega_{j}=y$ are in bijection with the
  summands such that $\abs{ \omega^{(1)} \circ \dots \circ
  \omega^{(k)}} = j$. 

  Suppose $\omega_{j}=y$, and let $\eta = \LE(\omega\cb{0,j})$. Let
  $t^{\ell},s^{\ell}$ be $t^{\ell}_{\omega}(\eta)$ and $s^{\ell}_{\eta}(\omega)$,
  respectively. Assume there are $k$ shrinking times for the walk
  $\omega$. Observing that $\omega$ closed implies $t^{k}=0$,
  $s^{k}=\abs{\omega}$ implies
  \begin{align}
    \label{eq:LWW-BC-1}
    \omega\cb{0,j} &= \eta^{-1}\cb{t^{k},t^{k-1}} \diamond \dots\diamond
    \eta^{-1}\cb{t^{2},t^{1}}
    \diamond \eta^{-1}\cb{t^{1},\abs{\eta}}\\
    \label{eq:LWW-BC-2}
    \omega\cb{j\splice} &= \omega\cb{j,s^{1}} \diamond \dots \diamond
    \omega \cb{s^{k-1},s^{k}}.
  \end{align}
  Call the subwalks on the right-hand sides of~\eqref{eq:LWW-BC-1}
  and~\eqref{eq:LWW-BC-2} the \emph{constituents} of $\omega\cb{0,j}$ and
  $\omega\cb{j\splice}$, respectively. Call a walk $\omega\colon x\to
  x$ an \emph{excursion} if the only occurrences of $x$ in $\omega$ are
  $\omega_{0}$ and $\omega_{\abs{\omega}}$.
  
  Separating any initial excursions from $x$ to $x$ from the first
  subwalk comprising $\omega\cb{0,j}$ gives the factor $\alpha_{0}$.
  To complete the claim, notice that any excursions immediately after
  a shrinking time that occur prior to the next hitting time of
  $\eta^{\ell}$ can be transferred to the previous subwalk comprising
  $\omega\cb{j\splice}$. In the case of the first constituent of
  $\omega\cb{j\splice}$ the excursions can be transferred to the last
  constituent of $\omega\cb{0,j}$.
\end{proof}

The next proposition handles the case of visits to the initial vertex
of a walk.
\begin{proposition}
  \label{prop:LWW-BC-Diag}
  \begin{equation}
    \label{eq:LWW-BC-Diag}
    \mathop{\sum_{\omega\colon x\to x}}_{\abs{\omega}\geq 1}
    \abs{\{j\geq 1 \mid \omega_{j}=x\}} 
    w_{\lambda,z}(\omega) = \alpha_{0}(\alpha_{0}-1).
  \end{equation}
\end{proposition}
\begin{proof}
  Write $\abs{\{j\geq 1\mid\omega_{j}=x\}}$ as $\sum_{j\geq 1}
  \indicatorthat{\omega_{j}=x}$. Inserting this into the left-hand
  side of~\eqref{eq:LWW-BC-Diag} and split each walk $\omega$ at time
  $j$. Summing the remainder after time $j$ gives a factor
  $\alpha_{0}$. Summing over $j$ gives $\alpha_{0}-1$ as $j\geq 1$
  implies the empty walk is excluded.
\end{proof}

To avoid explicitly writing the cumbersome
right-hand side of~\eqref{eq:LWW-Bubble-Chain} repeatedly it will be
convenient to introduce a short-hand definition:
\begin{definition}
  \label{def:LWW-True-BC}
  The \emph{bubble chain $\truebubblechain_{\lambda,z}(x,y)$ from $x$
    to $y$} is defined to be $\alpha_{0}(\alpha_{0}-1)$ if $x=y$ and
  the right-hand side of~\eqref{eq:LWW-Bubble-Chain} if $x\neq y$.
\end{definition}

The next decomposition formula is the analogue of
\Cref{prop:LWW-Bubble-Chain} for walks $\omega$ that are not
closed. Some notation will be needed: for $\eta$ a self-avoiding walk
ending at $x$ define $\truebubblechain_{\lambda,z}^{\eta}(x,y)$ to be
the bubble chain in $\Z^{d} \setminus \{\eta_{0}, \dots,
\eta_{\abs{\eta}-1}\}$. See \Cref{fig:LWW-Chain}.

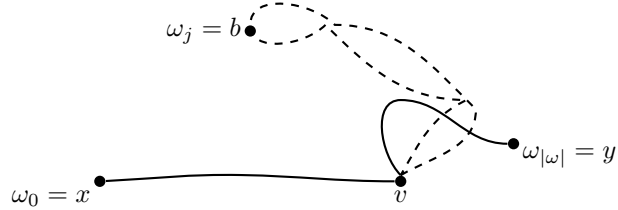
\begin{figure}[h]
  \centering
  \beginpgfgraphicnamed{fig6}
  \begin{tikzpicture}
    \node[black,fill,circle,inner sep=1.5pt,below] (s00) at (0,0) {};
    \node (s20) at (2,0) {};
    \node[black,fill,circle, inner sep = 1.5pt,below] (s40) at (4,0) {};
    \node (s60) at (6,0) {};
    \node (s51) at (5,1) {};
    \node (s32) at (3,2) {};
    \node[black,fill,circle, inner sep = 1.5pt,below] (s555) at (5.5,.5) {};
    \node[black,fill,circle,inner sep=1.5pt,below] (s22) at (2,2) {};
    \draw[black,thick] (s00) to[out=3,in=179] (s20) to[out=-1,in=180]
    (s40);
    \draw[black,thick,dashed] (4,0) to[out=30,in=270] (s51)
    to[out=135,in=0] 
    (s32) to[out=135,in=90] (s22) to[out=270,in=225] (3,2)
    to[out=-45,in=180] (s51) to[out=205,in=60] (4,0);
    \draw[black,thick] (4,0) to[out=135,in=180] (4,1) to[out=0,in=180]
    (s555);
    \node[below left] at (s00) {$\omega_{0}=x$};
    \node[below right] at (5.5,.5) {$\omega_{\abs{\omega}}=y$};
    \node[left] at (s22) {$\omega_{j}=b$};
    \node[below] at (s40) {$v$};
  \end{tikzpicture}
  \endpgfgraphicnamed
  \caption[A contribution to the sum in
  \Cref{eq:LWW-Bubble-Chain-Split}]{An illustration of a contribution
    to to the sum in \Cref{eq:LWW-Bubble-Chain-Split}. For clarity,
    only the loop erasure of each walk is shown. The dashed black path
    is the bubble chain portion of the walk. The vertex $v$ indicates
    the division between the path prior to the bubble chain and after
    the bubble chain.}
  \label{fig:LWW-Chain}
\end{figure}

\begin{proposition}
  \label{prop:LWW-Bubble-Chain-Split}
  Fix $x,y,b\in\Z^{d}$, $x\neq y$, $b\neq x$. Then
  \begin{align}
    \label{eq:LWW-Bubble-Chain-Split}
    \sum_{\omega\colon x\to y}&
    \indicatorthat{x\notin\omega\cb{1\splice}} \abs{\{j\geq 1 \mid
      \omega_{j}=b\}} w_{\lambda,z}(\omega) = \\&\sum_{a\in
      \Z^{d}} \mathop{\sum_{\omega^{(1)}\colon x\to a}}_{x\notin\omega\cb{1\splice}}
    \mathop{\sum_{\omega^{(2)}\colon a\to
        y}}_{\omega^{(2)}\cb{1\splice} \cap
      \LE(\omega^{(1)})=\emptyset}\!\!\!\!\!\!\! \ob{\delta_{a,b} +
    \truebubblechain_{\lambda,z}^{\LE(\omega^{(1)})}(a,b)}
    w_{\lambda,z}(\omega^{(1)}) w_{\lambda,z}(\omega^{(2)}).
  \end{align}
\end{proposition}
\begin{proof}
  This follows by writing $\abs{ \{j\geq 1 \mid \omega_{j}=b\}}$ as
  $\sum_{j\geq 1} \indicatorthat{\omega_{j}=b}$ and noting that this
  splits, by applying \Cref{eq:LWW-Segment-Split} with $\eta =
  \LE(\omega\cb{0,j})$, a walk $\omega$ into (i) an initial segment
  $\omega^{(1)}$ whose loop erasure is the subset of
  $\LE(\omega\cb{0,j})$ that is contained in $\LE(\omega)$, (ii) a
  bubble chain from the endpoint of $\omega^{(1)}$ to $b$ whose walks
  do not hit $\LE(\omega^{(1)})$; if the endpoint of $\omega^{(1)}$ is
  $b$ then it is also possible this walk is null, and (iii) a walk
  $\omega^{(2)}$ from the endpoint of $\omega^{(1)}$ to $y$ that does
  not, after the first vertex, hit $\LE(\omega^{(1)})$.
\end{proof}
The restriction in \Cref{prop:LWW-Bubble-Chain-Split} to walks
$\omega$ that do not return to their initial vertex is simply because
this is the type of sum that will occur most frequently in what follows.

\subsubsection{Two-Point Functions and Their Derivatives}
\label{sec:LWW-2PT-Deriv}

The quantity $I^{\omega}_{\lambda,z}(a,b)$ defined in
\Cref{eq:LWW-Specialization-I2PF} is inconvenient due to its
dependence on the details of $\omega$; the next definition introduces
a simple upper bound.
\begin{definition}
  \label{def:LWW-LRW-I2P}
  The \emph{interaction two-point function} $I_{\lambda,z}(x,y)$ is
  the function
  \begin{equation}
    \label{eq:LWW-LRW-I2P}
    I_{\lambda,z}(x,y) = \indicatorthat{x=y} + \indicatorthat{x\neq
      y}\ob{1 - e^{- \mu_{\lambda,z}(x,y;\emptyset)}}.
  \end{equation}
\end{definition}

\begin{lemma}
  \label{lem:LWW-Repulsion-I}
  Let $\omega$ be a walk of length $n$, and let $0\leq a<b\leq n$.
  \begin{equation}
    \label{eq:LWW-Repulsion-I}
    I_{\lambda,z}^{\omega}(a,b) \leq I_{\lambda,z}(\omega_{a},\omega_{b}). 
  \end{equation}
\end{lemma}
\begin{proof}
  If $\omega_{a}=\omega_{b}$ then~\eqref{eq:LWW-Repulsion-I} is an
  equality. If $\omega_{a}\neq \omega_{b}$ the inequality follows
  because the loop measure is decreasing in its final
  argument.
\end{proof}

The important aspect of the next bound is that it is independent of
$\eta$.
\begin{proposition}
  \label{prop:LWW-I2P-DB}
  Let $\eta\colon x\to y$ be a self-avoiding walk. Then
  \begin{equation}
    \label{eq:LWW-I2P-DB}
    \frac{d}{dz} I^{\eta}_{\lambda,z}(x,y) \leq \indicatorthat{x\neq
      y}z^{-1} \sum_{a\in \Z^{d}}
    \mathop{\sum_{\omega\colon a\to a}}_{\abs{\omega}\geq 1}
    \indicatorthat{x\in\omega} \indicatorthat{y\in\omega}
    w_{\lambda,z}(\omega).
  \end{equation}
  Further, the right-hand side of~\eqref{eq:LWW-I2P-DB} is an upper bound
  for $\frac{d}{dz} I_{\lambda,z}(x,y)$ as well.
\end{proposition}
\begin{proof}
   Differentiate, and then use $e^{-x}\leq 1$ for $x\geq 0$.
\end{proof}

\begin{definition}
  \label{def:LWW-Reduced-2PT}
  The \emph{scaled two-point functions} $\bar G_{\lambda,z}(x,y)$ and
  $\bar H_{\lambda,z}(x,y)$ are defined by
  \begin{equation}
    \bar G_{\lambda,z}(x,y) = \alpha_{0}(\lambda,z)^{-1}
    G_{\lambda,z}(x,y), \qquad \bar H_{\lambda,z}(x,y) =
    \alpha_{0}(\lambda,z)^{-1} H_{\lambda,z}(x,y).
  \end{equation}
\end{definition}

Let $\bar B_{\lambda,z}(x) = \bar H_{\lambda,z}(x)^{2}$. An upper
bound on $\truebubblechain_{\lambda,z}$ is obtained by dropping
the constraints $A_{i}$ and $B_{i}$.
\begin{definition}
  \label{def:LWW-BC}
  Define $\bubblechain_{\lambda,z}(x,y)$ by
  \begin{equation}
    \label{eq:LWW-BC}
    \bubblechain_{\lambda,z}(x,y) = \alpha_{0}
    \begin{cases}
      \sum_{k\geq 1} \lambda^{k}\underbrace{\bar B_{\lambda,z}\ast \dots
        \ast \bar B_{\lambda,z}}_{\textrm{$k$ terms}} (x,y) &
      x\neq y \\
      \alpha_{0} - 1 & x=y
    \end{cases}
  \end{equation}
\end{definition}

\begin{proposition}
  \label{prop:LWW-BC-Bound}
  Let $\eta$ be any self-avoiding walk ending at $x$. Then
  \begin{equation}
    \label{eq:LWW-BC-Bound}
    \truebubblechain_{\lambda,z}^{\eta}(x,y) \leq
    \truebubblechain_{\lambda,z}(x,y) \leq 
    \bubblechain_{\lambda,z}(x,y).
  \end{equation}
\end{proposition}
\begin{proof}
  The first inequality follows as the set of summands is increasing
  from left to right and all summands are non-negative. For the second
  inequality note that relaxing the conditions $A_{i}$ and $B_{i}$
  increases the set of summands. Using $H_{\lambda,z}(x,y) =
  H_{\lambda,z}(y,x)$, which follows from \Cref{thm:LWW-LM-Rep}, to
  reverse the direction of the walks $\omega^{(k+i)}$ gives the
  upper bound $\bubblechain_{\lambda,z}(x,y)$.
\end{proof}

The next lemma shows that if a sum over walks satisfying some
constraints is upper bounded by relaxing the constraints, an upper
bound on the derivative is obtained by differentiating the upper
bound. This will be used frequently.
\begin{lemma}
  \label{lem:LWW-Derivative-UB}
  Suppose $A,B$ are two sets of walks, and $A\subset B$. Then
  \begin{equation}
    \frac{d}{dz}\sum_{\omega\in A} w_{\lambda,z}(\omega) \leq
    \frac{d}{dz} \sum_{\omega\in B} w_{\lambda,z}(\omega).
  \end{equation}
\end{lemma}
\begin{proof}
  Each summand is non-negative as the weight of a walk $\omega$ is
  proportional to $z^{\abs{\omega}}$, and the set of summands on the
  right-hand side is larger.
\end{proof}

The formulas of \Cref{sec:LWW-Visits} yield diagrammatic bounds on
derivatives of two-point functions by applying the identity
\begin{equation}
  \label{eq:LWW-Walk-Leibniz}
  \abs{\omega} = \sum_{a\in \Z^{d}} \abs{ \{ j\geq 1 \mid \omega_{j}=a\}},
\end{equation}
where $j=0$ is not included because there $\abs{\omega}+1$ vertices
in a walk.

\begin{proposition}
  \label{prop:LWW-H-DB}
  For $x\in \Z^{d}$, $x\neq 0$,
  \begin{equation}
    \label{eq:LWW-H-DB}
    \frac{d}{dz} \bar G_{\lambda,z}(x) =
    \frac{d}{dz} \bar H_{\lambda,z}(x)
    \leq z^{-1} (1+\norm{\bubblechain_{\lambda,z}}_{1}) \bar G_{\lambda,z}\ast
    \bar H_{\lambda,z}(x).
  \end{equation}
\end{proposition}
\begin{proof}
  The first equality is straightforward as $\bar G_{\lambda,z}(x) =
  \delta_{0,x} + \bar H_{\lambda,z}(x)$ by \Cref{eq:LWW-G-H-Relation}.
  For the inequality observe that
  \begin{equation*}
    \frac{d}{dz} \bar H_{\lambda,z}(x) = z^{-1}
    \mathop{\sum_{\omega\colon 0\to x}}_{0\notin \omega\cb{1\splice}}  
    \abs{\omega} w_{\lambda,z}(\omega)
  \end{equation*}
  Applying~\eqref{eq:LWW-Walk-Leibniz}
  and~\Cref{prop:LWW-Bubble-Chain-Split} yields
  \begin{equation}
    z^{-1}\sum_{b\in \Z^{d}} \sum_{a\in \Z^{d}}
    \mathop{\sum_{\omega^{(1)}\colon 0\to
        a}}_{0\notin\omega\cb{1\splice}}
    \mathop{\sum_{\omega^{(2)}\colon a\to
        x}}_{\omega^{(2)}\cb{1\splice} \cap
      \LE(\omega^{(1)})=\emptyset} (\delta_{a,b} +
    \truebubblechain_{\lambda,z}^{\LE(\omega^{(1)})}(a,b))
    w_{\lambda,z}(\omega^{(1)}) w_{\lambda,z}(\omega^{(2)}).
  \end{equation}
  By \Cref{prop:LWW-BC-Bound} removing the restriction on the bubble
  chain gives an upper bound.  The sum over $b$ then gives the factor
  $1+ \norm{\bubblechain_{\lambda,z}}_{1}$. Dropping the constraint
  that $\omega^{(2)}$ does not intersect $\LE(\omega^{(1)})$ gives the
  claim.
\end{proof}

\begin{proposition}
  \label{prop:LWW-Alpha-DB}
  \begin{equation}
    \label{eq:LWW-Alpha-DB}
    \frac{d}{dz} \alpha_{0}(\lambda,z) = z^{-1}
    \norm{\bubblechain_{\lambda,z}}_{1} 
  \end{equation}
\end{proposition}
\begin{proof}
  As a zero step walk does not survive being differentiated,
  \begin{equation}
    \frac{d}{dz} \alpha_{0}(\lambda,z) =
    z^{-1}\mathop{\sum_{\omega\colon 0 \to 0}}_{\abs{\omega} \geq 1}
    \abs{\omega} w_{\lambda,z}(\omega).
  \end{equation}
  The proposition follows by (i) applying~\eqref{eq:LWW-Walk-Leibniz}
  to rewrite $\abs{\omega}$, (ii) using
  \Cref{prop:LWW-Bubble-Chain,prop:LWW-BC-Diag} to recognize the
  resulting sum as the $1$-norm of the bubble chain, and (iii) using
  \Cref{prop:LWW-BC-Bound} to upper bound the norm of the bubble
  chain.
\end{proof}

\begin{proposition}
  \label{prop:LWW-BC-DB}
  \begin{align}
    \label{eq:LWW-BC-DB}
    \frac{d}{dz} \norm{\truebubblechain_{\lambda,z}}_{1} &\leq
    z^{-1}\norm{\bubblechain_{\lambda,z}}_{1}\ob{ 3\alpha_{0} - 1 -
      \alpha_{0}^{2} + \norm{\bubblechain_{\lambda,z}}_{1}} \\ &+
    2\alpha_{0}z^{-1}\lambda \norm{ \bar H_{\lambda,z}\cdot \ob{\bar
        G_{\lambda,z}\ast \bar H_{\lambda,z}}}_{1}\ob{1 +
      \norm{\bubblechain_{\lambda,z}}_{1}}^{3}.
  \end{align}
\end{proposition}
\begin{proof}

  By \Cref{lem:LWW-Derivative-UB} it suffices to obtain bounds on the
  derivative of $\norm{\bubblechain_{\lambda,z}}_{1}$. For the summand
  with $x=y$ the an upper bound is
  $z^{-1}\norm{\bubblechain_{\lambda,z}}_{1}(\alpha_{0}-1) +
  z^{-1}\alpha_{0}\norm{\bubblechain_{\lambda,z}}_{1}$ by \Cref{prop:LWW-Alpha-DB}.

  For $x\neq y$ differentiating \Cref{eq:LWW-BC} and using
  \Cref{prop:LWW-Alpha-DB} gives an upper bound
  $z^{-1}\norm{\bubblechain_{\lambda,z}}_{1}\alpha_{0}^{-1}(\norm{\bubblechain_{\lambda,z}}_{1}
  - \alpha_{0}(\alpha_{0}-1))$ if the derivative is applied to
  $\alpha_{0}$. The factor of $\alpha_{0}^{-1}$ can be dropped to give
  an upper bound as $\alpha_{0}\geq 1$. When the derivative is not
  applied to $\alpha_{0}$ we have, using \Cref{prop:LWW-H-DB}, the upper bound
  \begin{align}
    \frac{d}{dz} \norm{\bubblechain_{\lambda,z}}_{1} &= \alpha_{0}\frac{d}{dz}
    \sum_{k\geq 1} \sum_{y} \lambda^{k} \underbrace{\bar
      H_{\lambda,z}^{2} \ast \dots
      \ast  \bar H_{\lambda,z}^{2}}_{\textrm{$k$ terms}}(y) \\
    &\leq 2\alpha_{0} \sum_{k\geq 1} \sum_{y} k\lambda^{k} \ob{\bar
      H_{\lambda,z}\frac{d}{dz} \bar H_{\lambda,z}}\ast
    \underbrace{\bar H_{\lambda,z}^{2} \ast \dots \ast \bar
      H_{\lambda,z}^{2}}_{\textrm{$k-1$ terms}}(y) \\
    \nonumber
    &= 2\alpha_{0}z^{-1}\lambda \norm{\bar H_{\lambda,z} \cdot \ob{\bar G_{\lambda,z}\ast
        \bar H_{\lambda,z}}}_{1} (1+\norm{\bubblechain_{\lambda,z}}_{1})^{3}.
  \end{align}
  Summing these upper bounds gives the result.
\end{proof}

\subsection{Diagrammatic Bounds}
\label{sec:LWW-DB-Pi}

The bounds derived in this section will be obtained under the
assumption that $f(z)\leq K$ for $z<z_{c}(\lambda)$. In particular the
results of \Cref{prop:LWW-SL5.10} hold. It will also be assumed that
the dimension $d$ is sufficiently large, i.e., $\beta$ is sufficiently
small.

\subsubsection{Initial Diagrammatic Bounds}
\label{sec:LWW-DB-Auxiliary}

\begin{proposition}
  \label{prop:LWW-UB-Alpha}
  If $z<z_{c}$ and $f(z)\leq K$ then $\alpha_{0}(\lambda,z)\leq 1+ c\beta$.
\end{proposition}
\begin{proof}
  By definition and \Cref{thm:LWW-LM-Rep}
  \begin{equation}
    \label{eq:LWW-UB-Alpha.1}
    \alpha_{0}(\lambda,z) = \exp \ob{\mu_{\lambda,z}(0)} 
    = 1 + \mathop{\sum_{\omega\colon 0 \to 0}}_{\abs{\omega} \geq 1} 
    w_{\lambda,z}(\omega).
  \end{equation}
  The walks contributing to the sum have their last vertex a neighbour
  of $0$, so 
  \begin{equation}
    \label{eq:LWW-UB-Alpha.2}
    \mathop{\sum_{\omega\colon 0 \to 0}}_{\abs{\omega} \geq 1} w_{\lambda,z}(\omega)
    = z\lambda\abs{\Omega}D\ast H_{\lambda,z}(0),
  \end{equation}
  which is bounded by
  $z\lambda\abs{\Omega}\norm{H_{\lambda,z}}_{\infty}$. The claim
  follows from $z\abs{\Omega}\leq f_{1}(z) \leq K$ and~\eqref{eq:LWW-SL5.10.3}.
\end{proof}

\begin{proposition}
  \label{prop:LWW-BC-Geometric}
  If $z<z_{c}$ and $f(z)\leq K$ then $\norm{\bubblechain_{\lambda,z}}_{1}\leq c\beta$.
\end{proposition}
\begin{proof}
  Repeatedly using $\norm{f\ast g}_{1}\leq \norm{f}_{1}\norm{g}_{1}$ implies
  \begin{equation}
    \norm{\bubblechain_{\lambda,z}}_{1} \leq \alpha_{0}\ob{ (\alpha_{0}-1) +
      \sum_{k\geq 1} \lambda^{k} \norm{\bar H_{\lambda,z}}_{2}^{2k}},
  \end{equation}
  The interchange of summations is valid as each term is
  non-negative. By \Cref{prop:LWW-UB-Alpha} $\alpha_{0}\leq 1 + c\beta$
  so $\alpha_{0}-1 \leq c\beta$. Since $\alpha_{0}\geq 1$, $\norm{\bar
    H_{\lambda,z}}_{2}^{2}\leq \norm{H_{\lambda,z}}_{2}^{2}$, so
  \Cref{eq:LWW-SL5.10.2} implies that for $\beta$ sufficiently
  small 
  \begin{equation}
    \sum_{k\geq 1}\lambda^{k}\norm{\bar H_{\lambda,z}}_{2}^{2k} \leq
    c\beta. \qedhere
  \end{equation}
\end{proof}

\begin{proposition}
  \label{prop:LWW-UB-I}
  Let $I_{\lambda,z}(x) = I_{\lambda,z}(0,x)$. If $z<z_{c}$ and $f(z)\leq K$ then
  $\norm{I_{\lambda,z}}_{1} \leq 1 + c\beta$.
\end{proposition}
\begin{proof}
  The inequality $1-e^{-x}\leq x$ implies that
  $1+\norm{\indicatorthat{x\neq 0} \mu_{\lambda,z}(0,x)}_{1}$ is an
  upper bound for $\norm{I_{\lambda,z}}_{1}$. The factor of $1$ is
  from the term $\indicatorthat{x=0}$ in $I_{\lambda,z}$. Observe that
  $\norm{\indicatorthat{x\neq 0}\mu_{\lambda,z}(0,x)}_{1}$ is bounded
  by
  \begin{equation}
    \label{eq:LWW-UB-I-1}
    \sum_{x\neq 0} \sum_{y} \mathop{\sum_{\omega\colon y \to y}}_{\abs{\omega}\geq 1}
    \indicatorthat{0\in \omega} \indicatorthat{x\in \omega}
    \frac{ w_{\lambda,z}(\omega)}{\abs{\omega}}
    \leq \sum_{y} \mathop{\sum_{\omega\colon y\to y}}_{\abs{\omega}\geq
      1} \indicatorthat{0\in\omega} w_{\lambda,z}(\omega),
  \end{equation}
  as $\sum_{x\neq 0}\indicatorthat{x\in\omega} \leq
  \abs{\range{\omega}}\leq \abs{\omega}$. By translation invariance
  this is
  \begin{equation}
    \label{eq:LWW-UB-I-2}
    \sum_{y} \mathop{\sum_{\omega\colon 0 \to 0}}_{\abs{\omega}\geq 1}
    \indicatorthat{-y\in \omega} w_{\lambda,z}(\omega) = 
    \norm{\mathop{\sum_{\omega\colon 0 \to 0}}_{\abs{\omega}\geq 1}
    \indicatorthat{y\in \omega} w_{\lambda,z}(\omega)}_{1},
  \end{equation}
  where the norm is with respect to $y$. To establish the proposition
  (i) bound $\indicatorthat{y\in\omega}$ by $\abs{ \{j\geq 1 \mid
    \omega_{j}=y\}}$, (ii) apply \Cref{prop:LWW-BC-Diag} for the summands
  with $y=0$, (iii) apply \Cref{prop:LWW-Bubble-Chain} and
  \Cref{prop:LWW-BC-Bound} for the summands with $y\neq 0$, and (iv)
  observe that the sum of these two bounds is
  $\norm{\bubblechain_{\lambda,z}}_{1}$ and apply
  \Cref{prop:LWW-BC-Geometric}.
\end{proof}

\subsubsection{Bounds for $\pi^{(1)}$}
\label{sec:LWW-Pi-DB-1.1}

\begin{proposition}
  \label{prop:LWW-DB-1-Rep}
  \begin{equation}
    \label{eq:LWW-DB-1-Rep}
      \pi^{(1)}(x) = \sum_{m}\pi^{(1)}_{m} 
      \begin{cases}
        =   z\lambda\abs{\Omega} D\ast \bar H_{\lambda,z}(0) & x = 0 \\
        \leq \bar H_{\lambda,z}(x)I_{\lambda,z}(0,x)
        e^{\mu_{\lambda,z}(0,x)} & x\neq 0,
    \end{cases}
  \end{equation}
\end{proposition}
\begin{proof}
  For $x=0$ the claim follows from the identities in
  \Cref{prop:LWW-2PT-Equivalence,eq:LWW-Intro-X-Gas,eq:LWW-X-Gas-pi-1-timelike,eq:LWW-UB-Alpha.2}. For
  $x\neq 0$ use \Cref{eq:LWW-X-Pi-1}. Recall the loop measure
  representation of the second product, i.e., the expression for
  $I_{\cc X}^{\omega}$ given by \Cref{eq:LWW-Specialization-I2PF}. The
  desired bound follows by forgetting the constraint in the loop
  measure and the rearrangement
  $e^{\mu_{\lambda,z}(0,x)}-1 = e^{\mu_{\lambda,z}(0,x)}
  I_{\lambda,z}(0,x)$.
\end{proof}

\begin{proposition}
  \label{prop:LWW-DB-1}
  Suppose $f(z)\leq K$. The following bounds hold for $u=0,1$ and $k\in \FS^{d}$:
  \begin{equation}
    \label{eq:LWW-DB-1}
    \norm{\abs{x}^{2u}\pi^{(1)}}_{1} \leq
    c\beta\ob{\indicatorthat{u=0}+ \norm{\abs{x}^{2u}
        \bar H_{\lambda,z}}_{\infty}},
  \end{equation}
  and
  \begin{equation}
    \label{eq:LWW-DB-1-Trig}
    \norm{(1-\cos k\cdot x)\pi^{(1)}(x)}_{1}  \leq c\beta
    \norm{(1-\cos k \cdot x)\bar H_{\lambda,z}}_{\infty}.
  \end{equation}
\end{proposition}
\begin{proof}
  The triangle inequality, $\norm{f\ast g}_{1}\leq \norm{f}_{\infty}
  \norm{g}_{1}$ with $g=I_{\lambda,z}$, and $1-\cos 0 =0$ imply
  \begin{align}
    \label{eq:LWW-DB-1.1}
    \norm{\abs{x}^{2u}\pi^{(1)}}_{1} &\leq \indicatorthat{u=0} 
    z \abs{\Omega} \norm{\bar H_{\lambda,z}}_{\infty} 
    + \norm{I_{\lambda,z}(0,x)
      e^{\mu_{\lambda,z}(0,x)}\indicatorthat{x\neq 0}}_{1} 
    \norm{ \abs{x}^{2u}\bar H_{\lambda,z}}_{\infty} \\
    \label{eq:LWW-DB-1-Trig.1}
    \norm{(1-\cos k\cdot x)\pi^{(1)}(x)}_{1} & \leq
    \norm{I_{\lambda,z}(0,x) e^{\mu_{\lambda,z}(0,x)}\indicatorthat{x\neq 0}}_{1} 
    \norm{(1-\cos k \cdot x)\bar H_{\lambda,z}}_{\infty}.
  \end{align}
  Using $z\abs{\Omega} \leq
  f_{1}(z) \leq K$ by $f_{3}\leq K$ and
  $\sup_{x}e^{\mu_{\lambda,z}(0,x)} \leq \alpha_{0}(\lambda,z)$ implies
  \begin{align*}
    \norm{I_{\lambda,z}(0,x)
      e^{\mu_{\lambda,z}(0,x)}\indicatorthat{x\neq 0}}_{1} 
    &\leq \alpha_{0}\norm{I_{\lambda,z}(0,x)\indicatorthat{x\neq 0}}_{1} \\
    &\leq \alpha_{0}\ob{ \norm{I_{\lambda,z}(0,x)}_{1} - 1}.
  \end{align*}
  The conclusion now follows from \Cref{prop:LWW-UB-Alpha,prop:LWW-UB-I}.
\end{proof}

\subsubsection{Bounds for $\pi^{(N)}$, $N\geq 2$}
\label{sec:LWW-Pi-DB-N.1}

\begin{proposition}
  \label{prop:LWW-Pi-Repulsive-Bound}
  Let $m\geq 2$, $x\in \Z^{d}$, and $N\geq 2$. Let $x_{0}=0,
  x_{N-1}^{\prime}=x$. Then
  \begin{align}
    \label{eq:LWW-Pi-Repulsive-Bound}
    \abs{\pi_{m}^{(N)}(x)} \leq &\sum_{\vec m} \mathop{\sum_{x_{1},
        \dots, x_{N-1}}}_{x_{0}^{\prime}, \dots, x_{N-2}^{\prime}}
    \mathop{\sum_{\omega^{(1)}\colon 0 \to
        x_{1}}}_{\abs{\omega^{(1)}}=m_{1}}
    \mathop{\sum_{\omega^{(2)}\colon x_{1} \to
        x_{0}^{\prime}}}_{\abs{\omega^{(2)}}=m_{2}} \dots
    \mathop{\sum_{\omega^{(2N-2)}\colon x_{N-1} \to
        x_{N-2}^{\prime}}}_{\abs{\omega^{(2N-2)}}=m_{2N-2}}
    \mathop{\sum_{\omega^{(2N-1)}\colon x_{N-2}^{\prime} \to
        x}}_{\abs{\omega^{(2N-1)}}=m_{2N-1}} \\
    &\prod_{j=0}^{N-1}\abs{ I_{\lambda,z}(x_{j},x_{j}^{\prime})}
    \prod_{k=1}^{2N-1}
    \alpha_{0}^{-1}\exp\ob{\mu_{\lambda,z}(\range{\omega^{(k)}})}
  \end{align}
  where the summation is over valid vectors $\vec m$ (recall
  \Cref{def:LWW-Valid}) of subinterval lengths such that $\sum m_{i}=m$.
\end{proposition}
\begin{proof}
  This follows from \Cref{sec:LWW-X-Gas-pi-N}. By
  \Cref{lem:LWW-Repulsion-I} the factors of $I^{\omega}_{\lambda,z}$
  can be replaced by $I_{\lambda,z}$. In the language of an $\cc X$
  gas, as $\alpha_{\omega}\geq 0$ for any walk $\omega$, the $\cc X$
  gas for $\lambda$-LWW is repulsive in the sense of
  \Cref{rem:LWW-X-Gas-Repulsion}, and this proves the claim.
\end{proof}

Upper bounds on $\norm{\pi^{(N)}(x)}_{1}$ can be efficiently found by
formulating \Cref{prop:LWW-Pi-Repulsive-Bound} in terms of
multiplication and convolution operators. Let $\cc M_{g}$ and $\cc
C_{g}$ denote multiplication and convolution by $g$, respectively:
$\cc M_{g}f = gf$ and $\cc C_{g}f = g\ast f$.

\begin{lemma}
  \label{lem:LWW-MC-Formulation}
  Fix $N\geq 2$ and let $\bar H = \bar
  H_{\lambda,z}$, $\bar G = \bar G_{\lambda,z}$, and $I =
  I_{\lambda,z}$. Then
  \begin{equation}
    \label{eq:LWW-MC-Bare}
    \sum_{x}\abs{\pi^{(N)}(x)} \leq \norm{ \ob{\cc C_{\bar H\ast I} \cc
        M_{\bar H}}\ob{\cc C_{\bar G\ast I} \cc
        M_{\bar H}}^{N-2} \bar H\ast I}_{\infty}.
  \end{equation}
\end{lemma}
\begin{proof}
  The definition of a valid vector of lengths implies that
  summing~\eqref{eq:LWW-Pi-Repulsive-Bound} over all valid vectors of
  lengths results in the sums over walks with indices $1$, $2j$, and
  $2N-1$ being replaced by $\bar H_{\lambda,z}$, and the remaining
  sums of walks are replaced by $\bar G_{\lambda,z}$. Consulting
  \Cref{fig:LWW-Diagrammatic-1}, this means that all horizontal solid
  lines except the leftmost and rightmost are weighted by $\bar
  G_{\lambda,z}$, while the rest are weighted by $\bar
  H_{\lambda,z}$. Formally,
  \begin{align}
    \label{eq:LWW-Pi-2PT-Bound}
    \abs{\pi^{(N)}(x)} \leq &\mathop{ \sum_{x_{1}, \dots, x_{N-1}}
    }_{x_{0}^{\prime}, \dots, x_{N-2}^{\prime}}
    \ob{\prod_{j=0}^{N-1} I_{\lambda,z}(x_{j},x_{j}^{\prime})}
    \bar H_{\lambda,z}(x_{0},x_{1})\\&
    \ob{\prod_{j=0}^{N-2} \bar H_{\lambda,z}(x_{j}^{\prime}, x_{j+1})}
    \ob{\prod_{j=0}^{N-3}\bar G_{\lambda,z}(x_{j}^{\prime},x_{j+2})}
    \bar H_{\lambda,z}(x_{N-2}^{\prime},x_{N-1}^{\prime}).
  \end{align}
  Replace the factor $I_{\lambda,z}(x_{0},x_{0}^{\prime})$ by
  $I_{\lambda, z}(y,x_{0}^{\prime})$ in~\eqref{eq:LWW-Pi-2PT-Bound} and
  call the resulting function $F(x,y)$. As $\sum_{x}\abs{F(x,0)} =
  \sum_{x}\abs{\pi^{(N)}(x)}$ the quantity
  $\sup_{y}\sum_{x}\abs{F(x,y)}$ is an upper bound for the left-hand
  side of~\eqref{eq:LWW-MC-Bare}. The associativity of convolution
  implies
  \begin{equation}
    \sum_{x}\abs{F(x,y)} = \ob{(\cc C_{\bar H} \cc C_{I}) \cc M_{\bar H} \ob{ \cc
      C_{\bar G} \cc C_{I} \cc M_{\bar H}}^{N-2} \bar H\ast I}(y).
  \end{equation}
  \Cref{eq:LWW-MC-Bare} follows as $\cc C_{\bar G}\cc C_{I} = \cc C_{\bar
    G\ast I}$ and $\cc C_{\bar H} \cc C_{I} = \cc C_{\bar H \ast I}$.
\end{proof}

The right-hand side of \Cref{lem:LWW-MC-Formulation} can be easily
estimated with the help of the next lemma.
\begin{lemma}[name = Lemma~4.6 of~\cite{Slade2006}]
  \label{lem:LWW-SL-Lp}
  Given non-negative even functions $f_{0}, f_{1}, \dots, f_{2M}$ on
  $\Z^{d}$, define $\cc C_{j}$ and $\cc M_{j}$ to be the operations of
  convolution with $f_{2j}$ and multiplication by $f_{2j-1}$ for $j=1,
  \dots, M$. Then for any $k\in \{0, \dots, 2M\}$,
  \begin{equation}
    \label{eq:LWW-SL-Lp}
    \norm{\cc C_{M}\cc M_{M} \dots \cc C_{1} \cc M_{1}f_{0}}_{\infty}
    \leq \norm{f_{k}}_{\infty} \prod \norm{ f_{j} \ast f_{j^{\prime}}}_{\infty},
  \end{equation}
where the product is over disjoint consecutive pairs $j,j^{\prime}$
taken from $\{0, \dots, 2M\} \setminus \{k\}$.
\end{lemma}

The strange formatting of the bounds in the next proposition are
strictly for typographic convenience; in applications we multiply
through by the denominators of the left-hand sides.
\begin{proposition}
  \label{prop:LWW-DB-N}
  Let $N\geq 2$. Then for $z<z_{c}$ and $u\in\{0,1\}$
  \begin{align}
    \label{eq:LWW-Pi-UB-N.1}
   \frac{\norm{\abs{x}^{2u}\pi^{(N)}(x)}_{1}}{(2N-1)^{u}}
    &\leq \norm{\abs{x}^{2u}\bar H_{\lambda,z}}_{\infty} 
    (c\beta)^{N-2 + \indicatorthat{N=2}}
    (1+c\beta)^{N+\indicatorthat{N\geq 3}} \\
    \label{eq:LWW-Pi-UB-N.2}
    \frac{\norm{(1-\cos(k\cdot x)) \pi^{(N)}(x)}_{1}}{(4N-1)(2N-1)}
    &\leq \norm{(1-\cos (k\cdot x))\bar
      H_{\lambda,z}(x)}_{\infty}
    (c\beta)^{N-2+\indicatorthat{N=2}}
    (1+c\beta)^{N+\indicatorthat{N\geq 3}}
  \end{align}
\end{proposition}
\begin{proof}
  Suppose that both
  \begin{equation}
    \label{eq:LWW-DB-N.1}
    \frac{\norm{\abs{x}^{2u}\pi^{(N)}(x)}_{1}}{(2N-1)^{u}} \leq 
    \norm{\abs{x}^{2u}\bar H_{\lambda,z}}_{\infty} \norm{\bar
      G_{\lambda,z} \ast \bar G_{\lambda,z}}_{\infty} \norm{\bar
      G_{\lambda,z}\ast \bar H_{\lambda,z}}^{N-2}_{\infty}
    \norm{I_{\lambda,z}}^{N}_{1}
    \end{equation}
    and
    \begin{align}
      \nonumber
    \frac{\norm{(1-\cos(k\cdot x)) \pi^{(N)}(x)}_{1}}{(4N-1)(2N-1)} \leq& 
    \norm{(1-\cos (k\cdot x))\bar H_{\lambda,z}(x)}_{\infty}
    \norm{\bar G_{\lambda,z}\ast \bar G_{\lambda,z}}_{\infty} 
      \\     \label{eq:LWW-DB-N.2} &
    \times\norm{\bar G_{\lambda,z}\ast \bar H_{\lambda,z}}^{N-2}_{\infty}
    \norm{I_{\lambda,z}}^{N}_{1}.
  \end{align}
  Suppose further that if $N=2$ the same bounds hold with each term
  $\bar G_{\lambda,z}$ replaced by $\bar H_{\lambda,z}$. The claim then
  follows, as \Cref{eq:LWW-G-H-Relation}, the triangle inequality,
  Cauchy-Schwarz, and $\bar H_{\lambda,z}\leq H_{\lambda,z}$ imply
  \begin{equation}
    \label{eq:LWW-Conv-Bound}
    \norm{\bar H_{\lambda,z}\ast \bar G_{\lambda,z}}_{\infty} = \norm{
      \bar H_{\lambda,z} + \bar H_{\lambda,z}\ast
      \bar H_{\lambda,z}}_{\infty}  \leq \norm{\bar H_{\lambda,z}}_{\infty} +
    \norm{\bar H_{\lambda,z}}_{2}^{2} \leq c\beta,
  \end{equation}
  and \Cref{prop:LWW-UB-I} implies $\norm{I_{\lambda,z}}_{1}\leq 1 +
  c\beta$. The rest of the proof establishes
  \Cref{eq:LWW-DB-N.1,eq:LWW-DB-N.2}.

  First observe that the difference between $N=2$ and $N\geq 3$ is
  only that all two-point functions in \Cref{lem:LWW-MC-Formulation} are
  $\bar H_{\lambda,z}$ for $N=2$, while for $N\geq 3$ factors of $\bar
  G_{\lambda,z}$ arise.

  If $u=0$ \Cref{eq:LWW-DB-N.1} follows by applying
  \Cref{lem:LWW-SL-Lp} to the right-hand side of
  \Cref{lem:LWW-MC-Formulation}, putting the sup norm on the final
  $I_{\lambda,z}\ast \bar H_{\lambda,z}$, and using the inequality
  \begin{equation}
    \norm{\bar G_{\lambda,z}\ast \bar H_{\lambda,z}\ast I_{\lambda,z}}_{\infty}
    \leq
    \norm{\bar G_{\lambda,z}\ast \bar H_{\lambda,z}}_{\infty}
    \norm{I_{\lambda,z}}_{1}. 
  \end{equation}

  For $u=1$, note that $x = x_{1} + \dots x_{2N-1}$, where $x_{j}$ is
  the displacement along the $j^{\mathrm{th}}$ subwalk in a summand
  contributing to $\pi^{(N)}$. As $\abs{x}^{2} \leq \sum
  \abs{x_{i}}^{2}$ it follows that an upper bound is given by
  \begin{equation}
    \sum_{j=1}^{2N-1} \norm{ \ob{\cc C_{\bar H\ast I} \cc
        M_{\bar H}}\ob{\cc C_{\bar G\ast I} \cc
        M_{\bar H}}^{N-2} \bar H\ast I}_{\infty},
  \end{equation}
  where the $j^{\mathrm{th}}$ two-point function $\bar G$ or $\bar H$
  is replaced with $\abs{x}^{2}\bar H$. The claim follows by (i)
  applying \Cref{lem:LWW-SL-Lp} and putting the sup norm on the term
  involving the factor of $\abs{x}^{2}$ (ii) noting that the resulting
  norms are of the form $\norm{\bar H \ast \bar H \ast I}_{\infty}$,
  $\norm{I\ast \bar G \ast \bar H}_{\infty}$, $\norm{I\ast I \ast\bar
    G \ast \bar G}_{\infty}$, or $\norm{\bar H \ast I}_{\infty}$ and
  (iii) iterating $\norm{f\ast g}_{\infty}\leq
  \norm{f}_{\infty}\norm{g}_{1}$. The uniform upper bound follows by
  using $\norm{\bar H\ast \bar H}_{\infty}\leq \norm{\bar H \ast \bar
    G}_{\infty}$.

  To prove \Cref{eq:LWW-DB-N.2} let $t= \sum_{j=1}^{n}
  t_{j}$. Then (see~\cite[Section~4.2.3]{Slade2006})
  \begin{equation}
    \label{eq:LWW-SL-Path-Inequality}
    (1-\cos t) \leq (2n+1)\sum_{j=1}^{n}(1-\cos t_{j}).
  \end{equation}
  Letting $t_{j} = k\cdot x_{j}$ where $x_{j}$ is the displacement
  along the $j^{\mathrm{th}}$ subwalk the argument used to
  prove~\eqref{eq:LWW-DB-N.1} with $u=1$ can be applied to
  give~\eqref{eq:LWW-DB-N.2}. The prefactor $(4N-1)(2N-1)$ arises as
  for an $N$ edge lace there are $2N-1$ subwalks, so $n=2N-1$
  in \Cref{eq:LWW-SL-Path-Inequality}. 
\end{proof}

\subsection{Completion of the Bootstrap}
\label{sec:LWW-Bootstrap}

This section begins by using the diagrammatic bounds of
\Cref{sec:LWW-Pi-DB-1.1,sec:LWW-Pi-DB-N.1} to establish that $\Pi$ is small
under the hypothesis $f(z)\leq K$.
\begin{lemma}
  \label{lem:LWW-SL5.11}
  Fix $z\in \ob{0,z_{c}}$ and assume $d$ is sufficiently large. If
  $f(z) \leq K$, then there is a constant $\bar c_{K}$ independent of
  $z$ and $d$ such that
  \begin{align}
    \label{eq:LWW-Pi-Small}
    \sum_{x\in \Z^{d}}\abs{\Pi_{z}(x)} &\leq \bar c_{K}\beta \\
    \label{eq:LWW-Pi-Cos-Small}
    \sum_{x\in \Z^{d}}(1-\cos(k\cdot x))\abs{\Pi_{z}(x)} & \leq \bar
    c_{K}\beta \hat C_{p(z)}(k)^{-1}.
  \end{align}
\end{lemma}
\begin{proof}
  This follows by combining the bounds of \Cref{prop:LWW-DB-1,prop:LWW-DB-N}
  for $u=0$ with the bound $\norm{(1-\cos k\cdot
    x)H_{\lambda,z}(x)}_{\infty} \leq c_{K}(1+\beta)\hat
  C_{p(z)}^{-1}(k)$ of \Cref{eq:LWW-SL5.10.1}.
\end{proof}

The remainder of this section is devoted to verifying the hypothesis
of \Cref{lem:LWW-Bootstrap} for $z_{1}=0$, $z_{2} = z_{c}(\lambda)$, $a=4$
and $b=1+O(\beta)$.

\begin{lemma}
  \label{lem:LWW-SL5.12}
  The function $f$ obeys $f(0)=1$.
\end{lemma}
\begin{proof}
  Clearly $f_{1}(0)=0$. The definition of $p(z)$ implies $p(0)=0$ as
  $\alpha_{0}(\lambda,0)=1$, so $f_{2}(0)=1$. Lastly, $f_{3}(0)=0$:
  $U_{0}=48$ while $\Delta_{k}\hat G_{\lambda,0} = 0$.
\end{proof}

\begin{lemma}
  \label{lem:LWW-SL5.14}
  The function $f$ is continuous on $\co{0,z_{c}}$.
\end{lemma}
\begin{proof}
  It suffices to show $f_{1},f_{2},f_{3}$ are continuous on $\cb{0,r}$
  for any $r<z_{c}$. For $f_{1}$ this follows as $\alpha(\lambda,z) \leq
  \alpha_{0}(\lambda,z)\leq \chi_{\lambda}(z)$, i.e.,
  $\alpha(\lambda,z)$ has a convergent power series representation. 

  Recall (see~\cite[Lemma~5.13]{Slade2006}) that the supremum of an
  equicontinuous family of functions over a compact interval is a
  continuous function, provided this supremum is finite. It follows
  that it is enough to prove a bound uniform in $k$ on the derivative
  of $f_{2}$ (resp.\ $f_{3}$) with respect to $z$.  Since
  equicontinuity of a family $\{\abs{g_{\alpha}}\}$ is equivalent to
  equicontinuity of $\{g_{\alpha}\}$, the absolute value on $\hat
  G_{\lambda,z}$ (resp.\ $\Delta_{k}\hat G_{\lambda,z}$) can be
  ignored. For $f_{2}$ the derivative is
   \begin{equation}
    \frac{d}{dz} \frac{\hat G_{\lambda,z}(k)}{\hat C_{p(z)}(k)} =
    \frac{1}{\hat C_{p(z)}(k)^{2}} \cb{ \hat C_{p(z)}(k) \frac{ d\hat
        G_{\lambda,z}(k)}{dz} - \hat G_{\lambda,z}(k) \frac{ d\hat
        C_{p(z)}(k)}{dp}|_{p=p(z)} \frac{dp(z)}{dz}}.
  \end{equation}
  Now note: $\abs{\hat G_{\lambda,z}(k)} \leq \chi_{\lambda}(r)$,
  $\abs{\frac{d}{dz} \hat G_{\lambda,z}(k)} \leq
  \abs{\frac{d}{dz}\chi_{\lambda}(r)}$, $\abs{ \partial_{p}\hat
   C_{p}(k)} \leq \abs{\Omega} \chi_{\lambda}(r)^{2}$. Further,
  \begin{align}
     \abs{\frac{d p(z)}{dz}} &= \abs{\frac{d}{dz}\abs{\Omega}^{-1}\ob{1 -
       \frac{\alpha_{0}(\lambda,z)}{\chi_{\lambda}(z)}}} \\
     &\leq \abs{\Omega}^{-1}\alpha_{0}(\lambda,r)\frac{d}{dz}\chi_{\lambda,}(r)
     \chi^{-2}_{\lambda}(0) + \chi_{\lambda}^{-1}(1)\frac{d}{dz}\alpha_{0}(\lambda,r),
  \end{align}
  and $\frac{d}{dz}\alpha_{0}(\lambda,r)$ is bounded above by
  $\frac{d}{dz}\chi_{\lambda}(r)$ by \Cref{lem:LWW-Derivative-UB}. A
  uniform bound on the derivative then follows from
  \begin{equation}
    \frac{1}{2} \leq \hat C_{p(z)}(k) \leq \hat C_{p(z)}(0) =
    \frac{\chi_{\lambda}(z)}{\alpha_{0}(\lambda,z)}
    \leq \chi_{\lambda}(r),
  \end{equation}
  where the second last equality follows from the definition of
  $p(z)$, and the last inequality from $\alpha_{0}(\lambda,z) \geq 1$.

  For $f_{3}$ the calculation is essentially the same. Calculating the
  derivative shows that what is needed is upper bounds on $\abs{\hat
    G_{\lambda,z}(k)}$, $\abs{\frac{d}{dz} \hat G_{\lambda,z}(k)}$,
  $\abs{ \partial_{p} \hat C_{p}(k)}$, and $\abs{\frac{d}{dz} p(z)}$,
  along with upper and lower bounds on $\hat C_{p(z)}$. These bounds
  have already been obtained.
\end{proof}

The next lemma completes the bootstrap argument.
\begin{lemma}
  \label{lem:LWW-SL5.16}
  Suppose $d$ is sufficiently large. Fix $z\in\ob{0,z_{c}}$, and
  suppose that $f(z)\leq 4$. Then there is a constant $c$ independent
  of $z$ and $d$ such that $f(z)\leq 1 + c\beta$.
\end{lemma}
\begin{proof}
  We prove $f_{j}(z)\leq 1 + c\beta$ for $j=1,2,3$ in sequence.
  
  Since $\alpha_{0}(\lambda,z)$ and $\chi_{\lambda}(z)$ are both
  positive and finite it follows that
  \begin{equation}
    \label{eq:LWW-Positive}
    \frac{\alpha_{0}(\lambda,z)}{\chi_{\lambda}(z)} = 1 -
    z\alpha(\lambda,z) \abs{\Omega} - \hat \Pi_{\lambda,z}(0) > 0.
  \end{equation}
  \Cref{eq:LWW-Positive} and \Cref{lem:LWW-SL5.11} together imply
  \begin{equation}
    f_{1}(z) = z\alpha(\lambda,z)\abs{\Omega} \leq 1 + \hat \Pi_{\lambda,z}(0)
    \leq 1 + \bar c_{4}\beta.
  \end{equation}

  \Cref{prop:LWW-UB-Alpha} implies $\alpha_{0} \leq 1 + \bar c\beta$, so
  $f_{2}\leq 1+O(\beta)$ follows if
  \begin{equation}
    \label{eq:LWW-C5.53}
    \frac{\hat G_{\lambda,z}(k)}{\alpha_{0}(\lambda,z) \hat
      C_{p(z)}(k)} = 1 + \frac{1 - p(z)\abs{\Omega}\hat D(k) - \hat
      F_{\lambda,z}(k)}{\hat F_{\lambda,z}(k)}
  \end{equation}
  is $1 + O(\beta)$, where
  \begin{equation}
    \hat F_{\lambda,z}(k) \equiv \hat G_{\lambda,z}(k)^{-1} = 1 -
    z\alpha(\lambda,z)\abs{\Omega}\hat D(k) - 
    \hat\Pi_{\lambda,z}(k).
  \end{equation}
  By definition, $p(z)\abs{\Omega} = z\alpha(\lambda,z)\abs{\Omega} + \hat
  \Pi_{\lambda,z}(0)$. Hence the numerator of the right-hand side
  of~\eqref{eq:LWW-C5.53} is
  \begin{equation}
    \label{eq:LWW-C5.54}
    1 - p(z)\abs{\Omega}\hat D(k) - \hat F_{\lambda,z}(k) = \hat
    \Pi_{\lambda,z}(0) \ob{ 1- \hat D(k)} - \ob{
      \hat \Pi_{\lambda,z}(0) - \hat \Pi_{\lambda,z}(k)},
  \end{equation}
  which is bounded above by $4\bar c_{4}\beta$. An alternative upper
  bound of the right hand side of~\eqref{eq:LWW-C5.54} follows
  from~\Cref{eq:LWW-Pi-Small,eq:LWW-Pi-Cos-Small}:
  \begin{equation}
    \label{eq:LWW-SL5.55}
    \hat
    \Pi_{\lambda,z}(0) \ob{ 1- \hat D(k)} - \ob{
      \hat \Pi_{\lambda,z}(0) - \hat \Pi_{\lambda,z}(k)} \leq \bar
    c_{4}\beta \ob{1-\hat D(k)} + \bar c_{4}\beta \ob{1 - p(z) 
      \abs{\Omega} \hat D(k)}.
  \end{equation}
   Since
  \begin{equation}
    \label{eq:LWW-SL5.56}
    \ob{1-\hat D(k)}\hat C_{p(z)}(k) = 1 + \underbrace{\hat
      D(k)}_{\leq 1} \underbrace{\frac{
        p(z)\abs{\Omega} - 1}{1-p(z)\abs{\Omega}\hat D(k)}}_{\leq 1} \leq 2,
  \end{equation}
  the numerator of~\eqref{eq:LWW-C5.53} is bounded by 
  \begin{equation}
    \label{eq:LWW-SL5.57}
    3\bar c_{4}\beta \ob{1-p(z)\abs{\Omega}\hat D(k)} \leq 3\bar
    c_{4}\beta \ob{ \hat F_{\lambda,z}(0) + \ob{1-\hat D(k)}}.
  \end{equation}

  The denominator of~\eqref{eq:LWW-C5.53} is
  \begin{align}
    \hat F_{\lambda,z}(k) &= \hat F_{\lambda,z}(0) + \ob{ \hat
      F_{\lambda,z}(k) - \hat F_{\lambda,z}(0)} \\
    \label{eq:LWW-SL5.58}
    &=\hat F_{\lambda,z}(0) + z\alpha(\lambda,z)\abs{\Omega}\ob{1-\hat
      D(k)} + \ob{ \hat \Pi_{\lambda,z}(0) - \hat \Pi_{\lambda,z}(k)}.
  \end{align}
  Let $\bar \lambda = \sup_{\eta\in\SAP}\lambda_{\eta}$, and
  $\lambda^{\star} = \max(1,\bar \lambda)$. For $z\leq
  (2\abs{\Omega}\sqrt{\lambda^{\star}})^{-1}$
  \Cref{prop:LWW-Trivial-G-Bound} (if $\lambda^{\star}>1$) or
  neglecting loops (if $\lambda^{\star}\leq 1$) implies $\hat
  F_{\lambda,z}(0) \geq \hat C_{z\sqrt{\lambda^{\star}}}(0)^{-1} \geq
  \frac{1}{2}$. Then $1-\hat D(k)\geq 0$ and~\eqref{eq:LWW-Pi-Cos-Small}
  imply 
  \begin{equation}
    \label{eq:LWW-SL5.59}
    \hat F_{\lambda,z}(k) \geq \hat F_{\lambda,z}(0) - 2\bar
    c_{4}\beta \geq \frac{1}{2} - 2\bar c_{4}\beta.
  \end{equation}

  For $(2\abs{\Omega}\lambda^{\star})^{-1} \leq z<z_{c}(\lambda)$
  \Cref{eq:LWW-Pi-Cos-Small}, $\hat F_{z}(0)>0$, and
  $\alpha(\lambda,z)\geq 1$ imply
  \begin{equation}
    1 - p(z)\abs{\Omega} \hat D(k) = 1- (1-\hat F_{\lambda,z}(0))\hat
    D(k) \leq 1 - \hat D(k) + \hat F_{\lambda,z}(0)
  \end{equation}
  and hence
  \begin{align}
    \hat F_{\lambda,z}(k) &\geq \hat F_{\lambda,z}(0) +
    \frac{1}{2\sqrt{\lambda^{\star}}}
    \ob{1 - \hat D(k)} - \bar c_{4}\beta \ob{1-p(z)\abs{\Omega}\hat
      D(k)} \\
    &\geq \ob{\frac{1}{2\sqrt{\lambda^{\star}}} - \bar c_{4}\beta}
    \ob{\hat F_{\lambda,z}(0) + \ob{1- \hat D(k)}}.
  \end{align}
  For $z\leq (2\abs{\Omega}\lambda^{\star})^{-1}$ or
  $(2\abs{\Omega}\lambda^{\star})^{-1} \leq z <z_{c}$ these lower and
  upper bounds combine to imply the right-hand side
  of~\eqref{eq:LWW-C5.53} is $1+O(\beta)$, and hence $f_{2}(z) = 1+
  O(\beta)$.

  Lastly consider $f_{3}(z)$. As for $f_{2}$, it suffices to prove the
  claim for $f_{3}/\alpha_{0}$. Let $\hat g_{\lambda,z}(k) =
  z\alpha(\lambda,z) \abs{\Omega} \hat D(k) + \hat
  \Pi_{\lambda,z}(k)$, so
  \begin{equation}
    \label{eq:LWW-SL5.62}
    \frac{\hat G_{\lambda,z}(k)}{\alpha_{0}(\lambda,z)} =
    \frac{1}{1-\hat g_{\lambda,z}(k)}.
  \end{equation}
  The symmetry of $D(x)$ and $\Pi_{\lambda,z}(x)$ implies
  that $g_{\lambda,z}(x) = g_{\lambda,z}(-x)$, so applying Lemma~5.7
  of~\cite{Slade2006} (a general fact
    about even functions) gives
  \begin{align}
    \label{eq:LWW-SL5.63}
    \frac{1}{2}\abs{\Delta_{k}\hat G_{\lambda,z}(\ell)} &\leq
    \frac{1}{2} \ob{ \hat G_{\lambda,z}(\ell-k) + \hat
      G_{\lambda,z}(\ell + k)} \hat G_{\lambda,z}(\ell) \ob{ \abs{\hat
      g_{\lambda,z}(0)} - \abs{\hat g_{\lambda,z}(k)}} \\
    &+ 4\hat G_{\lambda,z}(\ell-k) \hat G_{\lambda,z}(\ell) \hat
    G_{\lambda,z}(\ell + k) \ob{ \abs{\hat
      g_{\lambda,z}(0)} - \abs{\hat g_{\lambda,z}(k)}} \ob{ \abs{\hat
      g_{\lambda,z}(0)} - \abs{\hat g_{\lambda,z}(\ell)}}.
  \end{align}
  Using $f_{2}(z) \leq 1+ O(\beta)$ bounds each factor of $\hat
  G_{\lambda,z}$ by $\ob{1+O(\beta)}\hat C_{p(z)}$. Further,
  \begin{align}
    \abs{\hat g_{\lambda,z}(0)} - \abs{\hat g_{\lambda,z}(k)} &\leq
    \sum_{x}\ob{1-\cos(k\cdot x)}
    \ob{ z\alpha(\lambda,z) \abs{\Omega} + \abs{\Pi_{z}(x)}} \\
    &\leq z\alpha(\lambda,z)\abs{\Omega} \ob{1-\hat D(k)} + \bar c_{4}
    \beta \hat C_{p(z)}(k)^{-1} \\
    &\leq \ob{2+O(\beta)} \hat C_{p(z)}(k)^{-1},
  \end{align}
  where the second inequality is by~\eqref{eq:LWW-Pi-Cos-Small} and the
  third is by $f_{1}(z) \leq 1+ O(\beta)$ and~\eqref{eq:LWW-SL5.56}.
  Combining the bounds and using the definition of $U_{p(z)}$ gives
  $f_{3}(z) \leq 1 + O(\beta)$.
\end{proof}

\begin{corollary}
  \label{cor:LWW-IRB}
  For $d$ sufficiently large, $\lambda$-LWW satisfies a \emph{$k$-space
  infrared bound}: there is a constant $K = 1+O(\beta)$ such that for $0\leq z\leq
  z_{c}(\lambda)$ 
  \begin{equation}
    \hat G_{\lambda,z}(k) \leq K \hat C_{p(z)}(k).
  \end{equation}
\end{corollary}
\begin{proof}
  The proof of \Cref{lem:LWW-SL5.16} showed that $f_{2}(z)\leq
  1+O(\beta)$ without absolute values on $\hat G_{\lambda,z}$,
  uniformly for $z<z_{c}$. Taking a limit gives the claim.
\end{proof}

The fact that the quantities $T_{\lambda,z}$ and $S_{\lambda,z}$
defined below are small will be important in what follows.
\begin{definition}
  \label{def:LWW-Bubble-Etc}
  The
  \emph{triangle diagram} $T_{\lambda,z}$ and \emph{square diagram}
  $S_{\lambda,z}$ are the quantities
  \begin{equation}
    \label{eq:LWW-Bubble-Etc}
    T_{\lambda,z} = \norm{\hat H_{\lambda,z}^{3}}_{1}, \qquad
    S_{\lambda,z} = \norm{\hat H_{\lambda,z}^{4}}_{1}.
  \end{equation}
\end{definition}

\begin{corollary}
  \label{cor:LWW-TS}
  For $d$ sufficiently large and $z\leq z_{c}$ the triangle and square
  diagrams are bounded above by $c\beta$.
\end{corollary}
\begin{proof}
  For notational convenience write $\bar H_{\lambda,z} =
  \alpha_{0}^{-1} H_{\lambda,z}$, and similarly for $\bar
  G_{\lambda,z}$. By \Cref{eq:LWW-G-H-Relation} $\alpha_{0}^{-1}\hat
  H_{\lambda,z} = \alpha_{0}^{-1}\hat G_{\lambda,z}-1$. \Cref{cor:LWW-IRB}
  implies $\alpha_{0}^{-1}\hat G_{\lambda,z} \leq (1+O(\beta))\hat
  C_{p(z)}$ since $\alpha_{0}\leq 1 + O(\beta)$. The claim follows
  from \Cref{prop:LWW-Small}.
\end{proof}

\section{Proofs of the Main Results}
\label{sec:LWW-Further}

To go beyond the $k$-space infrared bound of \Cref{cor:LWW-IRB}
requires control of the derivatives of $G_{\lambda,z}$ and
$\Pi_{\lambda,z}$ with respect to $z$. This control is established in
\Cref{sec:LWW-Diagrammatic-Derivatives}. The remainder of the section
establishes \Cref{thm:LWW-Main} using arguments based
on~\cite[Chapter~6]{MadrasSlade2013}. Throughout let $z_{c} =
z_{c}(\lambda)$.

\subsection{Further Diagrammatic Bounds}
\label{sec:LWW-Diagrammatic-Derivatives}

Having verified that the bounds of \Cref{sec:LWW-DB-Pi} holds for
$z<z_{c}$, the monotone convergence theorem implies they continue to
hold at $z_{c}$.

\begin{proposition}
  \label{prop:LWW-BC-D-DB}
  For $d$ sufficiently large and $0<z\leq z_{c}$ 
  \begin{equation}
    \label{eq:LWW-BC-D-DB}
    \frac{d}{dz} \norm{\bubblechain_{\lambda,z}}_{1} \leq z_{c}^{-1}c\beta.
  \end{equation}
\end{proposition}
\begin{proof}
  The left-hand side is a polynomial with positive coefficients, so it
  suffices to obtain an upper bound at $z=z_{c}$. By 
  \Cref{prop:LWW-BC-DB}, $\alpha_{0}(\lambda,z_{c})\leq 1+ c\beta$,
  $\norm{\bubblechain_{\lambda,z_{c}}}_{1}\leq c\beta$, and
  \Cref{cor:LWW-TS}, the claim follows.
\end{proof}

\begin{proposition}
  \label{prop:LWW-H-D-DB}
  Let $d$ be sufficiently large, $0<z\leq z_{c}$, and $v = 1,2$. Then
  \begin{equation}
    \label{eq:LWW-H-D-DB}
    \norm{\partial^{v}_{z} \bar G_{\lambda,z}}_{\infty} =
    \norm{\partial^{v}_{z} \bar H_{\lambda,z}}_{\infty} \leq
    c\beta z_{c}^{-v}
  \end{equation}
\end{proposition}
\begin{proof}
  As for \Cref{prop:LWW-BC-D-DB} it suffices to consider
  $z=z_{c}$. The equality of the first two terms follows from
  \Cref{eq:LWW-G-H-Relation}. \Cref{prop:LWW-H-DB} implies
  \begin{equation}
    \frac{d}{dz} \bar H_{\lambda,z} \leq
    z^{-1}(1+\norm{\bubblechain_{\lambda,z}}_{1}) \bar H_{\lambda,z}
    \ast \bar G_{\lambda,z}.
  \end{equation}
  The claim follows for $v=1$ as $\norm{\bar H_{\lambda,z} \ast \bar
    G_{\lambda,z}}_{\infty} \leq c\beta$ by \Cref{eq:LWW-Conv-Bound} and
  $\norm{\bubblechain_{\lambda,z}}_{1}\leq c\beta$ by
  \Cref{prop:LWW-BC-Geometric}.

  For $v=2$ apply \Cref{lem:LWW-Derivative-UB}. After computing the
  derivative and using the triangle inequality (i) argue as for $v=1$
  for the term from differentiating $z^{-1}$ (ii) use
  \Cref{prop:LWW-BC-D-DB} when differentiating
  $\norm{\bubblechain_{\lambda,z}}_{1}$ and (iii) when differentiating
  either of the two-point functions use \Cref{prop:LWW-H-DB} and $
  \norm{\bar H_{\lambda,z} \ast \bar H_{\lambda,z} \ast\bar
    G_{\lambda,z}}_{\infty} \leq \norm{\bar H_{\lambda,z} \ast \bar
    H_{\lambda,z}}_{\infty} + \norm{\bar H_{\lambda,z} \ast \bar
    H_{\lambda,z} \ast \bar H_{\lambda,z}}_{\infty}$, and
  \Cref{cor:LWW-TS} to see that this is bounded by $c\beta$. Each term is
  therefore bounded by $c\beta z_{c}^{-2}$.
\end{proof}

\begin{proposition}
  \label{prop:LWW-I-D-DB}
  Let $d$ be sufficiently large, $0<z\leq z_{c}$, and $v=1,2$. Then
  \begin{equation}
    \label{eq:LWW-I-D-DB}
    \norm{\partial^{v}_{z}I_{\lambda,z}}_{1} \leq c\beta z_{c}^{-v}
  \end{equation}
\end{proposition}
\begin{proof}
  For $v=1$ note
  \begin{equation}
    \frac{d}{dz} I_{\lambda,z} = \frac{d}{dz} (1 -
    e^{-\mu_{\lambda,z}(0,x)}) \leq \frac{d}{dz} \mu_{\lambda,z}(0,x).
  \end{equation}
  This bound is increasing in $z$, so considering $z_{c}$ is
  enough. Translation invariance, as in the proof of \Cref{prop:LWW-UB-I},
  implies this is equal to the derivative in $z$ of
  $\norm{\bubblechain_{\lambda,z}}_{1}$. The claim
  follows for $v=1$ from \Cref{prop:LWW-BC-D-DB}.
  
  For $v=2$ it is enough to bound the derivative of the bound of
  \Cref{prop:LWW-BC-DB}. This is similar to the arguments already given;
  the only new terms that arise occur when differentiating $\norm{\bar
    H_{\lambda,z} \cdot \bar G_{\lambda,z} \ast \bar
    H_{\lambda,z}}_{1}$, which is $\bar H_{\lambda,z}\ast \bar
  G_{\lambda,z} \ast \bar H_{\lambda,z}(0)$. By \Cref{prop:LWW-H-DB} after
  taking a derivative the result is, up to a factor of $(1+O(\beta))$,
  a square diagram $\bar H_{\lambda,z}\ast \bar G_{\lambda,z} \ast
  \bar H_{\lambda,z} \ast \bar G_{\lambda,z}(0)$. Repeatedly using
  \Cref{eq:LWW-G-H-Relation} and \Cref{cor:LWW-TS} shows this is at most
  $c\beta$.
\end{proof}

\begin{proposition}
  \label{prop:LWW-Derivative-Small}
  For $d$ sufficiently large, $0<z<z_{c}$, and $v=1,2$
  \begin{equation}
    \label{eq:LWW-Derivative-Small}
    \norm{\partial^{v}_{z}\Pi_{\lambda,z}}_{1} \leq c\beta z_{c}^{-v}
  \end{equation}
\end{proposition}
\begin{proof}
  The Leibniz rule and \Cref{lem:LWW-Derivative-UB} imply that the result
  of differentiating $\Pi$ is a sum of terms of the form of the bounds
  of \Cref{prop:LWW-I2P-DB}, but where each term has one of the factors of
  $\bar G_{\lambda,z}$, $\bar H_{\lambda,z}$ or $I_{\lambda,z}$
  differentiated. Given this, the argument is as in the proofs of
  \Cref{prop:LWW-DB-1} and \Cref{prop:LWW-DB-N}. Let us describe the proof for
  $N=2$. For $N=1$ the proof is similar as
  $e^{\mu_{\lambda,z}(0,x)}\leq \alpha_{0}$.

  Consider $v=1$. There are $3N-1$ terms arising when differentiating
  $\pi^{(N)}_{\lambda,z}$. If $\bar G_{\lambda,z}$ or $\bar
  H_{\lambda,z}$ is differentiated apply~\Cref{prop:LWW-H-DB} and place
  the sup norm on this term when applying \Cref{lem:LWW-SL-Lp}, and
  then use \Cref{prop:LWW-H-D-DB} to bound this norm. If $I_{\lambda,z}$
  is differentiated use \Cref{lem:LWW-SL-Lp} placing the $\sup$ norm on
  a term $H_{\lambda,z}$ and use \Cref{prop:LWW-I-D-DB} to bound the one
  norm of the derivative of $I_{\lambda,z}$. This yields the claim as
  the factor of $3N-1$ is irrelevant for the convergence of the series.
  
  If $v=2$ there are $(3N-1)^{2}$ terms. If both derivatives fall on a
  single factor proceed as in the previous paragraph and use
  \Cref{prop:LWW-H-D-DB} or \Cref{prop:LWW-I-D-DB}. If the derivatives fall on
  distinct factors, one factor being $I_{\lambda,z}$, proceed as
  before. For the remaining case, where two distinct factors of $\bar
  H_{\lambda,z}$ (or $\bar G_{\lambda,z}$) are differentiated, place a
  sup norm on one term. The new term to bound when applying
  \Cref{lem:LWW-SL-Lp} is of the form $\norm{\bar H_{\lambda,z} \ast
    \bar G_{\lambda,z} \ast \bar G_{\lambda,z} \ast
    I_{\lambda,z}}_{\infty}$. It suffices to bound $\norm{\bar H \ast
    \bar G \ast \bar G}_{\infty}$, and this is bounded above by
  $\norm{\bar H\ast \bar G}_{\infty} + \norm{\bar H\ast \bar
    H}_{\infty} + \norm{\bar H \ast \bar H \ast \bar H}_{\infty}$, all
  of which are bounded by $c\beta$ by
  \Cref{cor:LWW-TS}.
\end{proof}

\begin{corollary}
  \label{cor:LWW-Pi-D-NZ}
  Let $d$ be sufficiently large and $0<z\leq z_{c}$. Then $-\frac{d}{dz}
  \hat F_{z}(0)\geq c>0$.
\end{corollary}
\begin{proof}
  The derivative is
  \begin{equation}
    \label{eq:LWW-Derivative-F}
    -\frac{d}{dz} \hat F_{\lambda,z}(0) = \abs{\Omega}\alpha(\lambda,z)
    + z\abs{\Omega}\frac{d}{dz}\alpha(\lambda,z) + \frac{d}{dz}\hat
    \Pi_{\lambda,z}(0).
  \end{equation}
  By \Cref{prop:LWW-Derivative-Small} $\abs{\frac{d}{dz}\hat
    \Pi_{\lambda,z}(k)}$ is bounded above by a constant since $z_{c}$
  is bounded below by a term of order $\beta$ by
  \Cref{prop:LWW-Trivial-G-Bound}. An argument as for \Cref{prop:LWW-Alpha-DB}
  shows the magnitude of the second term is bounded by a constant. As
  $\alpha(\lambda,z)\geq 1$ the first term dominates for
  $d$ sufficiently large.
\end{proof}

\subsubsection{Derivatives of Moments}
\label{sec:LWW-Derivative-Moments}

The next proposition (for $\lambda=0$) is~\cite[Exercise~5.17]{Slade2006}.
\begin{lemma}
  \label{lem:LWW-Pi-2M}
  For $d$ sufficiently large and $0\leq z<z_{c}$
  \begin{equation}
    \label{eq:LWW-Pi-2M}
    \norm{\abs{x}^{2}\Pi_{\lambda,z}(x)}_{1} \leq c\beta
  \end{equation}
\end{lemma}
\begin{proof}
  This follows from $\hat C_{p(z)}^{-1} \leq 1-\hat D(k)$
  and~\eqref{eq:LWW-Pi-Cos-Small}.
\end{proof}

\begin{proposition}
  \label{prop:LWW-Moment-Bounds}
  For $0\leq z\leq z_{c}$ the following bounds hold:
  \begin{align}
    \label{eq:LWW-Moment-Bounds-1}
    \norm{\abs{x}^{2}H_{\lambda,z}(x)}_{\infty} &\leq c\beta, \\
    \label{eq:LWW-Moment-Bounds-2}
    \norm{\abs{x}^{2}H_{\lambda,z}(x)}_{2} &\leq c.
  \end{align}
\end{proposition}
\begin{proof}
  The proof relies on the identity
  \begin{equation}
    \label{eq:LWW-MB-1}
    \abs{x_{\mu}}^{2}H_{\lambda,z}(x) = -
    \int_{\FS^{d}} \partial^{2}_{k_{\mu}} \hat H_{\lambda,z}(k)
    e^{-ik\cdot x}\, \frac{d^{d}k}{\FSint},
  \end{equation}
  where $\mu$ is a unit basis vector of $\Z^{d}$. Omitting the
  subscripts $\lambda$ and $z$ and letting a subscript $\mu$ denote
  partial differentiation with respect to $k_{\mu}$ the derivative can
  be calculated:
  \begin{equation}
    \label{eq:LWW-MB-2}
    \hat G_{\mu,\mu}(k) = z\alpha\abs{\Omega} \frac{\hat
      D_{\mu,\mu}(k)}{\hat F^{2}(k)} + 2(z\alpha\abs{\Omega})^{2}
    \frac{ \hat D_{\mu}^{2}(k)}{\hat F^{3}(k)} + \frac{ \hat
      \Pi_{\mu,\mu}(k)}{\hat F^{2}(k)} + 4z\alpha\abs{\Omega}
    \frac{\hat D_{\mu}(k)\hat \Pi_{\mu}(k)}{\hat F^{3}(k)} + 2
    \frac{\hat \Pi^{2}_{\mu}(k)}{\hat F^{3}(k)}.
  \end{equation}
  To obtain an estimate of $\norm{\abs{x}^{2}H_{\lambda,z}}_{\infty}$
  take the absolute value of~\eqref{eq:LWW-MB-1} inside of the integral
  and estimate the resulting one norms. Using $z\alpha\abs{\Omega}
  \leq 1+ O(\beta)$ an upper bound for the first
  two terms is
  \begin{equation}
    (1+O(\beta))\ob{ \norm{ \frac{\hat D_{\mu,\mu}(k)} {(1-\hat
          D(k))^{2}}}_{1} + 2\norm{ \frac{ \hat D_{\mu}^{2}(k)}
        {(1-\hat D(k))^{3}}}_{1}} \leq c\beta,
  \end{equation}
  where the second inequality follows by estimating the
  integrals, see~\cite[Appendix~A]{MadrasSlade2013}.

  For the remaining terms, $\norm{\abs{x}^{2}\Pi_{\lambda,z}}_{1}\leq
  c\beta$ implies $\norm{\hat \Pi_{\mu,\mu}}_{\infty}\leq
  c\beta$. Since $\hat \Pi_{\mu}(k)=0$ when $k_{\mu}=0$ Taylor's
  theorem and the above bound on $\norm{\hat \Pi_{\mu,\mu}}_{\infty}$
  imply $\norm{\hat \Pi_{\mu}}_{\infty}\leq c\beta
  \abs{k_{\mu}}$. Lastly, $\abs{\hat D_{\mu}(k)}_{\infty}\leq
  c\abs{k_{\mu}}$. These bounds combined with the $k$-space infrared
  bound \Cref{cor:LWW-IRB} imply each of the remaining
  terms are bounded by $c\beta$. This proves
  \Cref{eq:LWW-Moment-Bounds-1}. 

  For $\norm{\abs{x}^{2}H_{\lambda,z}}_{2}$ use Parseval's identity:
  $\norm{ \widehat{ \abs{x}^{2}H_{\lambda,z}} }_{2} =
  \norm{ \partial^{2}_{k}\hat H_{\lambda,z}}_{2}$. The previously
  described bounds for the numerators along with \Cref{cor:LWW-IRB} and
  \Cref{prop:LWW-Small} imply that $\hat G_{\mu,\mu}(k)$ is square
  integrable in sufficiently high dimensions. This implies
  \Cref{eq:LWW-Moment-Bounds-2}.
\end{proof}

\begin{proposition}
  \label{prop:LWW-Moment-Derivative-Small}
  For $d$ sufficiently large and $0<z\leq z_{c}$
  \begin{equation}
    \label{eq:LWW-Moment-Derivative-Small}
    \norm{\partial^{v}_{z} \abs{x}^{2} \Pi_{\lambda,z}}_{1} \leq
    c\beta z_{c}^{-v}
  \end{equation}
\end{proposition}
\begin{proof}
  Distribute the factor $\abs{x}^{2}$ along the factors of $\bar
  H_{\lambda,z}$ and $\bar G_{\lambda,z}$ as in the proof of
  \Cref{prop:LWW-DB-N}. The proof is now essentially the same as for
  \Cref{prop:LWW-Derivative-Small}. For each term place the sup norm on
  the factor with the term $\abs{x}^{2}$.  

  If a factor $\abs{x}^{2}G_{\lambda,z}$ has been
  differentiated once or twice the resulting term whose norm must be
  estimated has the form of either $\bar
  H_{\lambda,z}\ast \bar G_{\lambda,z}$ or $\bar H_{\lambda,z} \ast
  \bar G_{\lambda,z} \ast \bar G_{\lambda,z}$. In either
  case the factor $\abs{x}^{2}$ can again be split along the factors
  in the convolution. In the first case use \Cref{eq:LWW-G-H-Relation},
  the triangle inequality, and Young's inequality to obtain
  \begin{equation}
    \norm{(\abs{x}^{2}\bar H_{\lambda,z}) \ast \bar
      G_{\lambda,z}}_{\infty} \leq \norm{\abs{x}^{2}
      \bar H_{\lambda,z}}_{\infty} + \norm{\abs{x}^{2}\bar
      H_{\lambda,z}}_{2}\norm{\bar H_{\lambda,z}}_{2},
  \end{equation}
  and then use~\Cref{prop:LWW-Moment-Bounds} to see that this is
  bounded by $c\beta$. For the second case arguing similarly gives
  \begin{align}
    \nonumber
    \norm{(\abs{x}^{2}\bar H_{\lambda,z}) \ast \bar G_{\lambda,z} \ast
      \bar G_{\lambda,z}}_{\infty} \leq\, & 
    \norm{(\abs{x}^{2} \bar H_{\lambda,z}) \ast \bar
      G_{\lambda,z}}_{\infty} + \norm{(\abs{x}^{2} \bar H_{\lambda,z})
      \ast \bar H_{\lambda,z}}_{\infty} \\ &+ \norm{\abs{x}^{2} \bar
      H_{\lambda,z}}_{2}\norm{\bar H_{\lambda,z} \ast \bar 
      H_{\lambda,z}}_{2}.
  \end{align}
  The first case analysis implies the first two terms are bounded
  above by $c\beta$. Parseval's identity combined with \Cref{cor:LWW-TS}
  implies the last term is bounded by $c\beta$.  The rest of the
  analysis of these terms is in the proof of
  \Cref{prop:LWW-Derivative-Small}.

  The cases in which all derivatives fall on factors without the
  term $\abs{x}^{2}$ can be handled in the same manner as in the proof
  of \Cref{prop:LWW-Derivative-Small} by using Young's inequality, the
  triangle inequality, and \Cref{cor:LWW-TS}.
\end{proof}

\subsection{Linear Divergence of $\chi_{\lambda}(z)$ as $z\nearrow z_{c}$}
\label{sec:LWW-Susceptibility-MF}

Before proving the linear divergence of the susceptibility it will be
helpful to verify that it is only infinite at the critical point
$z=z_{c}$ itself.

\begin{lemma}
  \label{lem:LWW-Susceptibility-Infinite}
  For $d$ sufficiently large and $\abs{z}\leq z_{c}$ the inverse
  susceptibility $\hat F_{\lambda,z}(0)$ satisfies
  \begin{equation}
    \label{eq:LWW-Susceptibility-Infinite}
    \abs{\hat F_{\lambda,z}(0)} \geq \frac{\abs{\Omega}}{2} \abs{z_{c}-z}.
  \end{equation}
\end{lemma}
\begin{proof}
  As $\hat F_{\lambda,z_{c}}(0)=0$ the fundamental theorem of calculus
  implies
  \begin{equation}
    \abs{F_{\lambda,z}(0)} = \abs{\int_{z_{c}}^{z}-\frac{d}{dz}\hat F_{z}(0)\,dz}.
  \end{equation}
  Using $\hat F_{\lambda,z_{c}}(0)=0$, \Cref{eq:LWW-Derivative-F}, and
  integrating from $z_{c}$ to $z$ along the straight line $z_{t} = (1-t)z_{c} + tz$ implies
  \begin{equation}
    \abs{\hat F_{\lambda,z}(0)} = \abs{\Omega}\abs{z-z_{c}} \abs{
      \int_{0}^{1} \alpha(\lambda,z_{t}) + z_{t}\frac{d}{dz}\alpha(\lambda,z_{t})
      + \abs{\Omega}^{-1}\frac{d}{dz}\hat \Pi_{\lambda,z_{t}}(0)\,
      dt}.
  \end{equation}
  The last two terms are bounded by $c\beta$, see the proof of
  \Cref{cor:LWW-Pi-D-NZ}.  The claim follows by taking the dimension
  sufficiently large as $\int \alpha = 1 + O(\beta)$.
\end{proof}

Define constants $A=A(\lambda)$ and $D = D(\lambda)$ by
\begin{align}
  \label{eq:LWW-Constant-A}
  A(\lambda) &=  z_{c}^{-1}\ob{\alpha(\lambda,z_{c})\abs{\Omega} +
    z_{c}\abs{\Omega} \frac{d}{dz}\alpha(\lambda,z_{c}) + \frac{d}{dz}\hat
    \Pi_{\lambda,z_{c}}(0)}^{-1}, \\
  \label{eq:LWW-Constant-D}
  D(\lambda) &= A(\lambda) \ob{ - z_{c}\abs{\Omega}\alpha(\lambda,z_{c})
    \nabla^{2}_{k} \hat D(0) - \nabla^{2}_{k}\hat \Pi_{\lambda,z_{c}}(0)}.
\end{align}

\begin{theorem}
  \label{thm:LWW-Susceptibility-MF}
  For $d$ large enough, the susceptibility of $\lambda$-LWW diverges
  linearly as $z\nearrow z_{c}$:
  \begin{equation}
    \label{eq:LWW-Susceptibility-MF}
    \chi_{\lambda}(z) \sim \frac{Az_{c}}{z_{c}-z}.
  \end{equation}
  The constant $A$ in~\eqref{eq:LWW-Susceptibility-MF} is as in
  \Cref{eq:LWW-Constant-A}.
\end{theorem}
\begin{proof}
  Recall $\hat F_{\lambda,z}(0) = \hat G_{\lambda,z}(0)^{-1}$ is zero
  at $z_{c}$ since $\chi_{\lambda}(z)\nearrow \infty$ as $z\nearrow
  z_{c}$.
  \begin{align}
    \chi_{\lambda}(z) &= \frac{1}{\hat F_{\lambda,z}(0) - \hat
      F_{\lambda,z_{c}}(0)} \\
    &= \frac{1}{z_{c}-z} \cb{ \alpha(\lambda,z_{c})\abs{\Omega} +
      z\abs{\Omega} \frac{\alpha(\lambda,z_{c}) -
        \alpha(\lambda,z)}{z_{c}-z} + \frac{\hat
        \Pi_{\lambda,z_{c}}(0) - \hat
        \Pi_{\lambda,z}(0)}{z_{c}-z}}^{-1}.
  \end{align}
  The claim follows from \Cref{prop:LWW-Moment-Derivative-Small} and
  \Cref{prop:LWW-Alpha-DB} combined with $\alpha_{0}\leq 1 + c\beta$
  for $z\leq z_{c}$, which implies differentiability of $\alpha_{0}$ at $z_{c}$.
\end{proof}

\subsection{Growth Rate and Diffusive Scaling}
\label{sec:LWW-Growth}

To establish the growth rate of $\lambda$-LWW, as well as the
diffusive scaling, a Tauberian type theorem is needed. The statement
and proof of the next lemma in~\cite{MadrasSlade2013} involve
fractional derivatives of order $1+\epsilon$ for $0<\epsilon<1$, but the
arguments apply without modification for two ordinary derivatives.

\begin{lemma}[name = Lemma~6.3.4 of~\cite{MadrasSlade2013}]
  \label{lem:LWW-Tauberian}
  Let
  \begin{equation}
    \label{eq:LWW-Taub-1}
    f(z) = \frac{1}{\phi(z)} = \sum_{n=0}^{\infty}b_{n}z^{n},
  \end{equation}
  where $\phi(z) = \sum_{n=0}^{\infty}a_{n}z^{n}$. Suppose that
  \begin{equation}
    \label{eq:LWW-Taub-2}
    \sum_{n=0}^{\infty}n^{2}\abs{a_{n}}R^{n}<\infty,
  \end{equation}
  so in particular, $\phi(z)$, $\phi^{\prime}(z)$, and
  $\phi^{\prime\prime}(z)$ are finite when $\abs{z}=R$. Assume in
  addition that $\phi^{\prime}(R)\neq 0$. Suppose that $\phi(R)=0$ and
  $\phi(z)\neq 0$ for $\abs{z}\leq R$, $z\neq R$. Then
  \begin{equation}
    \label{eq:LWW-Taub-3}
    f(z) = \frac{1}{-\phi^{\prime}(R)}\frac{1}{R-z} + O(1)
  \end{equation}
  uniformly in $\abs{z}\leq R$, and
  \begin{equation}
    \label{eq:LWW-Taub-4}
    b_{n} = R^{-n-1}\cb{\frac{1}{-\phi^{\prime}(R)} + O(n^{-\alpha})}
    \quad \textrm{as $n\to\infty$},
  \end{equation}
  for every $\alpha<1$.
\end{lemma}

Recall that $c_{n}^{\lambda}$ is the total mass of $n$-step
$\lambda$-LWW, i.e.,
\begin{equation*}
  c_{n}^{\lambda} = \sum_{x} \mathop{\sum_{\omega\colon 0
    \to x}}_{\abs{\omega}=n} \lambda^{n_{L}(\omega)}.
\end{equation*}
\begin{theorem}
  \label{thm:LWW-Growth-Rate}
  For $d$ sufficiently large and any $\delta<1$
  \begin{equation*}
    c_{n}^{\lambda} = A(\lambda)z_{c}(\lambda)^{-n}(1+O(n^{-\delta})).
  \end{equation*}
\end{theorem}
\begin{proof}
  Apply \Cref{lem:LWW-Tauberian} to $\hat F_{\lambda,z}(0)$. The
  verification of the hypotheses of the theorem are the conclusions of
  \Cref{prop:LWW-Derivative-Small}, \Cref{cor:LWW-Pi-D-NZ}, and
  \Cref{lem:LWW-Susceptibility-Infinite}.
\end{proof}

The proof of the next theorem is essentially the proof for
self-avoiding walk in~\cite{MadrasSlade2013} verbatim; it is
reproduced here for the sake of completeness. The next lemma, which
will be used several times, is stated here for the convenience of the
reader.
\begin{lemma}[name = Lemma~6.3.2 of~\cite{MadrasSlade2013}]
  \label{lem:LWW-SL6.3.2}
  Let $f(z) = \sum_{n=0}^{\infty} a_{n}z^{n}$. Let $R>0$, and suppose
  $f^{\prime}(R) = \sum_{n=0}^{\infty}n \abs{a_{n}}R^{n-1}<\infty$, so
  in particular $f(z)$ converges for $\abs{z}\leq R$. Then for $\abs{z}\leq R$
  \begin{equation}
    \abs{f(z)- f(R)} \leq f^{\prime}(R) \abs{R-z}.
  \end{equation}
  If $f^{\prime\prime}(z)(R)<\infty$, then for $\abs{z}\leq R$
  \begin{equation}
    \abs{f(z) - f(R) - f^{\prime}(R)(z-R)} \leq \frac{1}{2}
    f^{\prime\prime}(R) \abs{R-z}^{2}.
  \end{equation}
\end{lemma}

\begin{theorem}
  \label{thm:LWW-Diffusive}
  For $d$ sufficiently large $\lambda$-LWW is diffusive:
  \begin{equation}
    \label{eq:LWW-Diffusive}
    \ab{ \abs{\omega(n)}^{2}}^{\lambda}_{n} = Dn(1+O(n^{-\delta}))
  \end{equation}
  as $n\to\infty$ for any $\delta<1$. The constant $D$ is that
  of~\eqref{eq:LWW-Constant-D}.
\end{theorem}
\begin{proof}
  Let $\nabla^{2}_{k}$ denote the $k$-space Laplacian. Then
  \begin{equation}
    \label{eq:LWW-Diffusive-1}
    \ab{ \abs{\omega(n)}^{2}}_{\lambda,n} = -\frac{\nabla^{2}_{k}\hat
      c_{n}^{\lambda}(0)}{ c_{n}^{\lambda}}. 
  \end{equation}
  Since $\hat c_{n}^{\lambda}(k)$ is the coefficient of $z^{n}$ in $\hat
  G_{\lambda,z}(k)$ Cauchy's formula implies
  \begin{equation}
    \label{eq:LWW-Diffusive-2}
    -\nabla^{2}_{k}\hat c_{n}^{\lambda}(0) = \frac{1}{2\pi i} \oint \frac{
      \nabla^{2}_{k} \hat F_{\lambda,z}(0)}{ \hat
      F_{\lambda,z}(0)^{2}} \frac{dz}{z^{n+1}},
  \end{equation}
  where the integral is around a small origin centred circle. Define
  $E(z)$ by
  \begin{equation}
    \label{eq:LWW-Diffusive-3}
    \frac{\nabla^{2}_{k} \hat F_{\lambda,z}(0)}{ \hat
      F_{\lambda,z}(0)^{2}} = \frac{ \nabla^{2}_{k}\hat F_{z_{c}}(0)}
    { \cb{ \frac{d}{dz}\hat F_{z_{c}}(0)}^{2}(z_{c}-z)^{2}} + E(z).
  \end{equation}
  Making this substitution into \Cref{eq:LWW-Diffusive-2} and calculating
  the first integral implies
  \begin{equation}
    \label{eq:LWW-Diffusive-4}
    -\nabla^{2}_{k}\hat c_{n}^{\lambda}(0) = \frac{ \nabla^{2}_{k}\hat F_{z_{c}}(0)}
    { \cb{ \frac{d}{dz}\hat F_{z_{c}}(0)}^{2}}(n+1)z_{c}^{-n-2} +
    \frac{1}{2\pi i} \oint E(z) \frac{dz}{z^{n+1}}.
  \end{equation}

  Assuming the integral of $E(z)$ is $O(n^{\delta}z_{c}^{-n})$ for every
  $\delta>0$ implies the theorem by inserting the behaviour of
  $c_{n}^{\lambda}$ given by \Cref{thm:LWW-Growth-Rate}.

  To verify the assumption it suffices by
  \Cref{lem:LWW-SL6.3.2}
  to prove $\abs{E(z)} \leq \mathrm{const.}  \abs{z_{c}-z}^{-1}$ for
  all $\abs{z}\leq z_{c}$. Split $E(z)$ as $E(z) = T_{1}(z) +
  T_{2}(z)$ with
  \begin{align}
    \label{eq:LWW-Diffusive-6}
    T_{1}(z) &= \cb{\frac{d}{dz}\hat F_{\lambda,z_{c}}(0)}^{-2} \frac{
      \nabla^{2}_{k} \hat F_{\lambda,z}(0) - \nabla^{2}_{k}\hat
      F_{\lambda,z_{c}}(0)}{(z_{c}-z)^{2}} \\
    \label{eq:LWW-Diffusive-7}
    T_{2}(z) &= \frac{ - \nabla^{2}_{k} \hat F_{\lambda,z}(0) \cb{
        \hat F_{\lambda,z}(0)^{2} - \cb{\frac{d}{dz}\hat
          F_{\lambda,z_{c}}(0)}^{2}(z_{c}-z)^{2}}
    }{\cb{\frac{d}{dz}\hat F_{\lambda,z_{c}}(0)}^{2} \hat
      F_{\lambda,z}(0)^{2}(z_{c}-z)^{2}}.
  \end{align}

  The numerator of $T_{1}(z)$ is differentiable in $z$ by
  \Cref{prop:LWW-Moment-Derivative-Small}, so (i) of
  \Cref{lem:LWW-SL6.3.2}
  implies the numerator is bounded above by a constant times
  $\abs{z_{c}-z}$. It follows that $\abs{T_{1}}\leq
  O(\abs{z_{c}-z}^{-1})$.

  For $T_{2}$ note that $\hat F_{\lambda,z}(0)^{2} \geq
  \textrm{const.}\abs{z_{c}-z}^{2}$ by
  \Cref{lem:LWW-Susceptibility-Infinite} so
  \begin{align}
    \label{eq:LWW-Diffusive-8}
    \abs{T_{2}(z)} \leq \mathrm{const.}\abs{z_{c}-z}^{-4} \cb{ \hat
      F_{\lambda,z}(0) + \frac{d}{dz}\hat
      F_{\lambda,z_{c}}(0)(z_{c}-z)} \cb{ \hat F_{\lambda,z}(0)
      - \frac{d}{dz} \hat F_{\lambda,z_{c}}(0)(z_{c}-z)},
  \end{align}
  as $\nabla^{2}_{k}\hat F_{\lambda,z}(0)$ is bounded by a constant by
  \Cref{prop:LWW-Moment-Bounds}. By (ii) of \Cref{lem:LWW-SL6.3.2},
  \Cref{prop:LWW-Moment-Derivative-Small}, and $\hat
  F_{\lambda,z_{c}}(0) = 0$, the middle term is
  $O(\abs{z_{c}-z}^{2})$. Using (i) of \Cref{lem:LWW-SL6.3.2} for
  $\hat F_{\lambda,z}(0)$ in the last term shows the last term is
  $O(\abs{z_{c}-z})$. Thus $\abs{T_{2}(z)}\leq O(\abs{z_{c}-z}^{-1})$,
  which proves the claim.
\end{proof}

\appendix

\section{Loop Measure Representation of $\lambda$-LWW}
\label{sec:LWW-Viennot}

The purpose of this appendix is to provide a proof of
\Cref{thm:LWW-LM-Rep}. 

A fundamental property of $\lambda$-LWW is that it admits a loop
measure representation. The representation follows from a theorem of
Viennot~\cite{Viennot1986} and is proved via the theory of heaps of
pieces in \Cref{app:LWW-LM-Rep}.

\begin{remark}
  \label{rem:LWW-Use-of-Heaps}
  The methods of~\cite[Chapter~9]{LawlerLimic2010} are sufficient to
  derive formulas that would suffice for the lace expansion analysis
  of $\lambda$-LWW. These methods have the benefit of brevity, but
  they do not reveal the connection with the loop $O(N)$ model. For
  this reason we have chosen to present a more scenic route here.
\end{remark}

The rest of this section will take place in the context of an
arbitrary graph $G$, as specializing to $\Z^{d}$ does not provide any
simplification. The theory of heaps of pieces will be freely used;
see~\cite{Viennot1986} or~\cite{Krattenthaler2006} for an introduction.

\subsection{Viennot's Theorem}
\label{app:LWW-Viennot}

\begin{definition}
  \label{def:LWW-Cycle}
  A \emph{trivial cycle} is a single edge of $G$. An \emph{oriented
    cycle} is either (i) an oriented cyclic subgraph of $G$ or (ii) a
  trivial cycle in $G$.
\end{definition}

An oriented cycle corresponds to an equivalence class of self-avoiding
polygons, where a self-avoiding polygon $\omega = (\omega_{0}, \dots,
\omega_{k}=\omega_{0})$ is equivalent to any cyclic permutation
$\tilde\omega = (\omega_{r}, \omega_{r+1}, \dots, \omega_{k},
\omega_{1}, \dots, \omega_{r})$. For example, a trivial cycle
$\{x,y\}$ corresponds to the self-avoiding polygons $(x,y,x)$ and
$(y,x,y)$, while an oriented 3-cycle corresponds to walks of the form
$(x,y,z,x)$ and cyclic permutations thereof for $x,y,z$ distinct.

\begin{definition}
  \label{def:LWW-Heap-Cycles}
  A \emph{heap of (oriented) cycles} is a heap of pieces whose labels
  are oriented cycles. Two oriented cycles
  $C_{1}$, $C_{2}$ are concurrent if $V(C_{1})\cap V(C_{2}) \neq
  \emptyset$, i.e., if the cycles share a vertex.
\end{definition}

\begin{definition}
  \label{def:LWW-IH-Legal-Pair}
  A pair $(\eta, H)$ where $\eta$ is a self-avoiding walk from $a$
  to $b$ and $H$ is a heap of cycles whose maximal elements' labels
  each contain a vertex in $\eta$ is called a \emph{legal $(a,b)$
    pair}. Let $\WH(a,b)$ denote the set of legal $(a,b)$ pairs, and
  $\WH$ denote the set of all legal pairs.
\end{definition}

\Cref{thm:LWW-LM-Rep}, the loop measure representation of
$\lambda$-LWW, is a byproduct of the proof of the following theorem
of Viennot.
\begin{theorem}[name={\cite[Proposition~6.3]{Viennot1986}}]
  \label{thm:LWW-Bijection}
  There is a bijection $\phi_{ab}$ from the set $\WH(a,b)$ of legal
  $(a,b)$ pairs to the set of walks $\Walks(a,b)$ from $a$ to
  $b$. Further,
  \begin{enumerate}
  \item The multi-set of edges in a legal $(a,b)$ pair $(\eta,H)$ is
    the same as the multi-set of edges in the walk
    $\phi_{ab}((\eta,H))$.
  \item The multi-set of oriented cycles $\{ \ell(x) \mid x\in H\}$
    for a heap $(H,\ell,\preceq)$ is the same as the multi-set of oriented
    cycles that are erased by applying loop erasure to $\phi_{ab}(
    (\eta,H))$.
  \end{enumerate}
\end{theorem}

\Cref{thm:LWW-Bijection} is not proven in~\cite{Viennot1986}. For the sake
of completeness and the convenience of the reader a proof is given in
\Cref{app:LWW-Viennot-Proof}. The remainder of this section consists of a
heuristic description of the proof; see also
\Cref{fig:LWW-IH-Bijection} which depicts the proof strategy.

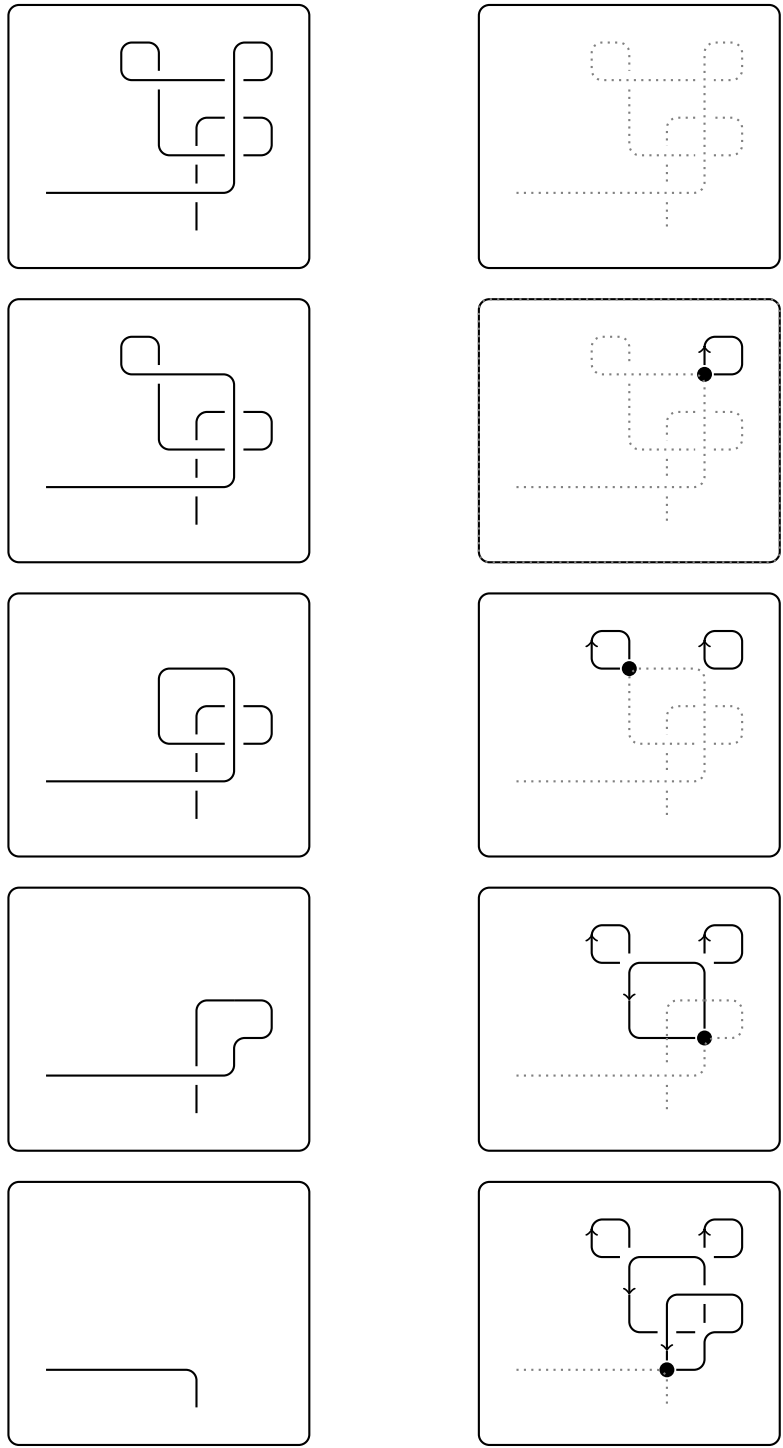
\begin{figure}[]
  \centering
  \begin{subfigure}{1.0\textwidth}
    \centering
  \beginpgfgraphicnamed{fig7a}
    \begin{tikzpicture}[scale=.5]
      \foreach \x in {0,...,6} \foreach \y in {0,...,5} \node (\x\y)
      at (\x,\y) {};

     \draw[black,thick, rounded corners] (-1,-1) rectangle (7,6);
     \draw[black, thick, rounded corners] (0,1) to (5,1) to (5,5) to (6,5) to
     (6,4) to (54);
     \draw[black, thick, rounded corners] (54) to (2,4) to (2,5) to (3,5) to (34);
     \draw[black, thick, rounded corners] (34) to (3,2) to (52);
     \draw[black, thick, rounded corners] (52) to (6,2) to (6,3) to
     (53);
     \draw[black, thick, rounded corners] (53) to (4,3) to (42);
     \draw[black, thick, rounded corners] (42) to (41);
     \draw[black,thick,rounded corners] (41) to (4,0);
   \end{tikzpicture}
   \endpgfgraphicnamed
   \qquad\qquad\qquad
  \beginpgfgraphicnamed{fig7b}
   \begin{tikzpicture}[scale=.5]
      \foreach \x in {0,...,6} \foreach \y in {0,...,5} \node (\x\y)
      at (\x,\y) {};
     \draw[black,thick,rounded corners] (-1,-1) rectangle (7,6);

     \draw[gray, thick, dotted, rounded corners] (0,1) to (5,1) to (5,5) to (6,5) to
     (6,4) to (54);
     \draw[gray, thick, dotted, rounded corners] (54) to (2,4) to
     (2,5) to (3,5) to (34); 
     \draw[gray, thick, dotted, rounded corners] (34) to (3,2) to (52);
     \draw[gray, thick, dotted, rounded corners] (52) to (6,2) to (6,3) to
     (53);
     \draw[gray, thick, dotted, rounded corners] (53) to (4,3) to (42);
     \draw[gray, thick, dotted, rounded corners] (42) to (41);
     \draw[gray, thick, dotted, rounded corners] (41) to (4,0);
   \end{tikzpicture}
   \endpgfgraphicnamed
 \end{subfigure}%

  \vspace{10pt}

  \begin{subfigure}{1.0\textwidth}
    \centering
  \beginpgfgraphicnamed{fig7c}
    \begin{tikzpicture}[scale=.5]
      \foreach \x in {0,...,6} \foreach \y in {0,...,5} \node (\x\y)
      at (\x,\y) {};

     \draw[black,thick, rounded corners] (-1,-1) rectangle (7,6);
     \draw[black,thick, rounded corners] (0,1) to (5,1) to (5,4)
     to (2,4) to (2,5) to (3,5) to (34);
     \draw[black, thick, rounded corners] (34) to (3,2) to (52);
     \draw[black, thick, rounded corners] (52) to (6,2) to (6,3) to
     (53);
     \draw[black, thick, rounded corners] (53) to (4,3) to (42);
     \draw[black, thick, rounded corners] (42) to (41);
     \draw[black,thick,rounded corners] (41) to (4,0);     
   \end{tikzpicture}
   \endpgfgraphicnamed
   \qquad\qquad\qquad
     \beginpgfgraphicnamed{fig7d}
   \begin{tikzpicture}[scale=.5]
     \foreach \x in {0,...,6} \foreach \y in {0,...,5} \node (\x\y) at
     (\x,\y) {}; 
     
     \draw[black,thick,rounded corners] (-1,-1) rectangle
     (7,6);
     \draw[black, thick, rounded corners,->]  (5,4.75) to (5,5) to (6,5) to
     (6,4) to (54) to (5,4.75);
     \node[black,inner sep=2pt,fill,circle] at (54) {};

     \draw[gray,thick,dotted, rounded corners] (-1,-1) rectangle (7,6);
     \draw[gray,thick,dotted, rounded corners] (0,1) to (5,1) to (5,4)
     to (2,4) to (2,5) to (3,5) to (34);
     \draw[gray,thick,dotted, rounded corners] (34) to (3,2) to (52);
     \draw[gray,thick,dotted, rounded corners] (52) to (6,2) to (6,3) to
     (53);
     \draw[gray,thick,dotted, rounded corners] (53) to (4,3) to (42);
     \draw[gray,thick,dotted, rounded corners] (42) to (41);
     \draw[gray,thick,dotted, rounded corners] (41) to (4,0); 
   \end{tikzpicture}
   \endpgfgraphicnamed
  \end{subfigure}%

  \vspace{10pt}

  \begin{subfigure}{1.0\textwidth}
    \centering
      \beginpgfgraphicnamed{fig7e}
    \begin{tikzpicture}[scale=.5]
      \foreach \x in {0,...,6} \foreach \y in {0,...,5} \node (\x\y)
      at (\x,\y) {};

     \draw[black,thick, rounded corners] (-1,-1) rectangle (7,6);
     \draw[black,thick, rounded corners] (0,1) to (5,1) to (5,4) to
     (3,4) to (3,2) to (52);
     \draw[black, thick, rounded corners] (52) to (6,2) to (6,3) to
     (53);
     \draw[black, thick, rounded corners] (53) to (4,3) to (42);
     \draw[black, thick, rounded corners] (42) to (41);
     \draw[black,thick,rounded corners] (41) to (4,0);
   \end{tikzpicture}
   \endpgfgraphicnamed
   \qquad\qquad\qquad
     \beginpgfgraphicnamed{fig7f}
   \begin{tikzpicture}[scale=.5]
     \foreach \x in {0,...,6} \foreach \y in {0,...,5} \node (\x\y) at
     (\x,\y) {}; 
     
     \draw[black,thick,rounded corners] (-1,-1) rectangle
     (7,6);
     \draw[black, thick, rounded corners,->]  (5,4.75) to (5,5) to (6,5) to
     (6,4) to (5,4) to (5,4.75);
     \draw[black,thick,rounded corners,->] (2,4.75) to (2,5) to
     (3,5) to (34) to (2,4) to (2,4.75);
     \node[black,inner sep=2pt,fill,circle] at (34) {};

     \draw[gray,thick,dotted,rounded corners] (0,1) to (5,1) to (5,4) to
     (3,4) to (3,2) to (52);
     \draw[gray,thick,dotted,rounded corners] (52) to (6,2) to (6,3) to
     (53);
     \draw[gray,thick,dotted,rounded corners] (53) to (4,3) to (42);
     \draw[gray,thick,dotted,rounded corners] (42) to (41);
     \draw[gray,thick,dotted,rounded corners] (41) to (4,0);
   \end{tikzpicture}
   \endpgfgraphicnamed
  \end{subfigure}%

  \vspace{10pt}

  \begin{subfigure}{1.0\textwidth}
    \centering
      \beginpgfgraphicnamed{fig7g}
    \begin{tikzpicture}[scale=.5]
      \foreach \x in {0,...,6} \foreach \y in {0,...,5} \node (\x\y)
      at (\x,\y) {};

     \draw[black,thick, rounded corners] (-1,-1) rectangle (7,6);
     \draw[black,thick, rounded corners] (0,1) to (5,1) to (5,2) to (6,2) to (6,3) to
     (5,3);
     \draw[black, thick, rounded corners] (5,3) to (4,3) to (4,2);
     \draw[black, thick, rounded corners] (4,2) to (41);
     \draw[black,thick,rounded corners] (41) to (4,0);
   \end{tikzpicture}
   \endpgfgraphicnamed
   \qquad\qquad\qquad
  \beginpgfgraphicnamed{fig7h}
   \begin{tikzpicture}[scale=.5]
     \foreach \x in {0,...,6} \foreach \y in {0,...,5} \node (\x\y) at
     (\x,\y) {}; 
     
     \draw[black,thick,rounded corners] (-1,-1) rectangle
     (7,6);
     \draw[black, thick, rounded corners,->]  (5,4.75) to (5,5) to (6,5) to
     (6,4) to (54) to (5,4.75);
     \draw[black,thick,rounded corners,->] (2,4.75) to (2,5) to
     (3,5) to (34) to (2,4) to (2,4.75);
     \draw[black, thick, rounded corners,->] (3,3) to
     (3,2) to (52) to (5,4) to (3,4) to (3,3);
     \node[black,inner sep=2pt,fill,circle] at (52) {};

    \draw[gray,thick,dotted, rounded corners] (0,1) to (5,1) to (5,2)
     to (6,2) to (6,3) to (5,3); 
     \draw[gray,thick,dotted, rounded corners] (5,3) to (4,3) to (4,2);
     \draw[gray,thick,dotted, rounded corners] (4,2) to (41);
     \draw[gray,thick,dotted, rounded corners] (41) to (4,0);
   \end{tikzpicture}
   \endpgfgraphicnamed
  \end{subfigure}%

  \vspace{10pt}

  \begin{subfigure}{1.0\textwidth}
    \centering
  \beginpgfgraphicnamed{fig7i}
    \begin{tikzpicture}[scale=.5]
      \foreach \x in {0,...,6} \foreach \y in {0,...,5} \node (\x\y)
      at (\x,\y) {};

     \draw[black,thick, rounded corners] (-1,-1) rectangle (7,6);
     \draw[black,thick, rounded corners] (0,1) to (4,1) to (4,0);
   \end{tikzpicture}
   \endpgfgraphicnamed
   \qquad\qquad\qquad
  \beginpgfgraphicnamed{fig7j}
   \begin{tikzpicture}[scale=.5]
     \foreach \x in {0,...,6} \foreach \y in {0,...,5} \node (\x\y) at
     (\x,\y) {}; 
     
     \draw[black,thick,rounded corners] (-1,-1) rectangle
     (7,6);
     \draw[black, thick, rounded corners,->]  (5,4.75) to (5,5) to (6,5) to
     (6,4) to (54) to (5,4.75);
     \draw[black,thick,rounded corners,->] (2,4.75) to (2,5) to
     (3,5) to (34) to (2,4) to (2,4.75);
     \draw[black, thick, rounded corners,->] (3,3) to
     (3,2) to (42) to (52) to (53) to (5,4) to (3,4) to (3,3);
     \draw[black,thick,rounded corners,->] (4,1.5) to (41) to (5,1) to (5,2) to
     (6,2) to (6,3) to (4,3) to (4,1.5);
     \node[black,inner sep=2pt,fill,circle] at (41) {};
     
     \draw[gray,thick,dotted,rounded corners] (0,1) to (4,1) to (4,0);
   \end{tikzpicture}
   \endpgfgraphicnamed
  \end{subfigure}%
  \caption{The figure illustrates the bijection between valid pairs
    $(\eta,H)$ and walks $\omega$ whose loop erasure is $\eta$. The
    left-hand side shows the results of successive applications of
    $\LE^{1}$, culminating in a self-avoiding walk. The right hand
    side shows the heaps of oriented cycles generated, with the walk
    displayed in dotted gray. Each heap has been given a distinguished
    vertex. The vertex indicates the oriented cycle that is maximal in
    the walk order as well as the location at which this oriented
    cycle is glued in to the corresponding walk when performing loop
    addition.}
  \label{fig:LWW-IH-Bijection}
\end{figure}
Let $\omega$ be a walk from $a$ to $b$. Trace $\omega$ until the first
time a vertex is visited twice. This identifies a first closed subwalk
$C_{1} = (\omega_{\tau^{\star}_{\omega}}, \dots,
\omega_{\tau_{\omega}})$. Remove $C_{1}$ by performing a single loop
erasure, and form a heap of pieces consisting of a single piece
labelled $C_{1}$. The first time a vertex is visited twice by the walk
$\LE^{1}(\omega)$ identifies a second closed subwalk, call this
$C_{2}$. Remove $C_{2}$ and form a new heap of pieces by adding a
second piece labelled $C_{2}$ to the heap consisting of
$C_{1}$. Continuing in this manner removes all of the closed subwalks
from $\omega$, resulting in a self-avoiding walk $\eta$ from $a$ to
$b$. Each maximal piece in the heap is labelled by a cycle that shares
a vertex with $\eta$. In other words, this procedure converts each walk from
$a$ to $b$ into a legal pair $(\eta,H)$.

Conversely, consider a legal pair $(\eta,H)$. To invert the procedure
what is required is a way to reduce the heap to the empty heap one
piece at a time, while inserting the labels of the removed pieces into
the (initially) self-avoiding walk $\eta$. This is relatively
straightforward: the maximal pieces of the heap $H$ have labels that
share a vertex with $\eta$, and hence the maximal pieces can be
ordered by using the linear order on vertices in $\eta$. Take the
maximal piece in this order, remove it from the heap to get a heap
$H^{\prime}$, and glue the corresponding label into $\eta$ to get a
walk $\eta^{\prime}$. The maximal elements of $H^{\prime}$ have labels
that share a vertex with $\eta^{\prime}$, and hence this procedure can
be iterated.

These operations are in fact inverses of one another. The next section
makes the preceding discussion precise.

\subsection{Proof of Viennot's Theorem}
\label{app:LWW-Viennot-Proof}

The theorem requires two algorithms, one which inserts oriented cycles
into a given walk, and one which removes oriented cycles from a
walk. Removing oriented cycles is achieved by loop erasure. The other
algorithm is introduced now.

\begin{definition}
  \label{def:LWW-Loop-Insertion}
  Let $\omega$ be a walk of length $n$, and let $C$ be an oriented
  cycle of length $k$. Assume that $C$ and $\omega$ have a vertex in
  common, and let $i$ be the minimal index such that $\omega_{i}$ is a
  vertex in $C$. Let $(c_{0}, \dots, c_{k})$ be the
  unique representative of $C$ such that $c_{0}=\omega_{i}$. The \emph{loop
    insertion $\omega \loopadd C$ of $C$ into $\omega$} is the walk
  $(\omega_{0}, \dots, \omega_{i-1}, c_{0}, \dots, c_{k}, \omega_{i+1},
  \dots, \omega_{n})$.
\end{definition}

In words, to insert a loop $C$ into a walk $\omega$ we find the first
vertex $\omega_{i}$ in $\omega$ that is contained in $C$. $C$ is then
rooted at $\omega_{i}$, $\omega$ is traversed until just before
reaching $\omega_{i}$, $C$ is traversed, and then the remainder of
$\omega$ is traversed.

\begin{lemma}
  \label{lem:LWW-Remove-Insert}
  Let $\omega$ be a walk, and let $C$ be the oriented
  cycle removed to create $\LE^{1}(\omega)$. Then
  $\LE(\omega)\loopadd C = \omega$.
\end{lemma}
\begin{proof}
  The definition of $\tau^{\star}_{\omega}$ and the definition of loop
  erasure implies that the first vertex in common between
  $\LE^{1}(\omega)$ and $C$ is $\omega_{\tau^{\star}_{\omega}}$, and hence
  the closed self-avoiding walk representing $C$ that is inserted by
  loop insertion is $(\omega_{\tau^{\star}_{\omega}},
  \dots, \omega_{\tau_{\omega}})$.
\end{proof}

Given a collection of oriented cycles that intersect a walk it is
necessary to determine the order in which the cycles should be
inserted. The next definition gives the correct order for inverting loop erasure.

\begin{definition}
  \label{def:LWW-Walk-Order}
  Let $\omega$ be a walk, and $C_{1},\dots, C_{k}$ a collection of
  oriented cycles that each share a vertex with $\omega$. Let $t_{j}=\min
  \{i \mid \omega_{i}\in C_{j}\}$.  The \emph{walk order} on the
  oriented cycles is given by setting $C_{m}\geq C_{n}$ if $t_{m}\geq t_{n}$.
\end{definition}

The following algorithm, called the \emph{loop addition algorithm},
constructs a walk beginning at the vertex $a$ and ending at the vertex
$b$ from a legal $(a,b)$ pair $(\eta,H)$.

\begin{enumerate}
\item Set $\omega^{0}=\eta$.
\item Suppose $H^{i-1}\neq\emptyset$. Set $\omega^{i} =
    \omega^{i-1}\loopadd C$, where $C$ is maximal in the walk order
    among the labels of the maximal pieces of $H^{i-1}$. Let $y$ be
    the piece whose label is $C$, and set $H^{i}=H^{i-1}\setminus \{y\}$.
\item If $H^{i-1}=\emptyset$, output $\omega=\omega^{i-1}$. Otherwise
  go to 2.
\end{enumerate}

The algorithm is well-defined as the labels of the maximal pieces in a
heap must be vertex disjoint, so the walk order is a strict total
order on the maximal pieces of the heap. Note that at each step of the
algorithm the walk $\omega^{i}$ begins at the vertex $a$ and ends at
$b$, so $\omega$ is a walk from $a$ to $b$ as claimed.

\begin{lemma}
  \label{lem:LWW-Insert-Remove}
  Suppose $(\eta,H)\in\WH$. Suppose the output of the loop addition
  algorithm is $\omega$. If $C$ is the last oriented cycle inserted,
  then the oriented cycle removed by loop erasure applied to $\omega$
  is $C$.
\end{lemma}
\begin{proof}
  The proof is by induction on the size of $H$. Suppose $C$ was the
  $(k+1)^{\mathrm{st}}$ oriented cycle added.
  \begin{enumerate}
  \item If $C$ was the label of a maximal element in $H^{k-1}$ then
    $C$ is vertex disjoint from the $k^{\mathrm{th}}$ added oriented
    cycle $C^{\prime}$.  The definition of the walk order implies that
    the first vertex $C$ shares with $\omega^{k-1}$ occurs prior to
    the first vertex in $C^{\prime}$ because $C$ is disjoint from
    $C^{\prime}$. It follows that $C$ is the oriented cycle erased by
    loop erasure, as $C$ closes prior to $C^{\prime}$, which was
    previously (by induction) the first oriented cycle to close.
  \item If $C$ was not the label of a maximal piece in $H^{k-1}$ then
    $C$ is the label of a piece that was below the $k^{\mathrm{th}}$
    inserted piece. Suppose the $k^{\mathrm{th}}$ piece had label
    $C^{\prime}$. As $C$ intersects $C^{\prime}$, $C$ is inserted into
    the subwalk $C^{\prime}$ of $\omega^{k-1}$. By induction
    $C^{\prime}$ was the first oriented cycle to close in
    $\omega^{k-1}$, so $C$ is the first oriented cycle to close in
    $\omega^{k}$.\qedhere
  \end{enumerate}
\end{proof}

To construct a legal pair $(\eta,H)$ from a walk is fairly
straightforward. By applying loop erasure oriented cycles are removed,
and they naturally form a heap by using the heap composition
operation. More precisely, we have the \emph{(total) loop erasure
  algorithm}:
\begin{enumerate}
\item Set $\omega^{0}=\eta$, and $H^{0}=\emptyset$, where
  $\emptyset$ is the empty heap of oriented cycles.
\item If $\omega^{i-1}$ is not a self-avoiding walk, set
  $\omega^{i}=\LE^{1}(\omega^{i-1})$, and if $C$ is the closed
  self-avoiding walk removed from
  $\omega^{i-1}$, let $H^{i}=H\circ \{\bar C\}$ where $\bar C$ is the
  oriented cycle corresponding to $C$.
\item If $\omega^{i-1}$ is a self-avoiding walk, output
  $(\omega^{i-1},H^{i-1})$. Otherwise go to 2.
\end{enumerate}

Single loop erasure removes a subwalk of length at least 2 from any
non-simple walk at each step, so iteratively applying $\LE^{1}$
stabilizes on a self-avoiding walk in a finite number of
iterations. It follows that the total loop erasure is well defined.

\begin{lemma}
  \label{lem:LWW-Erasure-Heap}
  The output of the loop erasure algorithm applied to a walk $\omega =
  (\omega_{0}, \dots, \omega_{n})$ is a pair
  $(\eta,H)\in\WH(\omega_{0},\omega_{n})$.
\end{lemma}
\begin{proof}
  At each step of the algorithm the maximal pieces of the heap $H^{i}$
  share a vertex with the remaining walk $\omega^{i}$, and the
  algorithm only terminates once the remaining walk is
  self-avoiding. Removing a cycle cannot change the initial vertex
  of a walk, so $\eta_{0}=\omega_{0}$. If the final vertex of $\omega$
  is removed it must be that visiting the final vertex completes a
  cycle, and hence $\eta$ ends at $\omega_{n}$.
\end{proof}

\begin{proof}[Proof of~\Cref{thm:LWW-Bijection}]
  We claim that loop erasure and loop addition are inverses of one
  another, and prove the claim by induction. Suppose the claim holds
  between walks whose loop erasure removes $k$ oriented cycles and
  pairs $(\eta,H)\in\WH(a,b)$ whose heap $H$ has $k$ pieces.

  On the one hand, inserting the final oriented cycle $C$ in the loop
  addition algorithm yields a walk, and $C$ is the first oriented
  cycle removed by loop erasure by~\Cref{lem:LWW-Insert-Remove}. By
  induction it follows that loop erasure applied to the loop addition
  of a pair $(\eta,H)\in\WH(a,b)$ returns $(\eta,H)$.

  On the other hand when a single oriented cycle $C$ is removed from
  $\omega$ the cycle $C$ is minimal in the walk order and the removed
  oriented cycle is the label of a maximal piece. So the
  reconstruction of the heap formed by loop erasure proceeds as if the
  piece with label $C$ was not present, and hence (by induction)
  recreates $\LE^{1}(\omega)$ correctly. \Cref{lem:LWW-Remove-Insert}
  then implies that $\omega$ is the output of applying loop erasure
  and then loop addition.
\end{proof}

\subsection{Proof of~\Cref{thm:LWW-LM-Rep}}
\label{app:LWW-LM-Rep}

The proof of \Cref{thm:LWW-LM-Rep} follows from two calculations. The
first is a straightforward consequence of the fact that the bijection
between walks and legal pairs is given by loop erasure. Let $\cc T$
denote the set of trivial heaps of oriented cycles, and $\cc H$ the
set of heaps of oriented cycles. Let $\OC(\eta)$ denote the set of
oriented cycles that \emph{do not} share a vertex with the set $\eta$,
and let $\cc H_{\eta}$ denote the set of heaps $H$ such that
$(\eta,H)$ is a legal pair. The definition of $\lambda$-LWW,
\Cref{thm:LWW-Bijection}, and the heap
theorem~\cite[Proposition~5.3]{Viennot1986} imply 
\begin{align}
  \label{eq:LWW-LM-Triv-1}
  \bar w_{\lambda,z}(\eta) &= \sum_{\omega\colon \LE(\omega)=\eta}
  w_{\lambda,z}(\omega) \\
    \label{eq:LWW-LM-Triv-2}
  &= z^{\abs{\eta}} \sum_{H\in\cc H_{\eta}} w_{\lambda,z}(H) \\
  \label{eq:LWW-LM-Triv-3}
  &= z^{\abs{\eta}} \frac{ \sum_{T\in\cc T(\OC(\eta))}
    (-1)^{\abs{T}}w_{\lambda,z}(T)} { \sum_{T\in \cc
      T}(-1)^{\abs{T}}w_{\lambda,z}(T)},
\end{align}
where $w_{\lambda,z}(H) = \prod_{x\in H}w_{\lambda,z}(\ell(x))$ for a
heap $(H,\ell,\preceq)$. In particular note that this definition
assigns a weight $z^{2}\lambda$ to a trivial cycle.

The second calculation is an expression for sums over trivial heaps of
oriented cycles. \Cref{thm:LWW-LM-Rep} follows by applying
\Cref{prop:LWW-SO} to the numerator and denominator
of~\eqref{eq:LWW-LM-Triv-3} and cancelling common factors. This
calculation is a calculation involving formal power series; to see
that it holds as a relation between power series, note that for $z$
sufficiently small the final expressions are bounded by random walk
quantities, which converge.
\begin{proposition}
  \label{prop:LWW-SO}
  \begin{equation}
    \label{eq:LWW-SO}
    \sum_{T\in\cc T(\OC(A))} (-1)^{\abs{T}}w_{\lambda,z}(T) = \exp \ob{ -
      \sum_{x\in \Z^{d}} \mathop{\sum_{\omega\colon x\to
          x}}_{\abs{\omega}\geq 1}
      \indicatorthat{\range{\omega}\cap A = \emptyset}
      \frac{w_{\lambda,z}(\omega)}{\abs{\omega}}}
  \end{equation}
\end{proposition}
\begin{proof}[Proof of \Cref{thm:LWW-LM-Rep}]
  Let $\bar z = sz$. Then
  $w_{\lambda,z}(\omega)=w_{\lambda,\bar z}(\omega)$ when
  $s=1$. Using this observe that
  \begin{equation}
    \sum_{T\in\cc T(\OC(A))} (-1)^{\abs{T}}w_{\lambda,z}(T) = \exp
    \int_{0}^{1}\frac{d}{ds} \log \sum_{T\in\cc T(\OC(A))}
    (-1)^{\abs{T}}w_{\lambda,\bar z}(T).
  \end{equation}
  In calculating the derivative the Leibniz rule for differentiating
  $s^{k}$ can be interpreted as selecting one of the $k$ vertices
  contained in the cycles of a trivial heap. The selected vertex
  distinguishes a self-avoiding polygon. \Cref{thm:LWW-Bijection} can
  be applied to transform this into a walk weighted by
  $w_{\lambda,z}$. The factor of $-1$ in the exponent arises from the
  application of \Cref{thm:LWW-Bijection}, as the distinguished cycle
  carried a factor of $-1$. Lastly, the term $\abs{\omega}^{-1}$
  arises from the integration of $s^{\abs{\omega}-1}$ from $0$ to
  $1$; the missing factor of $s$ is due to the differentiation which
  distinguished a vertex.
\end{proof}

\subsection{Relation to Correlations of the $O(N)$ Cycle Gas}
\label{sec:LWW-Cycle-Correlation}

Note that \Cref{eq:LWW-LM-Triv-1} and \Cref{eq:LWW-LM-Triv-3} imply
that $G_{\lambda,z}(0,x)$ is given by a ratio of partition
functions. The denominator is a sum over oriented mutually disjoint
cyclic subgraphs, where the weight of a subgraph $H$ is
$z^{\abs{E(H)}}(-\lambda)^{\# H}$, where $\# H$ denotes the number of
cyclic subgraphs contained in $H$. The numerator is a sum over
self-avoiding walks from $0$ to $x$ along with disjoint cyclic
subgraphs; the weight is the same as for the denominator except for
the fact that the walk does not receive a factor of $\lambda$.

For each cycle of length at least $3$ summing over the possible
orientations of the cycles results in a model of unoriented cycle,
where each unoriented cycle has weight $-2\lambda$, except for trivial
cycles, which have weight $-\lambda$. \emph{Defining} the two-point
correlation in the $O(N)$ cycle gas to be the ratio described in the
previous paragraph gives the relation between the $O(N)$
cycle gas and $\lambda$-LWW. Note that if cycles of length two are
assigned loop activity $0$ this yields a precise correspondence
between $\lambda$-LWW and the $O(-2\lambda)$ cycle gas.

\section*{Acknowledgements}
\label{sec:acknowledgements}

I would like to thank my PhD advisor, David Brydges, for many
interesting and inspiring discussions related to this work, which
formed a portion of my PhD thesis at the University of British
Columbia. This paper was revised while I was at ICERM for the semester
program \emph{Phase Transitions and Emergent Properties}, and I would
like to thank ICERM for their hospitality and support. Finally, I
would like to thank Gordon Slade, Mark Holmes, and the referees
for their helpful comments, critiques and references, which have
greatly improved this article.

\bibliographystyle{amsplain}
\bibliography{refs1}

\end{document}